\documentclass[12pt, a4paper]{amsart}

\usepackage{amsmath, amssymb, amsfonts, amsthm}
\usepackage{mathtools}
\usepackage[left=30mm, right=30mm, top=20mm, bottom=20mm]{geometry}
\usepackage{enumitem}
\usepackage{tikz-cd}
\usepackage{hyperref}

\usepackage{tikz}
\tikzset{
	node/.style={draw, fill, circle, inner sep=0pt, minimum size=2mm}
}

\newtheorem{thmintro}{Theorem}

\newtheorem{corintro}[thmintro]{Corollary}
\newtheorem*{thmintro*}{Theorem}

\theoremstyle{definition}
\newtheorem{rkintro}{Remark}

\theoremstyle{plain}
\newtheorem{theorem}{Theorem}[section]
\newtheorem{proposition}[theorem]{Proposition}
\newtheorem{lemma}[theorem]{Lemma}
\newtheorem{corollary}[theorem]{Corollary}

\theoremstyle{definition}
\newtheorem{definition}[theorem]{Definition}
\newtheorem{convention}[theorem]{Convention}
\newtheorem{example}[theorem]{Example}
\newtheorem{remark}[theorem]{Remark}
\newtheorem{notation}[theorem]{Notation}

\renewcommand{\phi}{\varphi}
\renewcommand{\epsilon}{\varepsilon}
\newcommand{\NN}{\mathbb{N}}
\newcommand{\FF}{\mathbb{F}}
\newcommand{\mc}[1]{\mathcal{#1}}
\DeclareMathOperator{\Aut}{Aut}
\DeclareMathOperator{\proj}{proj}

\begin{document}

\title[Uncountably many $2$-spherical groups of Kac--Moody type]{Uncountably many $\mathbf{2}$-spherical groups of Kac--Moody type of rank $\mathbf{3}$ over $\FF_2$}

\author[S.~Bischof]{Sebastian Bischof}

\email{sebastian.bischof@math.uni-paderborn.de}
\address{Mathematisches Institut, Arndtstra\ss e 2, 35392 Gie\ss en, Germany}



\begin{abstract}
	In this paper we show that Weyl-invariant commutator blueprints of type $(4, 4, 4)$ are faithful. As a consequence we answer a question of Tits from the late $1980$s about twin buildings. Moreover, we obtain the first example of a $2$-spherical Kac--Moody group over a finite field which is not finitely presented.
\end{abstract}

\maketitle

\section{Introduction}

\subsection*{Motivation and goals}

In \cite{Ti92} J.~Tits stated a local-to-global conjecture for Kac--Moody buildings of $2$-spherical type. This conjecture was proved for Kac--Moody buildings over fields of cardinality at least $4$ in \cite{MR95}. In \cite{AM97} it was observed that Tits' local-to-global conjecture is closely related to Curtis--Tits-presentations of $2$-spherical Kac--Moody groups. In fact, it is proved in that paper, that the Curtis--Tits presentation for BN-pairs of spherical type (see \cite[Theorem~$13.32$]{Ti74}) generalizes to $2$-spherical Kac--Moody groups over fields of cardinality  at least $4$. It follows from this result, that a $2$-spherical Kac--Moody group over a finite field of cardinality at least $4$ is finitely presented. Up until now it was an open question whether the local-to-global principle and the Curtis--Tits presentation hold without the restriction on the ground field. Our following result answers those questions.

\begin{thmintro*}
	Let $G$ be a Kac--Moody group of compact hyperbolic type $(4, 4, 4)$ over the field $\FF_2$. Then the following hold:
	\begin{itemize}
		\item The group $G$ is not finitely presented (cf.\ Theorem~\ref{Main theorem: KM group finitely presented}).
		
		\item The local-to-global principle does not hold for the Kac--Moody building associated with $G$ (cf.\ Theorem~\ref{Main result: Extension theorem}).
	\end{itemize}
\end{thmintro*}

In particular, we obtain the existence of a $2$-spherical Kac--Moody group over a finite field which is not finitely presented. In order to prove the theorem above, one has to construct exotic Kac--Moody buildings of type $(4, 4, 4)$ over $\FF_2$. The strategy developed for constructing such buildings yields the following result.

\begin{thmintro*}
	There exist uncountably many, pairwise non-isomorphic groups of Kac--Moody type over $\FF_2$ whose Weyl group is the compact hyperbolic group $(4, 4, 4)$ (cf.\ Corollary~\ref{Main result: Existence}).
\end{thmintro*}

\subsection*{Main result}

In \cite{Ti92} J.~Tits introduced RGD-systems in order to describe groups of Kac--Moody type (e.g.\ Kac--Moody groups over fields). Each RGD-system has a type which is given by a Coxeter system, and to any Coxeter system one can associate a set $\Phi$ of roots (viewed as half-spaces). An \emph{RGD-system of type $(W, S)$} is a pair $(G, (U_{\alpha})_{\alpha \in \Phi})$ consisting of a group $G$ together with a family of subgroups $(U_{\alpha})_{\alpha \in \Phi}$ (called \emph{root groups}) indexed by the set of roots $\Phi$ satisfying some axioms. The most important axiom is the existence of commutation relations between root groups corresponding to prenilpotent pairs of roots. In this context there appears naturally a family $(U_w)_{w\in W}$ of subgroups of $G$ indexed by the Coxeter group $W$.

In \cite{BiRGD} we introduced the notion of commutator blueprints (we refer to Section~\ref{Section: commutator blueprint of type ((W, S), D)} for the precise definition). These purely combinatorial objects can be seen as blueprints for constructing RGD-systems \emph{over $\FF_2$} (i.e.\ each root group has exactly $2$ elements) with prescribed commutation relations. By definition, each commutator blueprint gives rise to a family of abstract groups $(U_w)_{w\in W}$. To each RGD-system over $\FF_2$ one can associate a commutator blueprint. The blueprints arising in this way are called \emph{integrable}. One can show that integrable commutator blueprints satisfy the following two properties (cf.\ Definition~\ref{Definition: Properties of Commutator blueprints}): They are \emph{Weyl-invariant} (roughly speaking: the commutation relations are Weyl-invariant) and -- due to a result of J.~Tits \cite{Ti86} -- \emph{faithful} (for each $w\in W$ the canonical morphism $U_w \to U_+ := \lim U_w$ is injective, where $U_+$ is the direct limit of the family $(U_w)_{w\in W}$). In general it is a difficult problem to decide whether a given commutator blueprint is faithful. In this article we prove the following main result:

\begin{thmintro}\label{Main result: Weyl-invariant implies faithful}
	(See Corollary~\ref{Corollary: non-trivial}) Weyl-invariant commutator blueprints of type $(4, 4, 4)$ are faithful.
\end{thmintro}

Combining Theorem~\ref{Main result: Weyl-invariant implies faithful} with \cite[Theorem~A]{BiRGD}, we obtain the following equivalence which allows us to construct new RGD-systems of type $(4, 4, 4)$ over $\FF_2$:

\begin{thmintro}\label{Main result: Theorem integrable Weyl-invariant}
	For any commutator blueprint $\mathcal{M}$ of type $(4, 4, 4)$ the following are equivalent:
	\begin{enumerate}[label=(\roman*)]
		\item $\mathcal{M}$ is integrable.
		
		\item $\mathcal{M}$ is Weyl-invariant.
	\end{enumerate}
\end{thmintro}

\begin{rkintro}
	The proof of Theorem~\ref{Main result: Weyl-invariant implies faithful} utilizes the fact that $(W, S)$ is of type $(4, 4, 4)$ in both an algebraic and a geometric way. Let us briefly explain these facts.
	\begin{itemize}[label=$\bullet$]
		\item \textbf{Algebra:} On the algebraic side, we use a crucial fact about the family of groups $(U_w)_{w\in W}$ constructed from Weyl-invariant commutator blueprints of type $(4, 4, 4)$: if $\ell$ denotes the canonical length function of $(W, S)$, then for each $w\in W$ and all distinct $s, t \in S$ with $\ell(ws) = \ell(wt) = \ell(w) +1$ the subgroup
		\[ V_{wstst} := \langle U_{ws}, U_{wt} \rangle \leq U_{wstst} \]
		is a proper subgroup (for more details we refer to the beginning of Section~\ref{Section: Locally Weyl-invariant Commutator blueprints}).
		
		\item \textbf{Geometry:} On the geometric side, we heavily use the geometry of the set of roots $\Phi(W, S)$ (in particular, see the results in Section~\ref{Section: Preliminaries}).
	\end{itemize}
	Commutator blueprints of some other types share important algebraic and geometric properties with commutator blueprints of type $(4, 4, 4)$. Let us mention here the case $(6, 6, 6)$, as well as the case $\lvert S \rvert \geq 4$ and $m_{st} = 4$ for all distinct $s, t \in S$. It would be interesting to find more types for which Theorem~\ref{Main result: Weyl-invariant implies faithful} holds. 
\end{rkintro}

\subsection*{Consequences}

In the rest of the introduction we discuss several consequences of Theorem~\ref{Main result: Theorem integrable Weyl-invariant}, which reduces the question of existence of RGD-systems of type $(4, 4, 4)$ over $\FF_2$ with prescribed commutation relations to the existence of the corresponding Weyl-invariant commutator blueprints. Such blueprints were already constructed in \cite[Theorem~D]{BiConstruction}. Together with Theorem~\ref{Main result: Theorem integrable Weyl-invariant} we obtain the following result:

\begin{corintro}\label{Main result: uncountably many RGD-systems}
	There exist uncountably many RGD-systems of type $(4, 4, 4)$ over $\FF_2$.
\end{corintro}

We say that a group $G$ is \emph{of $(4, 4, 4)$-Kac--Moody type over $\FF_2$} if there exists a family of subgroups $(U_{\alpha})_{\alpha \in \Phi}$ such that $(G, (U_{\alpha})_{\alpha \in \Phi})$ is an RGD-system of type $(4, 4, 4)$ over $\FF_2$. In \cite[Theorem~A]{BiIsomorphism} we have studied the isomorphism problem for groups of $(4, 4, 4)$-Kac--Moody type over $\FF_2$. Thus Theorem~\ref{Main result: Theorem integrable Weyl-invariant} together with \cite[Theorem~A]{BiIsomorphism} and \cite[Theorem~D]{BiConstruction} yields the following:

\begin{corintro}\label{Main result: Existence}
	There exist uncountaly many isomorphism classes of groups of $(4, 4, 4)$-Kac--Moody type over $\FF_2$.
\end{corintro}

\begin{rkintro}
	\begin{enumerate}[label=(\alph*)]
		\item Let $\mathcal{D} = (G, (U_{\alpha})_{\alpha \in \Phi})$ be an RGD-system of type $(4, 4, 4)$ over $\FF_2$ such that $G = \langle U_{\alpha} \mid \alpha \in \Phi \rangle$. By \cite[Theorem A]{Bischof_On_Growth_Functions_of_Coxeter_Groups} $\mathcal{D}$ is a \emph{twin building lattice} (cf.\ \cite{CR09}). Such (irreducible) lattices are studied in \cite{CR09}.
		
		\item The existence of non-isomorphic Kac--Moody groups with isomorphic buildings is already known (cf.\ \cite{Re02}). As the buildings associated to RGD-systems of type $(4, 4, 4)$ over $\FF_2$ are isomorphic (cf.\ \cite{BCM21}), Corollary~\ref{Main result: Existence} provides uncountably many isomorphism classes of groups of $(4, 4, 4)$-Kac--Moody type over $\FF_2$ with isomorphic buildings.
	\end{enumerate}
\end{rkintro}

Next we will discuss finiteness properties. Abramenko and M\"uhlherr have shown in \cite{AM97} that $2$-spherical Kac--Moody groups over finite fields of cardinality at least $4$ are finitely presented. We obtain the first $2$-spherical Kac--Moody group (in the sense of \cite{Ti92}) over a finite field which is not finitely presented:

\begin{thmintro}\label{Main theorem: KM group finitely presented}
	(See Theorem \ref{Theorem: Kac--Moody group of type (4,4,4) over F_2 not finitely presented}) Kac--Moody groups of type $(4, 4, 4)$ over $\FF_2$ are not finitely presented.
\end{thmintro}

\begin{rkintro}
	\begin{enumerate}[label=(\alph*)]
		\item Theorem~\ref{Main theorem: KM group finitely presented} only makes a statement about Kac--Moody groups and not about general groups of Kac--Moody type. We expect that the methods proving Theorem~\ref{Main theorem: KM group finitely presented} provide at least infinitely many groups of $(4, 4, 4)$-Kac--Moody type over $\FF_2$ which are not finitely presented. The question whether any group of $(4, 4, 4)$-Kac--Moody type over $\FF_2$ is finitely presented is much harder.
		
		\item It is conjectured that groups of Kac--Moody type which are not of $2$-spherical type are not finitely presented. Special cases of this conjecture have been verified by Abramenko and Gates (see \cite[Theorem~5.2 and Theorem~5.3]{Abramenko_Gates:Finite_presentability_of_Kac-Moody_groups_over_finite_fields} which is based on Gates' PhD thesis \cite{GaDiss}). In particular, their result covers the rank $3$ case completely (see \cite[Theorem~5.4]{Abramenko_Gates:Finite_presentability_of_Kac-Moody_groups_over_finite_fields}).
	\end{enumerate}
\end{rkintro}

In \cite{Ab04} P.~Abramenko considered finiteness properties of parabolic subgroups of Kac--Moody groups. He announced that the stabilizer of a chamber in certain Kac--Moody groups of compact hyperbolic type of rank $3$ over $\FF_2$ is not finitely generated (cf.\ \cite[Counter-Example $1(2)$]{Ab04}). A consequence of the proof of our Theorem~\ref{Main result: Weyl-invariant implies faithful} confirms Abramenko's claim:

\begin{thmintro}\label{Main result: U+ not finitely generated}
	(See Theorem \ref{Uplusnotfinitelygenerated}) Let $\mathcal{D} = (G, (U_{\alpha})_{\alpha \in \Phi})$ be an RGD-system of type $(4, 4, 4)$ over $\FF_2$. Then the stabilizer $U_+ = \mathrm{Stab}_G(c)$ of a chamber $c$ is not finitely generated.
\end{thmintro}

\begin{rkintro}
	\begin{enumerate}[label=(\alph*)]
		\item By \cite[Section~$1.2$]{As23} the automorphism group of the Kac--Moody building of type $(4, 4, 4)$ over $\FF_2$ does not have Property (T). This result can be deduced from our Theorem~\ref{Main result: U+ not finitely generated} as follows: By \cite[Theorem $6.8$]{CR09b} and \cite[Theorem A]{Bischof_On_Growth_Functions_of_Coxeter_Groups} the group $U_+ = \mathrm{Stab}_G(c)$ is a lattice in $\Aut(\Delta_-)$. It is well-known that lattices of groups with Property (T) are finitely generated. But $U_+$ is not finitely generated by Theorem~\ref{Main result: U+ not finitely generated}.
		
		\item It is known that there exist RGD-systems which are not of $2$-spherical type having non-finitely generated chamber stabilizers (see for instance \cite[Counter-Example 1(1)]{Ab04} and \cite[Proposition~5.2 and Proposition~5.4]{Abramenko_Gates:Finite_presentability_of_Kac-Moody_groups_over_finite_fields}).
	\end{enumerate}
\end{rkintro}

We now focus on Property (FPRS) of RGD-systems introduced by Caprace and Rémy in \cite[Section~$2.1$]{CR09}. This property makes a statement about the set of fixed points of the action of the root groups on the associated building. It  implies that every root group is contained in a suitable contraction group. Property (FPRS) is used in \cite{CR09} to show that under some mild conditions the \emph{geometric completion} of an RGD-system (cf.\ \cite{RR06}) is topologically simple. Caprace and Rémy have shown in \cite{CR09} that almost all RGD-systems of $2$-spherical type as well as all Kac--Moody groups satisfy this property. According to \cite{CR09} it has been known that there exist RGD-systems that do not satisfy Property (FPRS). These are of right-angled type and are constructed by Abramenko--M\"uhlherr (cf.\ \cite[Remark before Lemma $5$]{CR09} and also \cite[Corollary~B]{BiRGD}). Until now it was unclear whether there are also examples of $2$-spherical type which do not satisfy (FPRS). We provide the existence of a $2$-spherical RGD-system which does not satisfy Property (FPRS):

\begin{thmintro}
	(See Theorem \ref{Theorem: not FPRS}) There exists an RGD-system of type $(4, 4, 4)$ over $\FF_2$ which does not satisfy Property (FPRS).
\end{thmintro}

\begin{rkintro}
	\begin{enumerate}[label=(\alph*)]
		\item Using similar arguments as in \cite[Lemma $5$]{CR09} one can construct infinitely many RGD-systems of type $(4, 4, 4)$ over $\FF_2$ satisfying Property (FPRS). The geometric completion of such groups belongs to the class $\mathcal{S}$ consisting of topologically simple, non-discrete, compactly generated, totally disconnected, locally compact groups, for which Caprace, Reid and Willis initiated a systematic study in \cite{CRW17b}.
		
		\item By \cite[Corollary~$3.1$]{CM11} the geometric completion of any RGD-system of irreducible type with finite root groups contains a closed cocompact normal subgroup which is topologically simple and, in particular, belongs to the class $\mathcal{S}$. Thus the geometric completion of each example mentioned in Corollary~\ref{Main result: uncountably many RGD-systems} gives rise to a group in $\mathcal{S}$. The question whether these examples are pairwise non-isomorphic is a difficult problem.
	\end{enumerate}
\end{rkintro}

Finally, we come back to Tits' local-to-global conjecture about buildings. More precisely, the conjecture is about the question whether the extension theorem for isometries of spherical buildings -- the decisive step in Tits' classification of irreducible spherical buildings of rank at least $3$ (cf.\ \cite{Ti74}) -- can be carried over to $2$-spherical twin buildings (cf.\ \cite[Remark~$5.9$(f) and Conjecture~1~\&~1']{Ti92}). For more information about the extension problem we refer to \cite{MR95} and \cite{BM23}. 

In \cite{MR95} M\"uhlherr and Ronan confirmed the conjecture under some mild condition -- they called (co) -- which excludes a very short list of small residues of rank $2$. First it was expected that condition (co) is merely needed in their proof and can be dropped in general. However, after a while experts started to have serious doubts about the general validity of Tits' conjecture. In this article we confirm those doubts:

\begin{thmintro}\label{Main result: Extension theorem}
	(See Theorem~\ref{Theorem: extension theorem}) The local-to-global principle does not hold for thick $2$-spherical twin buildings.
\end{thmintro}

\subsection*{Overview}

In Section~\ref{Section: Preliminaries} we fix notation and recall some facts about Coxeter systems and trees of groups. In Section~\ref{Section: commutator blueprint of type ((W, S), D)} we recall the definition of commutator blueprints of type $(4, 4, 4)$, which are the central objects in this paper. In Section~\ref{Section: Locally Weyl-invariant Commutator blueprints} we introduce some tree products related to locally Weyl-invariant commutator blueprints of type $(4, 4, 4)$. We prove some subgroup and isomorphism properties of those tree products. We highly recommend considering the diagrams in the appendix when reading Section~\ref{Section: Locally Weyl-invariant Commutator blueprints}. All statements look rather technical, but have a nice geometric interpretation which resolves the technicalities. In Section~\ref{Section: Natural subgroups} we define a sequence of groups $(G_i)_{i\in \NN}$. Each group $G_i$ is given by a presentation. Roughly speaking, it is generated by elements $u_{\alpha}$, where $\alpha$ is a positive root which does not contain a suitable $n$-ball around $1_W$, and the fundamental relations are only the \emph{obvious} relations. We show that the direct limit of the family $(G_i)_{i\in \NN}$ is isomorphic to the group $U_+ := \lim U_w$ (cf.\ Lemma~\ref{Lemma: U_+ isomorphic to G}). To show that any locally Weyl-invariant commutator blueprint of type $(4, 4, 4)$ is faithful, we have to show that the canonical homomorphisms $U_w \to U_+$ are injective. A priori it is not clear whether this is the case. However, this follows if all the homomorphisms $G_i \to G_{i+1}$ are injective. We end Section~\ref{Section: Natural subgroups} by introducing what it means for the group $G_i$ to be \emph{natural}. The main goal of Section~\ref{Section: Detecting tree products} is Proposition~\ref{Giinjective}, where we prove that the canonical homomorphism $G_i \to G_{i+1}$ is injective provided that $G_i$ is natural. Section~\ref{Section: Main result} is devoted to the proof that for each $i \geq 0$ the group $G_i$ is natural. This is done by induction on $i$. In Section~\ref{Section: Applications} we prove many consequences of this result or, more precisely, of Theorem~\ref{Main result: Theorem integrable Weyl-invariant}.

\begin{rkintro}
	We should mention here that in the proof of the statement that $G_0$ is natural (cf.\ Lemma~\ref{Lemma: G0 natural}) we use the existence of the Kac--Moody group $\mathcal{G}$ of type $(4, 4, 4)$ over $\FF_2$ as well as the existence of a canonical homomorphism $G_0 \to \mathcal{G}$. This ensures that $G_0$ is not \emph{too small}. For details we refer to \cite{BiRGDandTreeproducts}.
\end{rkintro}

\subsection*{Acknowledgement}

I would like to thank Bernhard M\"uhlherr for drawing my attention to this questions. I would like to thank him, Richard Weidmann and Stefan Witzel for helpful discussions. Finally, I would like to thank the anonymous referees for their careful reading of the paper and detailed comments.

\section{Preliminaries}\label{Section: Preliminaries}

\subsection*{Coxeter systems}

Let $(W, S)$ be a Coxeter system and let $\ell$ denote the corresponding length function. The \emph{rank} of the Coxeter system is the cardinality of the set $S$. For the purpose of this paper, we assume that all Coxeter systems are of finite rank.

\begin{convention}
	In this paper we let $(W, S)$ be a Coxeter system of finite rank.
\end{convention}

It is well-known that for each $J \subseteq S$ the pair $(\langle J \rangle, J)$ is a Coxeter system (cf.\ \cite[Ch. IV, §$1$ Theorem $2$]{Bo02}). For $s, t \in S$ we denote the order of $st$ in $W$ by $m_{st}$. The \emph{Coxeter diagram} corresponding to $(W, S)$ is the labeled graph $(S, E(S))$, where $E(S) = \{ \{s, t \} \mid m_{st}>2 \}$ and where each edge $\{s,t\}$ is labeled by $m_{st}$ for all $s, t \in S$. We say that $(W, S)$ is \emph{of type $(4, 4, 4)$} if $(W, S)$ is of rank $3$ and $m_{st} = 4$ for all $s\neq t \in S$.

A subset $J \subseteq S$ is called \textit{spherical} if $\langle J \rangle$ is finite. The Coxeter system is called \textit{$2$-spherical} if $\langle J \rangle$ is finite for all $J \subseteq S$ containing at most $2$ elements (i.e.\ $m_{st} < \infty$ for all $s, t \in S$). Given a spherical subset $J$ of $S$, there exists a unique element of maximal length in $\langle J \rangle$, which we denote by $r_J$ (cf.\ \cite[Corollary $2.19$]{AB08}).

\begin{lemma}[see {\cite[Lemma $3.4$]{Bischof_On_Growth_Functions_of_Coxeter_Groups}} and {\cite[Lemma $2.16$]{BiConstruction}}]\label{Lemma: not both down}
	Suppose $(W, S)$ is of type $(4, 4, 4)$ and $S = \{r, s, t\}$. Let $w \in W$ with $\ell(ws) = \ell(w) +1 = \ell(wt)$. Then $\ell(w) +2 \in \{ \ell(wsr), \ell(wtr) \}$. Moreover, if $\ell(wsr) = \ell(w)$, then $\ell(wsrt) = \ell(w)+1$.
\end{lemma}

\subsection*{Buildings}

A \textit{building of type $(W, S)$} is a pair $\Delta = (\mc{C}, \delta)$ where $\mc{C}$ is a non-empty set and where $\delta: \mc{C} \times \mc{C} \to W$ is a \textit{distance function} satisfying the following axioms, where $x, y\in \mc{C}$ and $w = \delta(x, y)$:
\begin{enumerate}[label=(Bu\arabic*)]
	\item $w = 1_W$ if and only if $x=y$.
	
	\item If $z\in \mc{C}$ satisfies $s := \delta(y, z) \in S$, then $\delta(x, z) \in \{w, ws\}$, and if, furthermore, $\ell(ws) = \ell(w) +1$, then $\delta(x, z) = ws$.
	
	\item If $s\in S$, there exists $z\in \mc{C}$ such that $\delta(y, z) = s$ and $\delta(x, z) = ws$.
\end{enumerate}
The \textit{rank} of $\Delta$ is the rank of the underlying Coxeter system. The elements of $\mc{C}$ are called \textit{chambers}. Given $s\in S$ and $x, y \in \mc{C}$, then $x$ is called \textit{$s$-adjacent} to $y$, if $\delta(x, y) = s$. The chambers $x, y$ are called \textit{adjacent}, if they are $s$-adjacent for some $s\in S$. A \textit{gallery} from $x$ to $y$ is a sequence $(x = x_0, \ldots, x_k = y)$ such that $x_{l-1}$ and $x_l$ are adjacent for all $1 \leq l \leq k$; the number $k$ is called the \textit{length} of the gallery. Let $(x_0, \ldots, x_k)$ be a gallery and suppose $s_i \in S$ with $\delta(x_{i-1}, x_i) = s_i$. Then $(s_1, \ldots, s_k)$ is called the \textit{type} of the gallery. A gallery from $x$ to $y$ of length $k$ is called \textit{minimal} if there is no gallery from $x$ to $y$ of length $<k$. In this case we have $\ell(\delta(x, y)) = k$ (cf.\ \cite[Corollary $5.17(1)$]{AB08}). Let $x, y, z \in \mc{C}$ be chambers such that $\ell(\delta(x, y)) = \ell(\delta(x, z)) + \ell(\delta(z, y))$. Then the concatenation of a minimal gallery from $x$ to $z$ and a minimal gallery from $z$ to $y$ yields a minimal gallery from $x$ to $y$.

Given a subset $J \subseteq S$ and $x\in \mc{C}$, the \textit{$J$-residue of $x$} is the set $R_J(x) := \{y \in \mc{C} \mid \delta(x, y) \in \langle J \rangle \}$. Each $J$-residue is a building of type $(\langle J \rangle, J)$ with the distance function induced by $\delta$ (cf.\ \cite[Corollary $5.30$]{AB08}). A \textit{residue} is a subset $R$ of $\mc{C}$ such that there exist $J \subseteq S$ and $x\in \mc{C}$ with $R = R_J(x)$. Since the subset $J$ is uniquely determined by $R$, the set $J$ is called the \textit{type} of $R$ and the \textit{rank} of $R$ is defined to be the cardinality of $J$. A residue is called \textit{spherical} if its type is a spherical subset of $S$. A \textit{panel} is a residue of rank $1$. An \textit{$s$-panel} is a panel of type $\{s\}$ for $s\in S$. The building $\Delta$ is called \textit{thick}, if each panel of $\Delta$ contains at least three chambers.

Given $x\in \mc{C}$ and a $J$-residue $R \subseteq \mc{C}$, then there exists a unique chamber $z\in R$ such that $\ell(\delta(x, y)) = \ell(\delta(x, z)) + \ell(\delta(z, y))$ holds for each $y\in R$ (cf.\ \cite[Proposition $5.34$]{AB08}). The chamber $z$ is called the \textit{projection of $x$ onto $R$} and is denoted by $\proj_R x$. Moreover, if $z = \proj_R x$ we have $\delta(x, y) = \delta(x, z) \delta(z, y)$ for each $y\in R$.

An \textit{(type-preserving) automorphism} of a building $\Delta = (\mc{C}, \delta)$ is a bijection $\phi:\mc{C} \to \mc{C}$ such that $\delta(\phi(c), \phi(d)) = \delta(c, d)$ holds for all chambers $c, d \in \mc{C}$. We remark that some authors distinguish between automorphisms and type-preserving automorphisms. An automorphism in our sense is type-preserving. We denote the set of all automorphisms of the building $\Delta$ by $\Aut(\Delta)$.

\begin{example}
	We define $\delta: W \times W \to W, (x, y) \mapsto x^{-1}y$. Then $\Sigma(W, S) := (W, \delta)$ is a building of type $(W, S)$, which we call the \emph{Coxeter building} of type $(W, S)$. The group $W$ acts faithfully on $\Sigma(W, S)$ by multiplication from the left, i.e.\ $W \leq \Aut(\Sigma(W, S))$.
\end{example}

A subset $\Sigma \subseteq \mc{C}$ is called \textit{convex} if for any two chambers $c, d \in \Sigma$ and any minimal gallery $(c_0 = c, \ldots, c_k = d)$, we have $c_i \in \Sigma$ for all $0 \leq i \leq k$. A subset $\Sigma \subseteq \mc{C}$ is called \textit{thin} if $P \cap \Sigma$ contains exactly two chambers for every panel $P \subseteq \mc{C}$ which meets $\Sigma$. An \textit{apartment} is a non-empty subset $\Sigma \subseteq \mc{C}$, which is convex and thin.

\subsection*{Roots}

A \textit{reflection} is an element of $W$ that is conjugate to an element of $S$. For $s\in S$ we let $\alpha_s := \{ w\in W \mid \ell(sw) > \ell(w) \}$ be the \textit{simple root} corresponding to $s$. A \textit{root} is a subset $\alpha \subseteq W$ such that $\alpha = v\alpha_s$ for some $v\in W$ and $s\in S$. We denote the set of all roots by $\Phi := \Phi(W, S)$. The set $\Phi_+ = \{ \alpha \in \Phi \mid 1_W \in \alpha \}$ is the set of all \textit{positive roots} and $\Phi_- = \{ \alpha \in \Phi \mid 1_W \notin \alpha \}$ is the set of all \textit{negative roots}. For each root $\alpha \in \Phi$, the complement $-\alpha := W \backslash \alpha$ is again a root; it is called the root \emph{opposite} to $\alpha$. We denote the unique reflection which interchanges these two roots by $r_{\alpha} \in W \leq \Aut(\Sigma(W, S))$. For $w\in W$ we define $\Phi(w) := \{ \alpha \in \Phi_+ \mid w \notin \alpha \}$. Note that for $w\in W$ and $s\in S$ we have $\Phi(sw) \backslash \{ \alpha_s \} = s\left( \Phi(w) \backslash \{ \alpha_s \} \right) = \{ s\alpha \mid \alpha \in \Phi(w) \backslash \{ \alpha_s\} \}$. In particular, for $s\in S$ and $\alpha \in \Phi_+ \backslash \{ \alpha_s \}$ we have $s\alpha \in \Phi_+$. A pair $\{ \alpha, \beta \}$ of roots is called \textit{prenilpotent} if both $\alpha \cap \beta$ and $(-\alpha) \cap (-\beta)$ are non-empty sets. For such a pair we will write $\left[ \alpha, \beta \right] := \{ \gamma \in \Phi \mid \alpha \cap \beta \subseteq \gamma \text{ and } (-\alpha) \cap (-\beta) \subseteq -\gamma \}$ and $(\alpha, \beta) := \left[ \alpha, \beta \right] \backslash \{ \alpha, \beta \}$. A pair $\{ \alpha, \beta \} \subseteq \Phi$ of two roots is called \textit{nested}, if $\alpha \subseteq \beta$ or $\beta \subseteq \alpha$.

\begin{lemma}\label{Lemma: intersection of roots}
	For $s\neq t \in S$ we have $\alpha_t \subseteq (-\alpha_s) \cup t\alpha_s$.
\end{lemma}
\begin{proof}
	Let $w\in \alpha_t$. If $\ell(sw) < \ell(w)$, then $w\in (-\alpha_s)$ and we are done. Thus we can assume $\ell(sw) > \ell(w)$. As $w\in \alpha_t$, we have $\ell(tw) > \ell(w)$ and hence $\ell(stw) = \ell(w) +2 > \ell(tw)$ (see \cite[Proposition~$2.20$]{AB08}). This implies $tw \in \alpha_s$ and we infer $w\in t\alpha_s$.
\end{proof}

Suppose $s, t \in S$ are such that $m_{st} < \infty$. Then $\{s, t\}$ is a spherical subset of $S$. Recall that we denote by $r_{\{s, t\}}$ the unique element of maximal length in $\langle s, t \rangle$.

\begin{lemma}[{\cite[Lemma~2.7]{BiRGDandTreeproducts}}]\label{Lemma: Subset contained in certain root}
	Suppose $(W, S)$ is of type $(4, 4, 4)$ and $S = \{r, s, t\}$. Then we have $tstr\alpha_s \cap stsr\alpha_t \cap (W\backslash \{ r_{\{s, t\}}r \}) \subseteq r_{\{s, t\}} \alpha_r$.
\end{lemma}

\subsection*{Coxeter buildings}

In this subsection we consider the Coxeter building $\Sigma(W, S)$. At first we note that roots are convex (cf.\ \cite[Lemma $3.44$]{AB08}). For $\alpha \in \Phi$ we denote by $\partial \alpha$ (resp.\ $\partial^2 \alpha$) the set of all panels (resp.\ spherical residues of rank $2$) stabilized by $r_{\alpha}$. Furthermore, we define $\mathcal{C}(\partial \alpha) := \bigcup_{P \in \partial \alpha} P$ and $\mathcal{C}(\partial^2 \alpha) := \bigcup_{R \in \partial^2 \alpha} R$. The set $\partial \alpha$ is called the \textit{wall} associated with $\alpha$. Let $\Gamma = (c_0, \ldots, c_k)$ be a gallery. We say that $\Gamma$ \textit{crosses the wall $\partial \alpha$} if there exists $1 \leq i \leq k$ such that $\{ c_{i-1}, c_i \} \in \partial \alpha$. It is a basic fact that a minimal gallery crosses a wall at most once (cf.\ \cite[Lemma $3.69$]{AB08}). Let $(c_0, \ldots, c_k)$ and $(d_0 = c_0, \ldots, d_k = c_k)$ be two minimal galleries from $c_0$ to $c_k$ and let $\alpha \in \Phi$. Then $\partial \alpha$ is crossed by the minimal gallery $(c_0, \ldots, c_k)$ if and only if it is crossed by the minimal gallery $(d_0, \ldots, d_k)$. For a minimal gallery $\Gamma = (c_0, \ldots, c_k)$ the \emph{sequence of roots crossed by $\Gamma$} is the unique sequence of roots $(\alpha_1, \ldots, \alpha_k)$ such that $c_{i-1} \in \alpha_i$ and $c_i \notin \alpha_i$ for all $1 \leq i \leq k$. Moreover, if $k \geq 1$, we define $\alpha_{\Gamma} := \alpha_k$, i.e.\ $\alpha_{\Gamma}$ is the last root crossed by $\Gamma$.

\begin{lemma}[{\cite[Lemma $2.18$]{BiConstruction}}]\label{mingallinrep}
	Suppose $(W, S)$ is of type $(4, 4, 4)$ and $S = \{r, s, t\}$. Let $\Gamma$ be a minimal gallery of type $(r, s, t, r)$ and let $(\beta_1, \beta_2, \beta_3, \beta_4)$ be the sequence of roots crossed by $\Gamma$. Then $\beta_1 \subsetneq \beta_3$ and $\beta_1 \subsetneq \beta_4$.
\end{lemma}

We denote the set of all minimal galleries $(c_0 = 1_W, \ldots, c_k)$ starting at $1_W$ by $\mathrm{Min}$. For $w\in W$ we denote the set of all $\Gamma \in \mathrm{Min}$ of type $(s_1, \ldots, s_k)$ with $w = s_1 \cdots s_k$ by $\mathrm{Min}(w)$. For $w\in W$ and $s\in S$ with $\ell(sw) = \ell(w) -1$ we let $\mathrm{Min}_s(w)$ be the set of all $\Gamma \in \mathrm{Min}(w)$ of type $(s, s_2, \ldots, s_k)$. We extend this notion to the case $\ell(sw) = \ell(w) +1$ by defining $\mathrm{Min}_s(w) := \mathrm{Min}(w)$. Let $w\in W$, $s\in S$ and $\Gamma = (c_0, \ldots, c_k) \in \mathrm{Min}_s(w)$. If $\ell(sw) = \ell(w) -1$, then $c_1 = s$ and we define $s\Gamma := (sc_1 = 1_W, \ldots, sc_k) \in \mathrm{Min}(sw)$. If $\ell(sw) = \ell(w) +1$, we define $s\Gamma := (1_W, sc_0 = s, \ldots, sc_k) \in \mathrm{Min}(sw)$.

Let $\Gamma = (c_0, \ldots, c_k) \in \mathrm{Min}$ and let $(\alpha_1, \ldots, \alpha_k)$ be the sequence of roots crossed by $\Gamma$. We define $\Phi(\Gamma) := \{ \alpha_i \mid 1 \leq i \leq k \}$. Using the indices we obtain an order $\leq_\Gamma$ on $\Phi(\Gamma)$ and, in particular, on $[\alpha, \beta] = [\beta, \alpha] \subseteq \Phi(\Gamma)$ for all $\alpha, \beta \in \Phi(\Gamma)$. Note that $\Phi(\Gamma) = \Phi(w)$ holds for every $\Gamma \in \mathrm{Min}(w)$.

For a positive root $\alpha \in \Phi_+$ we define $k_{\alpha} := \min \{ k\in \NN \mid \exists \Gamma = (c_0, \ldots, c_k) \in \mathrm{Min}: \alpha_{\Gamma} = \alpha \}$. We remark that $k_{\alpha} = 1$ if and only if $\alpha$ is a simple root. Furthermore, we define $\Phi(k) := \{ \alpha \in \Phi_+ \mid k_{\alpha} \leq k \}$ for $k \in \NN$. Let $R$ be a residue and let $\alpha \in \Phi_+$. Then we call $\alpha$ a \textit{simple root of $R$} if there exists $P \in \partial \alpha$ such that $P \subseteq R$ and $\proj_R 1_W = \proj_P 1_W$. In this case $R$ is also stabilized by $r_{\alpha}$ and hence $R \in \partial^2 \alpha$.

\begin{remark}\label{Remark: Existence residue}
	Let $\alpha \in \Phi_+$ be a positive root such that $k_{\alpha} >1$. Let $\Gamma = (c_0, \ldots, c_{k_{\alpha}}) \in \mathrm{Min}$ be a minimal gallery with $\{ c_{k_{\alpha} -1}, c_{k_{\alpha}} \} \in \partial \alpha$. Then $\alpha$ is not a simple root of the rank $2$ residue containing $c_{k_{\alpha} -2}, c_{k_{\alpha}-1}, c_{k_{\alpha}}$. In particular, there exists $R \in \partial^2 \alpha$ such that $\alpha$ is not a simple root of $R$.
\end{remark}

\subsection*{Roots in Coxeter systems of type $\mathbf{(4, 4, 4)}$}

Suppose that $(W, S)$ is of type $(4, 4, 4)$ and that $S = \{r, s, t\}$. Let $\alpha \in \Phi_+$ be a root such that $k_{\alpha} >1$, i.e.\ $\alpha$ is not a simple root. Let $R \in \partial^2 \alpha$ be a residue such that $\alpha$ is not a simple root of $R$ (for the existence of such a residue see Remark~\ref{Remark: Existence residue}). Let $P \neq P' \in \partial \alpha$ be contained in $R$. Then $\ell(1_W, \proj_P 1_W) \neq \ell(1_W, \proj_{P'} 1_W)$ and we can assume that $\ell(1_W, \proj_P 1_W) < \ell(1_W, \proj_{P'} 1_W)$. Let $\Gamma = (c_0, \ldots, c_k) \in \mathrm{Min}$ be of type $(s_1, \ldots, s_k)$ such that $c_{k-2} = \proj_R 1_W, c_{k-1} = \proj_P 1_W$ and $c_k \in P \backslash \{c_{k-1}\}$. For $P \neq Q := \{ x, y \} \in \partial \alpha$ with $x\in \alpha$ and $y \notin \alpha$ we let $P_0 = P, \ldots, P_n = Q$ and $R_1, \ldots, R_n$ be as in \cite[Proposition $2.7$]{CM06}. We assume that $r \notin \{ s_{k-1}, s_k \}$.

\begin{lemma}[{\cite[Lemma~$2.22$]{BiConstruction}}]\label{Lemma: Lemma 2.22 BiConstruction}
	We have $k = k_{\alpha}$ and the panel $P_{\alpha} := P$ is the unique panel in $\partial \alpha$ with $\ell( 1_W, \proj_{P_{\alpha}} 1_W ) = k_{\alpha} -1$.
\end{lemma}

\begin{lemma}[{\cite[Lemma $2.23$]{BiConstruction}}]\label{projgal}
	We define $R_{\alpha, Q}$ to be the residue $R_1$ if $R \neq R_1$ and $\ell(s_1 \cdots s_{k-1} r) = k-2$. In all other cases, we define $R_{\alpha, Q} := R$. Then there exists a minimal gallery $\Gamma' = (d_0 = c_0, \ldots, d_m = \proj_Q c_0, y)$ with the following properties:
	\begin{itemize}
		\item There exists $0 \leq i \leq m$ such that $d_i = \proj_{R_{\alpha, Q}} 1_W$.
		
		\item For each $i+1 \leq j \leq m$ there exists $L_j \in \partial^2 \alpha$ with $\{ d_{j-1}, d_j \} \subseteq L_j$. In particular, we have $d_j \in \mathcal{C}(\partial^2\alpha)$.
	\end{itemize}
\end{lemma}

\begin{lemma}[{\cite[Lemma $2.25$]{BiConstruction}}]\label{prenilpotentrootsintersectinoneresidue}
	Let $\beta \in \Phi(k) \backslash \{ \alpha_s \mid s\in S \}$ be a root such that $R \notin \partial^2 \beta$ and such that the order of the product $r_{\alpha} r_{\beta}$ is finite. Moreover, we assume that $\ell(s_1 \cdots s_{k-1}r) = k$. Then one of the following holds:
	\begin{enumerate}[label=(\alph*)]
		\item $\beta = \alpha_{\Gamma'}$, where $\Gamma' \in \mathrm{Min}$ is the minimal gallery of type $(s_1, \ldots, s_{k-1}, r)$.
		
		\item $\beta = \alpha_{\Gamma'}$, where $\Gamma' \in \mathrm{Min}$ is the minimal gallery of type $(s_1, \ldots, s_{k-2}, s_k, s_{k-1}, r)$, and we have $\ell(s_1 \cdots s_{k-2} s_k r) = k-2$.
	\end{enumerate}
\end{lemma}

\subsection*{Twin buildings}

Let $\Delta_+ = (\mc{C}_+, \delta_+)$ and $\Delta_- = (\mc{C}_-, \delta_-)$ be two buildings of the same type $(W, S)$. A \textit{codistance} (or a \textit{twinning}) between $\Delta_+$ and $\Delta_-$ is a mapping $\delta_* : \left( \mc{C}_+ \times \mc{C}_- \right) \cup \left( \mc{C}_- \times \mc{C}_+ \right) \to W$ satisfying the following axioms, where $\epsilon \in \{+,-\}$, $x\in \mc{C}_{\epsilon}$, $y\in \mc{C}_{-\epsilon}$ and $w=\delta_*(x, y)$:
\begin{enumerate}[label=(Tw\arabic*)]
	\item $\delta_*(y, x) = w^{-1}$.
	
	\item If $z\in \mc{C}_{-\epsilon}$ is such that $s := \delta_{-\epsilon}(y, z) \in S$ and $\ell(ws) = \ell(w) -1$, then $\delta_*(x, z) = ws$.
	
	\item If $s\in S$, there exists $z\in \mc{C}_{-\epsilon}$ such that $\delta_{-\epsilon}(y, z) = s$ and $\delta_*(x, z) = ws$.
\end{enumerate}
A \textit{twin building of type $(W, S)$} is a triple $\Delta = (\Delta_+, \Delta_-, \delta_*)$ where $\Delta_+ = (\mc{C}_+, \delta_+)$, $\Delta_- = (\mc{C}_-, \delta_-)$ are buildings of type $(W, S)$ and where $\delta_*$ is a twinning between $\Delta_+$ and $\Delta_-$.

Let $\epsilon \in \{+,-\}$. For $x\in \mc{C}_{\epsilon}$ we put $x^{\mathrm{op}} := \{ y\in \mc{C}_{-\epsilon} \mid \delta_*(x, y) = 1_W \}$. It is a direct consequence of (Tw$1$) that $y\in x^{\mathrm{op}}$ if and only if $x\in y^{\mathrm{op}}$ for any pair $(x, y) \in \mc{C}_{\epsilon} \times \mc{C}_{-\epsilon}$. If $y\in x^{\mathrm{op}}$ then we say that $y$ is \textit{opposite} to $x$ or that \textit{$(x, y)$ is a pair of opposite chambers}.

A \textit{residue} (resp.\ \textit{panel}) of $\Delta$ is a residue (resp.\ panel) of $\Delta_+$ or $\Delta_-$; given a residue $R$ of $\Delta$ then we define its type and rank as before. The twin building $\Delta$ is called \textit{thick} if $\Delta_+$ and $\Delta_-$ are thick.

Let $\Sigma_+ \subseteq \mc{C}_+$ and $\Sigma_- \subseteq \mc{C}_-$ be apartments of $\Delta_+$ and $\Delta_-$, respectively. Then the set $\Sigma := \Sigma_+ \cup \Sigma_-$ is called \textit{twin apartment} if $\lvert x^{\mathrm{op}} \cap \Sigma \rvert = 1$ holds for each $x\in \Sigma$. If $(x, y)$ is a pair of opposite chambers, then there exists a unique twin apartment containing $x$ and $y$. We will denote it by $\Sigma(x, y)$.

An \textit{automorphism} of $\Delta$ is a bijection $\phi: \mc{C}_+ \cup \mc{C}_- \to \mc{C}_+ \cup \mc{C}_-$ such that $\phi$ preserves the sign, the distance functions $\delta_{\epsilon}$ and the codistance $\delta_*$.

\subsection*{Root group data}\label{Section: Root group data}

An \textit{RGD-system of type $(W, S)$} is a pair $\mathcal{D} = \left( G, \left( U_{\alpha} \right)_{\alpha \in \Phi}\right)$ consisting of a group $G$ together with a family of subgroups $U_{\alpha}$ (called \textit{root groups}) indexed by the set of roots $\Phi$, which satisfies the following axioms, where $H := \bigcap_{\alpha \in \Phi} N_G(U_{\alpha})$ and $U_{\epsilon} := \langle U_{\alpha} \mid \alpha \in \Phi_{\epsilon} \rangle$ for $\epsilon \in \{+, -\}$:
\begin{enumerate}[label=(RGD\arabic*)] \setcounter{enumi}{-1}
	\item For each $\alpha \in \Phi$, we have $U_{\alpha} \neq \{1\}$.
	
	\item For each prenilpotent pair $\{ \alpha, \beta \} \subseteq \Phi$ with $\alpha \neq \beta$, the commutator group $[U_{\alpha}, U_{\beta}]$ is contained in the group $U_{(\alpha, \beta)} := \langle U_{\gamma} \mid \gamma \in (\alpha, \beta) \rangle$.
	
	\item For every $s\in S$ and each $u\in U_{\alpha_s} \backslash \{1\}$, there exist $u', u'' \in U_{-\alpha_s}$ such that the product $m(u) := u' u u''$ conjugates $U_{\beta}$ onto $U_{s\beta}$ for each $\beta \in \Phi$.
	
	\item For each $s\in S$, the group $U_{-\alpha_s}$ is not contained in $U_+$.
	
	\item $G = H \langle U_{\alpha} \mid \alpha \in \Phi \rangle$.
\end{enumerate}
For $w\in W$ we define $U_w := \langle U_{\alpha} \mid w \notin \alpha \in \Phi_+ \rangle$. Let $\Gamma \in \mathrm{Min}(w)$ and let $(\alpha_1, \ldots, \alpha_k)$ be the sequence of roots crossed by $\Gamma$. Then we have $U_w = U_{\alpha_1} \cdots U_{\alpha_k}$ (cf.\ \cite[Corollary~$8.34(1)$]{AB08}). An RGD-system $\mathcal{D} = (G, (U_{\alpha})_{\alpha \in \Phi})$ is said to be \textit{over $\FF_2$} if every root group has cardinality $2$. In this case we denote for $\alpha \in \Phi$ the non-trivial element in $U_{\alpha}$ by $u_{\alpha}$.

Let $\mathcal{D} = (G, (U_{\alpha})_{\alpha \in \Phi})$ be an RGD-system of type $(W, S)$ and let $H = \bigcap_{\alpha \in \Phi} N_G(U_{\alpha})$, $B_{\epsilon} = H \langle U_{\alpha} \mid \alpha \in \Phi_{\epsilon} \rangle$ for $\epsilon \in \{+, -\}$. It follows from \cite[Theorem $8.80$]{AB08} that there exists an \textit{associated} twin building $\Delta(\mathcal{D}) = (\Delta(\mathcal{D})_+, \Delta(\mathcal{D})_-, \delta_*)$ of type $(W, S)$ such that $\Delta(\mathcal{D})_{\epsilon} = ( G/B_{\epsilon}, \delta_{\epsilon} )$ for $\epsilon \in \{ +, - \}$ and $G$ acts on $\Delta(\mathcal{D})$ by multiplication from the left. There is a distinguished pair of opposite chambers in $\Delta(\mathcal{D})$ corresponding to the subgroups $B_{\epsilon}$ for $\epsilon \in \{+, -\}$. We will denote this pair by $(c_+, c_-)$.

\subsection*{Graphs of groups}

This subsection is based on \cite[Section $2$]{KWM05} and \cite{Se79}.

Following Serre, a \textit{graph} $\Gamma$ consists of a vertex set $V\Gamma$, an edge set $E\Gamma$, the inverse function $^{-1}:E\Gamma \to E\Gamma$ and two edge endpoint functions $o: E\Gamma \to V\Gamma$, $t: E\Gamma \to V\Gamma$ satisfying the following axioms:
\begin{enumerate}[label=(\roman*)]
	\item The function $^{-1}$ is a fixed-point free involution on $E\Gamma$.
	
	\item For each $e\in E\Gamma$ we have $o(e) = t(e^{-1})$.
\end{enumerate}
A \textit{tree of groups} is a triple $\mathbb{G} = (T, (G_v)_{v\in V\Gamma}, (G_e)_{e\in E\Gamma})$ consisting of a finite tree $T$ (i.e.\ $VT$ and $ET$ are finite), a family of \textit{vertex groups} $(G_v)_{v\in VT}$ and a family of \textit{edge groups} $(G_e)_{e\in ET}$. Every edge $e \in ET$ comes equipped with two \textit{boundary monomorphisms} $\alpha_e: G_e \to G_{o(e)}$ and $\omega_e: G_e \to G_{t(e)}$. We assume that for each $e\in ET$ we have $G_{e^{-1}} = G_e$, $\alpha_{e^{-1}} = \omega_e$ and $\omega_{e^{-1}} = \alpha_e$. We let $G_T := \lim \mathbb{G}$ be the direct limit of the inductive system formed by the vertex groups, edge groups and boundary monomorphisms and call $G_T$ a \textit{tree product}. A \textit{sequence of groups} is a tree of groups where the underlying graph is a sequence. If the tree $T$ is a \emph{segment}, i.e.\ $VT = \{v, w\}$ and $ET = \{ e, e^{-1} \}$, then the tree product $G_T$ is an amalgamated product. We will use the notation from amalgamated products and we will write $G_T = G_v \star_{G_e} G_w$. We extend this notation to arbitrary \emph{sequences} $T$: if $VT = \{ v_0, \ldots, v_n \}$, $ET = \{ e_i, e_i^{-1} \mid 1 \leq i \leq n \}$ and $o(e_i) = v_{i-1}, t(e_i) = v_i$, then we will write $G_T = G_{v_0} \star_{G_{e_1}} G_{v_1} \star_{G_{e_2}} \cdots \star_{G_{e_n}} G_{v_n}$. If $T$ is a \emph{star}, i.e.\ $VT = \{ v_0, \ldots, v_n \}$, $ET = \{ e_i, e_i^{-1} \mid 1 \leq n \}$ and $o(e_i) = v_0$, $t(e_i) = v_i$, then we will write $G_T = \star_{G_0} G_i$.

\begin{proposition}[{\cite[Theorem $1$]{KS70}}]\label{treeproducts}
%	Let $\mathbb{G} = (T, (G_v)_{v\in VT}, (G_e)_{e\in ET})$ be a tree of groups. If $T$ is partitioned into subtrees whose tree products are $G_1, \ldots, G_n$ and the subtrees are contracted to vertices, then $G_T$ is isomorphic to the tree product of the tree of groups whose vertex groups are the $G_i$ and the edge groups are the $G_e$, where $e$ is the unique edge which joins two subtrees. Moreover, $G_i \to G_T$ is injective.
	Let $\mathbb{G} = (T, (G_v)_{v\in VT}, (G_e)_{e\in ET})$ be a tree of groups. Suppose that $T$ is partitioned into subtrees $T_1, \ldots, T_n$. Then we define a new tree $\overline{T}$ with vertex set $V \overline{T} = \{ v_1, \ldots, v_n \}$, edge set $E \overline{T} = \{ e_{ij} \mid 1 \leq i \neq j \leq n, \exists f \in ET: o(f) \in VT_i, t(f) \in V T_j \}$, and $o(e_{ij}) = v_i$, $t(e_{ij}) = v_j$ and $e_{ij}^{-1} = e_{ji}$ for all $1 \leq i \neq j \leq n$. For each $i \in \{1, \ldots, n\}$ we denote by $H_{v_i}$ the tree product of the tree of groups $\mathbb{G}$ restricted to $T_i$. We define a new tree of groups $\overline{\mathbb{G}} = (\overline{T}, (H_v)_{v\in V \overline{T}}, (H_e)_{e \in E\overline{T}})$, where $H_{e_{ij}} := G_f$ where $f\in ET$ is such that $o(f) \in V T_i$ and $t(f) \in V T_j$. Then the tree product $G_T$ is isomorphic to the tree product $G_{\overline{T}}$. Moreover, $H_{v_i} \to G_T$ is injective.
\end{proposition}

\begin{proposition}[{\cite[Proposition $4.3$]{KWM05}} and {\cite[Proposition $20$]{Se79}}]\label{treeofgroupsinjective}
	Let $T$ be a tree and let $T'$ be a subtree of $T$. Moreover, we let $\mathbb{G} = (T, (G_v)_{v\in VT}, (G_e)_{e\in ET})$ and $\mathbb{H} = (T', (H_v)_{v\in VT'}, (H_e)_{e\in ET'})$ be two trees of groups and suppose the following:
	\begin{enumerate}[label=(\roman*)]
		\item For each $v\in VT'$ we have $H_v \leq G_v$.
		
		\item\label{Case: ii} For each $e\in ET'$ we have $\alpha_e^{-1}(H_{o(e)}) = \omega_e^{-1}(H_{t(e)})$.
		
		\item For each $e\in ET'$ the group $H_e$ coincides with the group in \ref{Case: ii}.
	\end{enumerate}
	Then the canonical homomorphism $\nu: H_{T'} \to G_T$ between the tree product $H_{T'}$ and the tree product $G_T$ is injective. In particular, we have $\nu(H_{T'}) \cap G_v = H_v$ for each $v\in VT'$.
\end{proposition}
\begin{proof}
	This follows from \cite[Proposition~$4.3$]{KWM05} and \cite[Proposition~$20$]{Se79}.
\end{proof}

\begin{corollary}[{\cite[Corollary~4.3]{BiRGDandTreeproducts}}]\label{intersectionwithasubtree}
	Let $\mathbb{G} = (T, (G_v)_{v\in VT}, (G_e)_{e\in ET})$ be a tree of groups and let $H_v \leq G_v$ for each $v\in VT$. Assume that $H_e := \alpha_e^{-1}(H_{o(e)}) = \omega_e^{-1}(H_{t(e)})$ for all $e \in ET$ and let $\mathbb{H} = (T, (H_v)_{v\in VT}, (H_e)_{e\in ET})$ be the associated tree of groups. Let $T'$ be a subtree of $T$ and let $\mathbb{L} = (T', (G_v)_{v\in VT'}$, $(G_e)_{e\in ET'})$, $\mathbb{K} = (T', (H_v)_{v\in VT'}, (H_e)_{e\in ET'})$. Then $H_T \cap L_{T'} = K_{T'}$ in $G_T$.
\end{corollary}

\begin{corollary}[{\cite[Corollary~4.4]{BiRGDandTreeproducts}}]\label{AcapBisC}
	Let $A, B, C$ be groups and let $C \to A$, $C \to B$ be two monomorphisms. Then $A \cap B = C$ in $A\star_C B$.
\end{corollary}

\begin{remark}\label{Remark: isomorphism preserves amalgamated product}
	Let $A', A, B, C$ be groups, let $\alpha: C \to A$, $\beta:C \to B$ and $\alpha': C \to A'$ be monomorphisms and let $\phi: A \to A'$ be an isomorphism. If $\alpha' = \phi \circ \alpha$, then the amalgamated products $A \star_C B$ and $A' \star_C B$ are isomorphic. One can prove this by constructing two unique homomorphisms $A \star_C B \to A' \star_C B$ and $A' \star_C B \to A \star_C B$ such that the concatenation is the identity on $A$ (resp.\ $A'$) and on $B$.
\end{remark}

\begin{lemma}[{\cite[Lemma~4.6]{BiRGDandTreeproducts}}]\label{folding}
	Let $\mathbb{G} = (T, (G_v)_{v\in VT}, (G_e)_{e\in ET})$ be a tree of groups. Let $e \in ET$ and $G_e \leq H_{o(e)} \leq G_{o(e)}$. Let $VT' = VT \cup \{x\}$, $ET' = \left( ET \backslash \{ e, e^{-1} \} \right) \cup \{ f, f^{-1}, h, h^{-1} \}$ with $o(f) = o(e)$, $t(f) = x = o(h)$, $t(h) = t(e)$, $G_x := H_{o(e)} =: G_f$, $G_h := G_e$. Then the two tree products of the trees of groups are isomorphic.
\end{lemma}

\section{Commutator blueprints of type $(4, 4, 4)$}\label{Section: commutator blueprint of type ((W, S), D)}

In \cite{BiRGD} we have introduced \emph{commutator blueprints} of type $(W, S)$. In this paper we are only interested in the case where $(W, S)$ is of type $(4, 4, 4)$. For more information about general commutator blueprints we refer to \cite[Section $3$]{BiRGD}.

\begin{convention}
	In this section we let $(W, S)$ be of type $(4, 4, 4)$.
\end{convention}

We abbreviate $\mathcal{I} := \{ (\Gamma, \alpha, \beta) \in \mathrm{Min} \times \Phi_+ \times \Phi_+ \mid \alpha, \beta \in \Phi(\Gamma), \alpha \leq_{\Gamma} \beta \}$. Let $\left(M_{\alpha, \beta}^{\Gamma} \right)_{(\Gamma, \alpha, \beta) \in \mathcal{I}}$ be a family consisting of subsets $M_{\alpha, \beta}^{\Gamma} \subseteq (\alpha, \beta)$ ordered via $\leq_{\Gamma}$. For $w\in W$ we define the group $U_w$ via the following presentation:
\[ U_w := \left\langle \{ u_{\alpha} \mid \alpha \in \Phi(w) \} \;\middle|\; \begin{cases*}
	\forall \alpha \in \Phi(w): u_{\alpha}^2 = 1, \\
	\forall (\Gamma, \alpha, \beta) \in \mathcal{I}, \Gamma \in \mathrm{Min}(w): [u_{\alpha}, u_{\beta}] = \prod\nolimits_{\gamma \in M_{\alpha, \beta}^{\Gamma}} u_{\gamma}
\end{cases*} \right\rangle.
\]
Here the product is understood to be ordered via the order $\leq_{\Gamma}$, i.e.\ if $(\Gamma, \alpha, \beta) \in \mathcal{I}$ with $\Gamma \in \mathrm{Min}(w)$ and $M_{\alpha, \beta}^{\Gamma} = \{ \gamma_1 \leq_{\Gamma} \ldots \leq_{\Gamma} \gamma_k \} \subseteq (\alpha, \beta) \subseteq \Phi(\Gamma)$, then $\prod\nolimits_{\gamma \in M_{\alpha, \beta}^{\Gamma}} u_{\gamma} = u_{\gamma_1} \cdots u_{\gamma_k}$. Note that there could be $\Gamma, \Gamma' \in \mathrm{Min}(w)$, $\alpha, \beta \in \Phi(w)$ with $\alpha \leq_{\Gamma} \beta$ and $\beta \leq_{\Gamma'} \alpha$. In this case we have two commutation relations, namely 
\begin{align*}
	&[u_{\alpha}, u_{\beta}] = \prod_{\gamma \in M_{\alpha, \beta}^{\Gamma}} u_{\gamma} &&\text{and} &&[u_{\beta}, u_{\alpha}] = \prod_{\gamma \in M_{\beta, \alpha}^{\Gamma'}} u_{\gamma}.
\end{align*}
From now on we will implicitly assume that each product $\prod_{\gamma \in M_{\alpha, \beta}^{\Gamma}} u_{\gamma}$ is ordered via the order $\leq_{\Gamma}$.

\begin{definition}
	A \emph{commutator blueprint of type $(4, 4, 4)$} is a family $\mathcal{M} = \left(M_{\alpha, \beta}^{\Gamma} \right)_{(\Gamma, \alpha, \beta) \in \mathcal{I}}$ of subsets $M_{\alpha, \beta}^{\Gamma} \subseteq (\alpha, \beta)$ ordered via $\leq_{\Gamma}$ satisfying the following axioms:
	\begin{enumerate}[label=(CB\arabic*)]
		\item Let $\Gamma = (c_0, \ldots, c_k) \in \mathrm{Min}$ and let $\Gamma' = (c_0, \ldots, c_m)$ for some $1 \leq m \leq k$. Then $M_{\alpha, \beta}^{\Gamma'} = M_{\alpha, \beta}^{\Gamma}$ holds for all $\alpha, \beta \in \Phi(\Gamma')$ with $\alpha \leq_{\Gamma'} \beta$.
		
		\item Let $s\neq t \in S$, let $\Gamma \in \mathrm{Min}(r_{\{s, t\}})$, let $(\alpha_1, \ldots, \alpha_4)$ be the sequence of roots crossed by $\Gamma$ and let $1 \leq i < j \leq 4$. Then we have
		\[ M_{\alpha_i, \alpha_j}^{\Gamma} = \begin{cases}
			(\alpha_i, \alpha_j) & \{ \alpha_i, \alpha_j \} = \{ \alpha_s, \alpha_t \} \\
			\emptyset & \{ \alpha_i, \alpha_j \} \neq \{ \alpha_s, \alpha_t \}
		\end{cases} = \begin{cases}
		\{ \alpha_2, \alpha_3 \} & (i, j) = (1, 4), \\
		\emptyset & \text{else}.
		\end{cases} \]
		
		\item For each $w\in W$ we have $\lvert U_w \rvert = 2^{\ell(w)}$, where $U_w$ is defined as above.
	\end{enumerate}
\end{definition}

\begin{remark}\label{Remark: Product mapping is a bijection}
	Let $\Gamma = (c_0, \ldots, c_k) \in \mathrm{Min}(w)$ and let $(\alpha_1, \ldots, \alpha_k)$ be the sequence of roots crossed by $\Gamma$. Note that it is a direct consequence of (CB$3$) that the product map $U_{\alpha_1} \times \cdots \times U_{\alpha_k} \to U_w, (u_1, \ldots, u_k) \mapsto u_1 \cdots u_k$ is a bijection, where $\mathbb{Z}_2 \cong U_{\alpha_i} = \langle u_{\alpha_i} \rangle \leq U_w$.
\end{remark}

\begin{example}\label{expintrgrable}
	Let $\mathcal{D} = (G, (U_{\alpha})_{\alpha \in \Phi})$ be an RGD-system of type $(4,4,4)$ over $\FF_2$, let $\Gamma = (c_0, \ldots, c_k) \in \mathrm{Min}$ and let $(\alpha_1, \ldots, \alpha_k)$ be the sequence of roots crossed by $\Gamma$. Then we have $\Phi(\Gamma) = \{ \alpha_1 \leq_{\Gamma} \cdots \leq_{\Gamma} \alpha_k \}$. By \cite[Corollary $8.34(1)$]{AB08} there exists for all $1 \leq m < i < n \leq k$ a unique $\epsilon_{i, m, n} \in \{ 0, 1 \}$ such that $[u_{\alpha_m}, u_{\alpha_n}] = \prod_{i=m+1}^{n-1} u_{\alpha_i}^{\epsilon_{i, m, n}}$, and $\epsilon_{i, m, n} = 1$ implies $\alpha_i \in (\alpha_m, \alpha_n)$. We define $ M(\mathcal{D})_{\alpha_m, \alpha_n}^{\Gamma} := \{ \alpha_i \in \Phi(\Gamma) \mid \epsilon_{i, m, n} = 1 \} \subseteq (\alpha_m, \alpha_n)$ and $\mathcal{M}_{\mathcal{D}} := \left( M(\mathcal{D})_{\alpha, \beta}^{\Gamma} \right)_{(\Gamma, \alpha, \beta) \in \mathcal{I}}$. Then $\mathcal{M}_{\mathcal{D}}$ is a commutator blueprint of type $(4, 4, 4)$ (cf.\ \cite[Example $3.4$]{BiRGD}).
\end{example}

\begin{definition}\label{Definition: U_+}
	Let $\mathcal{M} = \left(M_{\alpha, \beta}^{\Gamma} \right)_{(\Gamma, \alpha, \beta) \in \mathcal{I}}$ be a commutator blueprint of type $(4, 4, 4)$. Using \cite[Lemma $3.6$]{BiRGD} and the axiom (CB$1$), the canonical mapping $u_{\alpha} \mapsto u_{\alpha}$ induces a monomorphism from $U_w$ to $U_{ws}$ for all $w\in W$, $s\in S$ with $\ell(ws) = \ell(w) +1$. We let
	\[ U_+ := \lim U_w \]
	denote the direct limit of the groups $(U_w)_{w\in W}$ with natural inclusions $U_w \to U_{ws}$ if $\ell(ws) = \ell(w) +1$.
\end{definition}

\begin{definition}\label{Definition: Properties of Commutator blueprints}
	Let $\mathcal{M} = \left(M_{\alpha, \beta}^{\Gamma} \right)_{(\Gamma, \alpha, \beta) \in \mathcal{I}}$ be a commutator blueprint of type $(4, 4, 4)$.
	\begin{enumerate}[label=(\alph*)]
		\item $\mathcal{M}$ is called \textit{faithful}, if the canonical homomorphisms $U_w \to U_+$ are injective.
		
		\item $\mathcal{M}$ is called \emph{Weyl-invariant} if for all $w\in W$, $s\in S$, $\Gamma \in \mathrm{Min}_s(w)$ and $\alpha, \beta \in \Phi(\Gamma) \backslash \{ \alpha_s \}$ with $\alpha \leq_{\Gamma} \beta$ we have $M_{s\alpha, s\beta}^{s\Gamma} = sM_{\alpha, \beta}^{\Gamma} := \{ s\gamma \mid \gamma \in M_{\alpha, \beta}^{\Gamma} \}$.
		
		\item $\mathcal{M}$ is called \textit{locally Weyl-invariant} if for all $w\in W$, $s\in S$, $\Gamma \in \mathrm{Min}_s(w)$ and $\alpha, \beta \in \Phi(\Gamma) \backslash \{ \alpha_s \}$ with $\alpha \leq_{\Gamma} \beta$ we have $M_{s\alpha, s\beta}^{s\Gamma} = sM_{\alpha, \beta}^{\Gamma} := \{ s\gamma \mid \gamma \in M_{\alpha, \beta}^{\Gamma} \}$ if $r_{\alpha} r_{\beta}$ is of finite order.
		
		\item $\mathcal{M}$ is called \textit{integrable} if there exists an RGD-system $\mathcal{D}$ of type $(4, 4, 4)$ over $\FF_2$ such that the two families $\mathcal{M}$ and $\mathcal{M}_{\mathcal{D}}$ coincide pointwise.
	\end{enumerate}
\end{definition}

\section{Locally Weyl-invariant Commutator blueprints of type $(4, 4, 4)$} \label{Section: Locally Weyl-invariant Commutator blueprints}

In this section we let $(W, S)$ be of type $(4, 4, 4)$ and $\mathcal{M} = \left( M_{\alpha, \beta}^{\Gamma} \right)_{(\Gamma, \alpha, \beta) \in \mathcal{I}}$ be a locally Weyl-invariant commutator blueprint of type $(4, 4, 4)$. Moreover, we let $S = \{r, s, t\}$. In this section we will introduce several tree products and prove subgroup and isomorphism theorems about them. These results will be used in Section~\ref{Section: Main result} in order to show that $\mathcal{M}$ is faithful.

\begin{remark}
	We refer the reader to the appendix for many useful pictures.
\end{remark}

For a residue $R$ of $\Sigma(W, S)$ we put $w_R := \proj_R 1_W$. Let $R$ be a residue of type $\{ s, t \}$. Then we have $\ell(w_R s) = \ell(w_R) +1 = \ell(w_R t)$. We define the group $V_{w_R r_{\{ s, t \}}} := \langle U_{w_R s} \cup U_{w_R t} \rangle \leq U_{w_R r_{\{s, t\}}}$. Using (CB$3$) and fact that $\mathcal{M}$ is locally Weyl-invariant, the group $V_{w_R r_{\{s, t\}}}$ is an index $2$ subgroup of $U_{w_R r_{\{ s, t \}}}$ (cf.\ Remark \ref{Remark: Product mapping is a bijection}). For each $i \in \NN$ we let $\mathcal{R}_i$ be the set of all rank $2$ residues $R$ with $\ell(w_R) = i$ (e.g.\ $\mathcal{R}_0 = \{ R_{\{s, t\}}(1_W) \mid s\neq t \in S \}$). We let $\mathcal{T}_{i, 1}$ be the set of all residues $R \in \mathcal{R}_i$ with $\ell(w_R sr) = \ell(w_R) +2 = \ell(w_R tr)$, where $\{ s, t \}$ is the type of $R$. Let $R \in \mathcal{R}_i \backslash \mathcal{T}_{i, 1}$ be of type $\{s, t\}$. Then we have $\ell(w_R) \in \{ \ell(w_R sr), \ell(w_R tr) \}$. By Lemma \ref{Lemma: not both down} we have $\{ \ell(w_R), \ell(w_R) +2 \} = \{ \ell(w_R sr), \ell(w_R tr) \}$. Let $u \neq v \in \{s, t\}$ be such that $\ell(w_R ur) = \ell(w_R)$. Then $T_R := R_{\{v, r\}}(w_R u) \neq R$ and $T_R \in \mathcal{R}_i$ by Lemma \ref{Lemma: not both down}. In particular, $T_R \in \mathcal{R}_i \backslash \mathcal{T}_{i, 1}$ and we have $T_{(T_R)} = R$. We define $\mathcal{T}_{i, 2} := \{ \{ R, T_R \} \mid R \in \mathcal{R}_i \backslash \mathcal{T}_{i, 1} \}$. Moreover, we let $\mathcal{T}_i := \mathcal{T}_{i, 1} \cup \mathcal{T}_{i, 2}$.

We have already mentioned that we will introduce several trees of groups, more precisely, sequences of groups. The groups in the sequences of groups will always be generated by elements $u_{\alpha}$ for suitable $\alpha \in \Phi_+$. Let $A$ and $B$ vertex groups such that the corresponding vertices are joint by an edge, and let $C$ be the edge group. Let $\Phi_A, \Phi_B \subseteq \Phi_+$ be such that $A = \langle u_{\alpha} \mid \alpha \in \Phi_A \rangle$ and $B = \langle u_{\alpha} \mid \alpha \in \Phi_B \rangle$. If we do not specify $C$, then we will implicitly assume that $C = \langle u_{\alpha} \mid \alpha \in \Phi_A \cap \Phi_B \rangle$. If $C$ is as in this case, then it will always be clear that we have canonical homomorphisms $C \to A$ and $C \to B$ which are injective, and we define $A \hat{\star} B := A \star_C B$.

The following lemma will be crucial and mainly used in the proofs of the rest of this section.

\begin{lemma}\label{Lemma: Key lemma}
	Suppose $w\in W$ with $\ell(ws) = \ell(w) +1 = \ell(wt)$.
	\begin{enumerate}[label=(\alph*)]
		\item $V_{w r_{\{s, t\}}} \cap U_{w st} = U_{ws}$ and $U_{ws} \cap U_{wt} = U_w$ hold in $U_{w r_{\{s, t\}}}$.
		
		\item $U_{w r_{\{s, t\}}} \cap U_{w tstrs} = U_{w tst}$ holds in $U_{w tst r_{\{r, s\}}}$.
		
		\item $U_{w r_{\{s, t\}}} \cap U_{w str} = U_{w st}$ and $V_{w r_{\{s, t\}}} \cap U_{w str} = U_{ws}$ hold in $U_{w r_{\{s, t\}}} \hat{\star} V_{wst r_{\{r, s\}}}$.
		
		\item $V_{w sts r_{\{r, t\}}} \cap U_{w tstrs} = U_{w tst}$ holds in $U_{w sts r_{\{r, t\}}} \hat{\star} V_{w r_{\{s, t\}} r r_{\{s, t\}}} \hat{\star} U_{w tst r_{\{r, s\}}}$.
	\end{enumerate}
\end{lemma}
\begin{proof}
	Part $(a)$ and $(b)$ follow essentially from Remark \ref{Remark: Product mapping is a bijection} and the fact that $V_{w r_{\{s, t\}}}$ has index two in $U_{wr_{\{s, t\}}}$. For part $(c)$ we use Corollary \ref{AcapBisC}. We deduce $U_{w r_{\{s, t\}}} \cap U_{w str} \subseteq U_{w sts}$ and hence 
	\[ U_{w r_{\{s, t\}}} \cap U_{w str} = U_{w r_{\{s, t\}}} \cap U_{w str} \cap U_{w sts} = U_{w r_{\{s, t\}}} \cap U_{w st} = U_{w st}. \]
	Using the same arguments and part $(a)$, we infer $V_{w r_{\{s, t\}}} \cap U_{w str} = V_{w r_{\{s, t\}}} \cap U_{w st} = U_{ws}$. For part $(d)$ we first observe that by Corollary \ref{AcapBisC} and Proposition~\ref{treeproducts} we have $V_{w sts r_{\{r, t\}}} \cap U_{w tstrs} \subseteq U_{w r_{\{s, t\}} rs}$ and by part $(c)$ we have
	\[ V_{w sts r_{\{r, t\}}} \cap U_{w tstrs} = V_{w sts r_{\{r, t\}}} \cap U_{w tstrs} \cap U_{w r_{\{s, t\}} rs} = U_{w tstrs} \cap U_{w r_{\{s, t\}}}. \]
	Now the claim follows from part $(b)$.
\end{proof}

\subsection*{The groups $\mathbf{V_R}$ and $\mathbf{O_R}$}

For a residue $R \in \mathcal{T}_{i, 1}$ of type $\{s, t\}$ we define the group $V_R$ to be the tree product of the sequence of groups with vertex groups
\allowdisplaybreaks
\begin{align*}
	U_{w_R sr}, V_{w_R r_{\{s, t\}}}, U_{w_R tr}.
\end{align*}
Furthermore, we define the group $O_R$ to be the tree product of the sequence of groups with vertex groups
\allowdisplaybreaks
\begin{align*}
	V_{w_R sr_{\{r, t\}}}, U_{w_R r_{\{s, t\}}}, V_{w_R tr_{\{r, s\}}}.
\end{align*}

\begin{remark}\label{Remark: tree product generated by root group elements}
	For $V_R$ we consider $\alpha := w_R s\alpha_r$. Using Lemma \ref{mingallinrep} we see that $-w_R \alpha_t \subseteq \alpha$. As $w_R t \in (-w_R \alpha_t)$, we deduce $w_R tr, w_R r_{\{s, t\}} \in \alpha$ and hence $u_{\alpha}$ is neither a generator of $V_{w_R r_{\{s, t\}}}$ nor of $U_{w_R tr}$. Now we consider $w_R \alpha_s$. As $-w_R t\alpha_r \subseteq w_R \alpha_s$ by Lemma \ref{mingallinrep} we deduce that $u_{w_R \alpha_s}$ is not a generator of $U_{w_R tr}$. Using similar methods we infer that $V_R$ is generated by $\{ u_{\alpha} \mid \exists v\in \{ w_R sr, w_R tr  \}: v\notin \alpha \}$. A similar result holds for $O_R$.
\end{remark}

\begin{lemma}[{\cite[Lemma~4.13]{BiRGDandTreeproducts}}]\label{VRtoORinjective}
	Let $R \in \mathcal{T}_{i, 1}$. Then the canonical homomorphism $V_R \to O_R$ is injective.
\end{lemma}

\subsection*{The groups $\mathbf{H_R}, \mathbf{G_R}$ and $\mathbf{J_{R, t}}$}

Let $R \in \mathcal{T}_{i, 1}$ be of type $\{s, t\}$. We define the group $H_R$ to be the tree product of the sequence of groups with vertex groups
\allowdisplaybreaks
\begin{align*}
	U_{w_Rsr_{\{r, t\}}}, V_{w_R str_{\{r, s\}}}, U_{w_R r_{\{s, t\}}}, V_{w_R ts r_{\{r, t\}}}, U_{w_R t r_{\{r, s\}}}.
\end{align*}
We define the group $J_{R, t}$ to be the tree product of the sequence of groups with vertex groups
\allowdisplaybreaks
\begin{align*}
	U_{w_R s r_{\{r, t\}}}, V_{w_R st r_{\{r, s\}}}, V_{w_R tst r_{\{r, s\}}}, U_{w_R ts r_{\{r, t\}}}, V_{w_R tsr r_{\{s, t\}}}, U_{w_R t r_{\{r, s\}}}.
\end{align*}
Furthermore, we define the group $G_R$ to be the tree product of the sequence of groups with vertex groups
\allowdisplaybreaks
\begin{align*}
	U_{w_Rsr_{\{r, t\}}}, V_{w_R strr_{\{s, t\}}}, U_{w_Rst r_{\{r, s\}}}, V_{w_Rstsr r_{\{s, t\}}}, \\
	U_{w_Rstsr_{\{r, t\}}}, V_{w_R r_{\{s, t\}}rr_{\{s, t\}}}, U_{w_Rtstr_{\{r, s\}}}, \\
	V_{w_Rtstrr_{\{s, t\}}}, U_{w_R ts r_{\{r, t\}}}, V_{w_R tsr r_{\{s, t\}}}, U_{w_R t r_{\{r, s\}}}.
\end{align*}
It follows similarly as in Remark \ref{Remark: tree product generated by root group elements} that $H_R, J_{R, t}$ and $G_R$ are generated by suitable $u_{\alpha}$.

\begin{lemma}\label{HRtoGRinjective}
	Let $R \in \mathcal{T}_{i, 1}$ be of type $\{s, t\}$. Then the canonical homomorphisms $H_R \to J_{R, t}$ and $J_{R, t} \to G_R$ are injective. In particular, the canonical homomorphism $H_R \to G_R$ is injective.
\end{lemma}
\begin{proof}
	We first show that $H_R \to J_{R, t}$ is injective. Using Proposition \ref{treeproducts} the group $J_{R, t}$ is isomorphic to the tree product of the sequence of groups with vertex groups
	\allowdisplaybreaks
	\begin{align*}
		U_{w_R s r_{\{r, t\}}}, V_{w_R st r_{\{r, s\}}}, V_{w_R tst r_{\{r, s\}}}, U_{w_R ts r_{\{r, t\}}}, V_{w_R tsr r_{\{s, t\}}} \hat{\star} U_{w_R t r_{\{r, s\}}}.
	\end{align*}
	We will apply Proposition \ref{treeofgroupsinjective}. Therefore we first see that each vertex group of $H_R$ is contained in the corresponding vertex group of the previous tree product, e.g.\ $U_{w_R tr_{\{r, s\}}} \leq V_{w_R tsrr_{\{s, t\}}} \hat{\star} U_{w_R tr_{\{r, s\}}}$. Next we have to show that the preimages of the boundary monomorphisms are equal and coincide with the edge groups of $H_R$. But this follows from Lemma \ref{Lemma: Key lemma} (similar as in the proof of Lemma \ref{VRtoORinjective}). Now Proposition \ref{treeofgroupsinjective} yields that $H_R \to J_{R, t}$ is injective.
	
	Now we will show that $J_{R, t} \to G_R$ is injective. Using Proposition \ref{treeproducts} the group $G_R$ is isomorphic to the tree product of the following sequence of groups with vertex groups
	\allowdisplaybreaks
	\begin{align*}
		U_{w_R s r_{\{r, t\}}} \hat{\star} V_{w_R str r_{\{s, t\}}}, U_{w_R st r_{\{r, s\}}} \hat{\star} V_{w_R stsr r_{\{s, t\}}}, U_{w_R sts r_{\{r, t\}}} \hat{\star} V_{w_R r_{\{s, t\}} r r_{\{s, t\}}} \hat{\star} U_{w_R tst r_{\{r, s\}}}, \\
		V_{w_R tstr r_{\{s, t\}}} \hat{\star} U_{w_R ts r_{\{r, t\}}}, V_{w_R tsr r_{\{s, t\}}}, U_{w_R t r_{\{r, s\}}}.
	\end{align*}
	One easily sees that each vertex group of $J_{R, t}$ is contained in the corresponding vertex group of the previous tree product. Again we deduce from Lemma \ref{Lemma: Key lemma} and Proposition \ref{treeofgroupsinjective} that $J_{R, t} \to G_R$ is injective.
\end{proof}

\begin{lemma}\label{Lemma: J_R,t cong H_R star V_T O_T}
	Let $R \in \mathcal{T}_{i, 1}$ be a residue of type $\{s, t\}$ and let $T = R_{\{r, t\}}(w_R ts)$. Then $T \in \mathcal{T}_{i+2, 1}$, the canonical homomorphism $V_T \to H_R$ is injective and we have $J_{R, t} \cong H_R \star_{V_T} O_T$.
\end{lemma}
\begin{proof}
	Note that $T \in \mathcal{T}_{i+2, 1}$. By \cite[Lemma~$4.16$]{BiRGDandTreeproducts} the mapping $V_T \to H_R$ is injective. Using Proposition \ref{treeproducts}, Proposition \ref{treeofgroupsinjective}, Remark \ref{Remark: isomorphism preserves amalgamated product}, Lemma \ref{folding} and Lemma \ref{VRtoORinjective} we obtain the following isomorphisms:
	\allowdisplaybreaks
	\begin{align*}
		J_{R, t} &\cong U_{w_R s r_{\{r, t\}}} \hat{\star} V_{w_R st r_{\{r, s\}}} \star_{U_{w_R sts}} \left( O_T \star_{U_{w_R tsrs}} U_{w_R t r_{\{r, s\}}} \right) \\
		&\cong U_{w_R s r_{\{r, t\}}} \hat{\star} V_{w_R st r_{\{r, s\}}} \star_{U_{w_R sts}} \left( \left( U_{w_R t r_{\{r, s\}}} \star_{U_{w_R tsrs}} V_T \right) \star_{V_T} O_T \right) \\
		&\cong U_{w_R s r_{\{r, t\}}} \hat{\star} V_{w_R st r_{\{r, s\}}} \star_{U_{w_R sts}} \left( V_T \star_{U_{w_R tsrs}} U_{w_R t r_{\{r, s\}}} \right) \star_{V_T} O_T \\
		&\cong H_R \star_{V_T} O_T. \qedhere
	\end{align*}
\end{proof}

\subsection*{The groups $\mathbf{E_{R, s}}$ and $\mathbf{U_{R, s}}$}

Let $R \in \mathcal{T}_{i, 1}$ be of type $\{s, t\}$ such that $\ell(w_R rs) = \ell(w_R) -2$. We put $R' = R_{\{r, s\}}(w_R)$ and $w' = w_{R'}$. We define the group $E_{R, s}$ to be the tree product of the sequence of groups with vertex groups
\allowdisplaybreaks
\begin{align*}
	U_{w'rs r_{\{r, t\}}}, V_{w'rsrt r_{\{r, s\}}}, U_{w'rsrr_{\{s, t\}}}, V_{w_R srtr_{\{r, s\}}}, U_{w_Rsr_{\{r, t\}}}, \\
	V_{w_R str_{\{r, s\}}}, U_{w_R r_{\{s, t\}}}, V_{w_R ts r_{\{r, t\}}}, U_{w_R t r_{\{r, s\}}}.
\end{align*}
Furthermore, we define the group $U_{R, s}$ to be the tree product of the sequence of groups with vertex groups
\allowdisplaybreaks
\begin{align*}
	U_{w'rs r_{\{r, t\}}}, V_{w'rsrt r_{\{r, s\}}}, U_{w'rsrr_{\{s, t\}}}, V_{w_R srtr_{\{r, s\}}}, U_{w_Rsr_{\{r, t\}}}, V_{w_R strr_{\{s, t\}}}, \\
	U_{w_Rst r_{\{r, s\}}}, V_{w_Rstsr r_{\{s, t\}}}, U_{w_Rstsr_{\{r, t\}}}, V_{w_R r_{\{s, t\}}rr_{\{s, t\}}}, U_{w_Rtstr_{\{r, s\}}},	\\
	V_{w_Rtstrr_{\{s, t\}}}, U_{w_R ts r_{\{r, t\}}}, V_{w_R tsr r_{\{s, t\}}}, U_{w_R t r_{\{r, s\}}}.
\end{align*}
It follows similarly as in Remark \ref{Remark: tree product generated by root group elements} that $E_{R, s}$ and $U_{R, s}$ are generated by suitable $u_{\alpha}$.

\begin{lemma}\label{ERstoURsinjective}
	Let $R \in \mathcal{T}_{i, 1}$ be of type $\{s, t\}$ such that $\ell(w_R rs) = \ell(w_R) -2$. Then the canonical homomorphisms $H_R \to E_{R, s}$ and $E_{R, s} \to U_{R, s}$ are injective and we have $E_{R, s} \star_{H_R} G_R \cong U_{R, s}$.
\end{lemma}
\begin{proof}
	The first four vertex groups of the underlying sequences of groups of $E_{R, s}$ and $U_{R, s}$ coincide. Thus we denote the tree product of these first four vertex groups by $F_4$. Using Proposition \ref{treeproducts} we deduce $E_{R, s} \cong F_4 \star_{U_{w_R srtr}} H_R$ and $U_{R, s} \cong F_4 \star_{U_{w_R srtr}} G_R$. In particular, $H_R \to E_{R, s}$ is injective. Using Lemma \ref{HRtoGRinjective}, Proposition \ref{treeproducts}, Remark \ref{Remark: isomorphism preserves amalgamated product} and Lemma \ref{folding} we infer
	\[ U_{R, s} \cong F_4 \star_{U_{w_R srtr}} G_R \cong F_4 \star_{U_{w_R srtr}} H_R \star_{H_R} G_R \cong E_{R, s} \star_{H_R} G_R. \]
	Proposition \ref{treeproducts} yields that $E_{R, s} \to U_{R, s}$ is injective and the claim follows.
\end{proof}

\subsection*{The group $X_R$}

Let $R \in \mathcal{T}_{i, 1}$ be a residue of type $\{s, t\}$ such that $\ell(w_R rs) = \ell(w_R) -2$ and $\ell(w_R rt) = \ell(w_R)$. Let $R' = R_{\{r, s\}}(w_R)$ and let $w' = w_{R'}$. We define the group $X_R$ to be the tree product of the sequence of groups with vertex groups
\allowdisplaybreaks
\begin{align*}
	U_{w'rs r_{\{r, t\}}}, V_{w'rsrt r_{\{r, s\}}}, U_{w'rsrr_{\{s, t\}}}, V_{w_R srtr_{\{r, s\}}}, U_{w_Rsr_{\{r, t\}}}, \\
	V_{w_R str_{\{r, s\}}}, U_{w_R r_{\{s, t\}}}, V_{w_R t r_{\{r, s\}}}, U_{w' s r_{\{r, t\}}}.
\end{align*}
It follows similarly as in Remark \ref{Remark: tree product generated by root group elements} that $X_R$ is generated by suitable $u_{\alpha}$.

\begin{remark}
	Let $R \in \mathcal{T}_{i, 1}$ be a residue of type $\{s, t\}$ such that $\ell(w_R rs) = \ell(w_R) -2$ and $\ell(w_R rt) = \ell(w_R)$ and let $T := R_{\{r, s\}}(w_Rt)$. Note that $T \in \mathcal{T}_{i+1, 1}$. In the next lemma we consider $X_R \star_{V_T} O_T$. Similar as in Remark \ref{Remark: tree product generated by root group elements} we have to show that if $x_{\alpha}$ is a generator of $X_R$ and $y_{\alpha}$ is a generator of $O_T$, then $x_{\alpha} = y_{\alpha}$ holds in $X_R \star_{V_T} O_T$. It suffices to consider $w_Rtr \alpha_s$ and $w_R ts \alpha_r$. As $-w_R \alpha_s \subseteq w_Rtr \alpha_s, w_R ts \alpha_r$ by Lemma~\ref{mingallinrep}, we deduce that $x_{\alpha}$ is not a generator of $X_R$ for $\alpha \in \{ w_Rtr \alpha_s, w_R ts \alpha_r \}$.
\end{remark}

\begin{lemma}\label{Lemma: X_R a}
	Let $R \in \mathcal{T}_{i, 1}$ be a residue of type $\{s, t\}$ such that $\ell(w_R rs) = \ell(w_R) -2$ and $\ell(w_R rt) = \ell(w_R)$ and let $T := R_{\{r, s\}}(w_Rt)$. Then the canonical homomorphisms $V_T \to X_R$ and $E_{R, s} \to X_R \star_{V_T} O_T$ are injective.
\end{lemma}
\begin{proof}
	The first part follows from Proposition \ref{treeproducts} and Proposition \ref{treeofgroupsinjective}. Let $F_6$ be the tree product of the first six vertex groups of the underlying sequence of groups of $X_R$. Using Proposition \ref{treeproducts}, Remark \ref{Remark: isomorphism preserves amalgamated product}, Lemma \ref{folding} and Lemma \ref{VRtoORinjective} we obtain the following isomorphisms (where $R' = R_{\{r, s\}}(w_R)$ and $w' = w_{R'}$):
	\allowdisplaybreaks
	\begin{align*}
		X_R \star_{V_T} O_T &\cong \left( F_6 \star_{U_{w_R sts}} U_{w_R r_{\{s, t\}}} \hat{\star} V_{w_R t r_{\{r, s\}}} \hat{\star} U_{w' s r_{\{r, t\}}} \right) \star_{V_T} O_T \\
		&\cong \left( F_6 \star_{U_{w_R sts}} U_{w_R r_{\{s, t\}}} \star_{U_{w_R tst}} U_{w_R tst} \hat{\star} V_{w_R t r_{\{r, s\}}} \hat{\star} U_{w' s r_{\{r, t\}}} \right) \star_{V_T} O_T \\
		&\cong \left( F_6 \star_{U_{w_R sts}} U_{w_R r_{\{s, t\}}} \star_{U_{w_R tst}} V_T \right) \star_{V_T} O_T \\
		&\cong F_6 \star_{U_{w_R sts}} U_{w_R r_{\{s, t\}}} \star_{U_{w_R tst}} V_T \star_{V_T} O_T \\
		&\cong F_6 \star_{U_{w_R sts}} U_{w_R r_{\{s, t\}}} \star_{U_{w_R tst}} O_T \\
		&\cong E_{R, s} \star_{U_{w_R trs}} V_{w_Rtrr_{\{s, t\}} }. \qedhere
	\end{align*}
\end{proof}

\begin{lemma}\label{Lemma: X_R 1}
	Let $R \in \mathcal{T}_{i, 1}$ be a residue of type $\{s, t\}$ such that $\ell(w_R rs) = \ell(w_R) -2$ and $\ell(w_R rt) = \ell(w_R)$. Let $Z := R_{\{r, s\}}(w_R)$ be and suppose that $Z \in \mathcal{T}_{i-2, 1}$. Then $X_R \to G_Z$ is injective.
\end{lemma}
\begin{proof}
	As the last nine vertex groups of the underlying sequence of groups of $G_Z$ coincide with the vertex groups of the underlying sequence of groups of $X_R$, the claim follows from Proposition \ref{treeproducts}.
\end{proof}

\subsection*{The groups $\mathbf{H_{\{R, R'\}}}, \mathbf{G_{\{R, R'\}}}$ and $\mathbf{J_{(R, R')}}$}

Let $\{R, R'\} \in \mathcal{T}_{i, 2}$. Let $w = w_R, w' = w_{R'}$ and let $\{r, s\}$ (resp.\ $\{r, t\}$) be the type of $R$ (resp.\ $R'$). Let $T = R_{\{r, t\}}(w)$ and $T' = R_{\{r, s\}}(w')$. We define the group $H_{\{R, R'\}}$ to be the tree product of the sequence of groups with vertex groups
\allowdisplaybreaks
\begin{align*}
	U_{w_T rtr r_{\{s, t\}}}, V_{w_T r_{\{r, t\}}sr_{\{r, t\}}}, U_{w_T trt r_{\{r, s\}}}, V_{w_T trtsr_{\{r, t\}}}, U_{w_T trr_{\{s, t\}}}, \\
	V_{wrsr_{\{r, t\}}}, U_{wr_{\{r, s\}}}, V_{wsrr_{\{s, t\}}}, U_{w'r_{\{r, t\}}}, V_{w'rtr_{\{r, s\}}}, \\
	U_{w_{T'}srr_{\{s, t\}}}, V_{w_{T'} srstr_{\{r, s\}}}, U_{w_{T'} srsr_{\{r, t\}}}, V_{w_{T'} r_{\{r, s\}}tr_{\{r, s\}}}, U_{w_{T'} rsrr_{\{s, t\}}}.
\end{align*}
We define the group $J_{(R, R')}$ to be the tree product of the sequence of groups with vertex groups
\allowdisplaybreaks
\begin{align*}
	U_{w_T rtr r_{\{s, t\}}}, V_{w_T r_{\{r, t\}}sr_{\{r, t\}}}, U_{w_T trt r_{\{r, s\}}}, V_{w_T trtsr_{\{r, t\}}}, \\
	U_{w_T trr_{\{s, t\}}}, V_{w rst r_{\{r, s\}}}, U_{w rs r_{\{r, t\}}}, V_{w rsr r_{\{s, t\}}}, V_{w srr_{\{s, t\}}}, U_{w' r_{\{r, t\}}}, V_{w' rtr_{\{r, s\}}}, \\
	U_{w_{T'}srr_{\{s, t\}}}, V_{w_{T'} srstr_{\{r, s\}}}, U_{w_{T'} srsr_{\{r, t\}}}, V_{w_{T'} r_{\{r, s\}}tr_{\{r, s\}}}, U_{w_{T'} rsrr_{\{s, t\}}}.
\end{align*}
Furthermore, we define the group $G_{\{R, R'\}}$ to be the tree product of the sequence of groups with vertex groups
\allowdisplaybreaks
\begin{align*}
	U_{w_T rtr r_{\{s, t\}}}, V_{w_T r_{\{r, t\}}sr_{\{r, t\}}}, U_{w_T trt r_{\{r, s\}}}, V_{w_T trtsr_{\{r, t\}}}, \\
	U_{w_T trr_{\{s, t\}}}, V_{wrstr_{\{r, s\}}}, U_{wrsr_{\{r, t\}}}, V_{wrsrtr_{\{r, s\}}}, U_{wrsrr_{\{s, t\}}}, V_{wr_{\{r, s\}}t r_{\{r, s\}}}, U_{wsrsr_{\{r, t\}}}, \\
	V_{wsrstr_{\{r, s\}}}, U_{wsrr_{\{s, t\}}}, V_{w'trtsr_{\{r, t\}}}, \\
	U_{w'trtr_{\{r, s\}}}, V_{w' r_{\{r, t\}}sr_{\{r, t\}}}, U_{w'rtrr_{\{s, t\}}}, V_{w'rtrsr_{\{r, t\}}}, U_{w'rtr_{\{r, s\}}}, V_{w'rtsr_{\{r, t\}}}, U_{w_{T'}srr_{\{s, t\}}}, \\
	V_{w_{T'} srstr_{\{r, s\}}}, U_{w_{T'} srsr_{\{r, t\}}}, V_{w_{T'} r_{\{r, s\}}tr_{\{r, s\}}}, U_{w_{T'} rsrr_{\{s, t\}}}.
\end{align*}
It follows similarly as in Remark \ref{Remark: tree product generated by root group elements} that $H_{\{R, R'\}}, G_{\{R, R'\}}$ and $J_{(R, R')}$ are generated by suitable $u_{\alpha}$.

\begin{lemma}\label{HRR'toGRR'injective}
	Let $\{R, R'\} \in \mathcal{T}_{i, 2}$, let $\{r, s\}$ be the type of $R$ and let $\{r, t\}$ be the type of $R'$. Then the canonical homomorphisms $H_{\{R, R'\}} \to J_{(R, R')}$ and $J_{(R, R')} \to G_{\{R, R'\}}$ are injective. In particular, the canonical homomorphism $H_{\{R, R'\}} \to G_{\{R, R'\}}$ is injective.
\end{lemma}
\begin{proof}
	We first show that the homomorphism $H_{\{R, R'\}} \to J_{(R, R')}$ is injective. Using Proposition \ref{treeproducts} the group $J_{(R, R')}$ is isomorphic to the tree product of the following sequence of groups with vertex groups
	\allowdisplaybreaks
	\begin{align*}
		U_{w_T rtr r_{\{s, t\}}}, V_{w_T r_{\{r, t\}}sr_{\{r, t\}}}, U_{w_T trt r_{\{r, s\}}}, V_{w_T trtsr_{\{r, t\}}}, \\
		U_{w_T trr_{\{s, t\}}} \hat{\star}	V_{w rst r_{\{r, s\}}}, U_{w rs r_{\{r, t\}}}, V_{w rsr r_{\{s, t\}}}, V_{w srr_{\{s, t\}}}, U_{w' r_{\{r, t\}}}, V_{w' rtr_{\{r, s\}}}, \\
		U_{w_{T'}srr_{\{s, t\}}}, V_{w_{T'} srstr_{\{r, s\}}}, U_{w_{T'} srsr_{\{r, t\}}}, V_{w_{T'} r_{\{r, s\}}tr_{\{r, s\}}}, U_{w_{T'} rsrr_{\{s, t\}}}.
	\end{align*}
	One easily sees that each vertex group of $H_{\{R, R'\}}$ is contained in the corresponding vertex group of the previous tree product. Again we deduce from Lemma \ref{Lemma: Key lemma} and Proposition \ref{treeofgroupsinjective} that $H_{\{R, R'\}} \to J_{(R, R')}$ is injective.
	
	Now we show that $J_{(R, R')} \to G_{\{R, R'\}}$ is injective. Using Proposition \ref{treeproducts} the group $G_{\{R, R'\}}$ is isomorphic to the tree product of the following sequence of groups with vertex groups
	\allowdisplaybreaks
	\begin{align*}
		U_{w_T rtr r_{\{s, t\}}}, V_{w_T r_{\{r, t\}}sr_{\{r, t\}}}, U_{w_T trt r_{\{r, s\}}}, V_{w_T trtsr_{\{r, t\}}}, U_{w_T trr_{\{s, t\}}}, V_{w rstr_{\{r, s\}}}, \\
		U_{w rsr_{\{r, t\}}} \hat{\star} V_{w rsrtr_{\{r, s\}}}, U_{w rsrr_{\{s, t\}}} \hat{\star} V_{w r_{\{r, s\}}t r_{\{r, s\}}} \hat{\star} U_{w srsr_{\{r, t\}}}, \\
		V_{w srstr_{\{r, s\}}} \hat{\star} U_{w srr_{\{s, t\}}} \hat{\star} V_{w' trtsr_{\{r, t\}}}, \\
		U_{w' trtr_{\{r, s\}}} \hat{\star} V_{w' r_{\{r, t\}}sr_{\{r, t\}}} \hat{\star} U_{w' rtrr_{\{s, t\}}}, V_{w' rtrsr_{\{r, t\}}} \hat{\star} U_{w' rtr_{\{r, s\}}}, V_{w' rtsr_{\{r, t\}}} \hat{\star} U_{w_{T'}srr_{\{s, t\}}}, \\
		V_{w_{T'} srstr_{\{r, s\}}}, U_{w_{T'} srsr_{\{r, t\}}}, V_{w_{T'} r_{\{r, s\}}tr_{\{r, s\}}}, U_{w_{T'} rsrr_{\{s, t\}}}.
	\end{align*}
	One easily sees that each vertex group of $J_{(R, R')}$ is contained in the corresponding vertex group of the previous tree product. Again we deduce from Lemma \ref{Lemma: Key lemma} and Proposition \ref{treeofgroupsinjective} that $J_{(R, R')} \to G_{\{R, R'\}}$ is injective.
\end{proof}

\begin{lemma}\label{Lemma: ERs to GTT' injective}
	Let $R \in \mathcal{T}_{i, 1}$ be of type $\{s, t\}$ such that $\ell(w_Rrs) = \ell(w_R) -2 = \ell(w_Rrt)$. Let $T = R_{\{r, s\}}(w_R)$ and $T' = R_{\{r, t\}}(w_R)$. Then $\{ T, T' \} \in \mathcal{T}_{i-2, 2}$ and the canonical homomorphism $E_{R, s} \to G_{\{T, T'\}}$ is injective.
\end{lemma}
\begin{proof}
	Since $R \in \mathcal{T}_{i, 1}$, we have $\{T, T'\} \in \mathcal{T}_{i-2, 2}$. The second assertion follows directly from Proposition \ref{treeproducts}, as the vertex groups of $E_{R, s}$ and the vertex groups $7-15$ of $G_{\{T, T'\}}$ coincide.
\end{proof}

\begin{lemma}\label{Lemma: J_R,R' cong H_R,R' star V_Z O_Z}
	Let $\{R, R'\} \in \mathcal{T}_{i, 2}$, let $\{r, s\}$ be the type of $R$, let $\{r, t\}$ be the type of $R'$, and let $Z = R_{\{r, t\}}(w_R rs)$. Then $Z \in \mathcal{T}_{i+2, 1}$, the canonical homomorphism $V_Z \to H_{\{R, R'\}}$ is injective and we have $J_{(R, R')} \cong H_{\{R, R'\}} \star_{V_Z} O_Z$.
\end{lemma}
\begin{proof}
	Note that $Z \in \mathcal{T}_{i+2, 1}$. By Proposition \ref{treeproducts}, $U_{w_R rr_{\{s, t\}}} \hat{\star} V_{w_R rs r_{\{r, t\}}} \hat{\star} U_{w_R rsrs} \to H_{\{R, R'\}}$ is injective. Using Proposition \ref{treeofgroupsinjective}, we deduce that
	\allowdisplaybreaks
	\begin{align*}
		V_Z = U_{w_R rsts} \hat{\star} V_{w_R rs r_{\{r, t\}}} \hat{\star} U_{w_R rsrs} \to U_{w_R rr_{\{s, t\}}} \hat{\star} V_{w_R rs r_{\{r, t\}}} \hat{\star} U_{w_R rsrs}
	\end{align*}
	is injective and hence also the concatenation $V_T \to H_{\{R, R'\}}$. Let $F_i$ be the tree product of the first $i$ vertex groups and let $L_j$ be the tree product of the last $j$ vertex groups of the underlying sequence of groups of $J_{(R, R')}$. Note that by Proposition \ref{treeofgroupsinjective} and Lemma \ref{VRtoORinjective} the homomorphism $F_5 \star_{U_{w_R rsts}} V_Z \to F_5 \star_{U_{w_R rsts}} O_Z$ is injective. We deduce from Proposition \ref{treeproducts} and Lemma \ref{folding} that $F_5 \star_{U_{w_R rsts}} V_Z \star_{U_{w_R srs}} L_8 \cong H_{\{R, R'\}}$. Note also, that $U_{w_R srs} \to V_Z$ is injective. Using Proposition \ref{treeproducts}, Remark \ref{Remark: isomorphism preserves amalgamated product}, Lemma \ref{folding} and Lemma \ref{VRtoORinjective} we obtain the following isomorphisms:
	\allowdisplaybreaks
	\begin{align*}
		J_{(R, R')} &\cong F_5 \star_{U_{w_R rsts}} V_{w_Rrst r_{\{r, s\}}} \hat{\star} U_{w_R rs r_{\{r, t\}}} \hat{\star} V_{w_R rsr r_{\{s, t\}}} \star_{U_R srs} L_8 \\
		&\cong F_5 \star_{U_{w_R rsts}} O_Z \star_{U_{w_R srs}} L_8 \\
		&\cong \left( F_5 \star_{U_{w_R rsts}} O_Z \right) \star_{\left( F_5 \star_{U_{w_R rsts}} V_Z \right)} \left( F_5 \star_{U_{w_R rsts}} V_Z \right) \star_{U_{w_R srs}} L_8 \\
		&\cong \left( F_5 \star_{U_{w_R rsts}} V_Z \star_{V_Z} O_Z \right) \star_{\left( F_5 \star_{U_{w_R rsts}} V_Z \right)} \left( F_5 \star_{U_{w_R rsts}} V_Z \star_{U_{w_R srs}} L_8 \right) \\
		&\cong \left( O_Z \star_{V_Z} ( F_5 \star_{U_{w_R rsts}} V_Z) \right) \star_{\left( F_5 \star_{U_{w_R rsts}} V_Z \right)} H_{\{R, R'\}} \\
		&\cong O_Z \star_{V_Z} H_{\{R, R'\}}. \qedhere
	\end{align*}
\end{proof}

\begin{lemma}\label{Lemma: X_R 2}
	Let $R \in \mathcal{T}_{i, 1}$ be a residue of type $\{s, t\}$ such that $\ell(w_R rs) = \ell(w_R) -2$ and $\ell(w_R rt) = \ell(w_R)$. Let $Z := R_{\{r, s\}}(w_R)$ and suppose that $Z \notin \mathcal{T}_{i-2, 1}$. Let $P_Z \in \mathcal{T}_{i-2, 2}$ be the unique element with $Z \in P_Z$. Then $X_R \to G_{P_Z}$ is injective.
\end{lemma}
\begin{proof}
	As the vertex groups $13 - 21$ of the underlying sequence of groups of $G_{P_Z}$ coincide with the vertex groups of the underlying sequence of groups of $X_R$, the claim follows from Proposition \ref{treeproducts}.
\end{proof}

\subsection*{The groups $\mathbf{C}$ and $\mathbf{C_{(R, R')}}$}

Let $\{ R, R' \} \in \mathcal{T}_{i, 2}$. Let $R$ be of type $\{r, s\}$ and let $R'$ be of type $\{r, t\}$. We let $T = R_{\{r, t\}}(w_R)$ and $T' = R_{\{r, s\}}(w_{R'})$. We define the group $C$ to be the tree product of the sequence of groups with vertex groups
\allowdisplaybreaks
\begin{align*}
	U_{w_T r_{\{r, t\}}}, V_{w_T tr r_{\{s, t\}}}, U_{w_R r_{\{r, s\}}}, V_{w_R srr_{\{s, t\}}}, U_{w_{R'} r_{\{r, t\}}}, V_{w_{T'} sr r_{\{s, t\}}}, U_{w_{T'} r_{\{r, s\}}}.
\end{align*}
Furthermore, we define the group $C_{(R, R')}$ to be the tree product of the sequence of groups with vertex groups
\allowdisplaybreaks
\begin{align*}
	U_{w_Trtrr_{\{s, t\}}}, V_{w_Tr_{\{r, t\}}sr_{\{r, t\}}}, U_{w_Rrtr_{\{r, s\}}}, V_{w_Rrtsr_{\{r, t\}}}, U_{w_Rrr_{\{s, t\}}}, V_{w_Rrsr_{\{r, t\}}}, \\
	U_{w_R r_{\{r, s\}}}, V_{w_R srr_{\{s, t\}}}, U_{w_{R'} r_{\{r, t\}}}, V_{w_{R'}rr_{\{s, t\}}}, U_{w_{T'}r_{\{r, s\}}}.
\end{align*}
For completeness, the group $C_{(R', R)}$ is the tree product of the following sequence of groups with vertex groups
\allowdisplaybreaks
\begin{align*}
	U_{w_T r_{\{r, t\}}}, V_{w_R rr_{\{s, t\}}}, U_{w_R r_{\{r, s\}}}, V_{w_R srr_{\{s, t\}}}, U_{w_{R'} r_{\{r, t\}}}, \\
	V_{w_{R'} rtr_{\{r, s\}}}, U_{w_{R'} rr_{\{s, t\}}}, V_{w_{R'} rstr_{\{r, s\}}}, U_{w_{R'}rsr_{\{r, t\}}}, V_{w_{T'}r_{\{r, s\}}tr_{\{r, s\}}}, U_{w_{T'}rsrr_{\{s, t\}}}.
\end{align*}
It follows similarly as in Remark \ref{Remark: tree product generated by root group elements} that $C_R$, $C_{(R, R')}$ and $C_{(R', R)}$ are generated by suitable $u_{\alpha}$.

\begin{remark}
	Note that the vertex groups of $C_{(R', R)}$ can be obtained from $C_{(R, R')}$ by interchanging $s$ and $t$ and starting with the last vertex group of $C_{(R, R')}$. Interchanging $s$ and $t$ and the order of the vertex groups of $C$ does not change the group $C$.
\end{remark}

\begin{lemma}\label{RR'CCleftCright}
	Let $\{ R, R' \} \in \mathcal{T}_{i, 2}$. Then the canonical homomorphisms $C\to C_{(R, R')}, C_{(R', R)}$ are injective and we have $H_{\{R, R'\}} \cong C_{(R, R')} \star_C C_{(R', R)}$.
\end{lemma}
\begin{proof}
	We first show that $C \to C_{(R, R')}$ is injective. Let $\{r, s\}$ be the type of $R$ and let $\{r, t\}$ be the type of $R'$. Using Proposition \ref{treeproducts} the group $C_{(R, R')}$ is isomorphic to the tree product of the following sequence of groups with vertex groups
	\allowdisplaybreaks
	\begin{align*}
		U_{w_Trtrr_{\{s, t\}}} \hat{\star} V_{w_Tr_{\{r, t\}}sr_{\{r, t\}}} \hat{\star} U_{w_Rrtr_{\{r, s\}}}, V_{w_Rrtsr_{\{r, t\}}} \hat{\star} U_{w_Rrr_{\{s, t\}}}, \\ V_{w_Rrsr_{\{r, t\}}} \hat{\star} U_{w_R r_{\{r, s\}}}, V_{w_R srr_{\{s, t\}}}, U_{w_{R'} r_{\{r, t\}}}, V_{w_{R'}rr_{\{s, t\}}}, U_{w_{T'}r_{\{r, s\}}}.
	\end{align*}
	One easily sees that each vertex group of $C$ is contained in the corresponding vertex group of the previous tree product. Again we deduce from Lemma \ref{Lemma: Key lemma} and Proposition \ref{treeofgroupsinjective} that $C \to C_{(R, R')}$ is injective. Using similar arguments, we obtain that $C \to C_{(R', R)}$ is injective. Let $F_7$ be the tree product of the first seven vertex groups of the underlying sequence of groups of $H_{\{R, R'\}}$ and let $L_7$ be the tree product of the last seven vertex groups of the underlying sequence of groups of $H_{\{R, R'\}}$. It follows from the computations above that $U_{left} := U_{w_T r_{\{r, t\}}} \hat{\star} V_{w_R rr_{\{s, t\}}} \hat{\star} U_{w_R r_{\{r, s\}}} \to F_7$ and $U_{right} := U_{w_{R'} r_{\{r, t\}}} \hat{\star} V_{w_{R'} rr_{\{s, t\}}} \hat{\star} U_{w_{T'} r_{\{r, s\}}} \to L_7$ are injective. Moreover, $U_{right} \to C$ is injective by Proposition \ref{treeproducts}. Using Proposition \ref{treeproducts}, Lemma \ref{folding} and Remark \ref{Remark: isomorphism preserves amalgamated product} we obtain the following isomorphisms:
	\allowdisplaybreaks
	\begin{align*}
		H_{\{R, R'\}} &\cong F_7 \star_{U_{w_R srs}} V_{w_R srr_{\{s, t\}}} \star_{U_{w_{R'} trt}} L_7 \\
		&\cong F_7 \star_{U_{w_R srs}} V_{w_R srr_{\{s, t\}}} \star_{U_{w_{R'} trt}} U_{right} \star_{U_{right}} L_7 \\
		&\cong C_{(R, R')} \star_{U_{right}} L_7 \\
		&\cong C_{(R, R')} \star_C C \star_{U_{right}} L_7 \\
		&\cong C_{(R, R')} \star_C \left( C \star_{U_{right}} L_7 \right) \\
		&\cong C_{(R, R')} \star_C \left( U_{left} \star_{U_{w_R srs}} V_{w_R srr_{\{s, t\}}} \star_{U_{w_{R'} trt}} U_{right} \star_{U_{right}} L_7 \right) \\
		&\cong C_{(R, R')} \star_C \left( U_{left} \star_{U_{w_R srs}} V_{w_R srr_{\{s, t\}}} \star_{U_{w_{R'} trt}} L_7 \right) \\
		&\cong C_{(R, R')} \star_C C_{(R', R)}. \qedhere
	\end{align*}
\end{proof}

\begin{lemma}\label{CleftrightcapHTHT'}
	Let $\{ R, R' \} \in \mathcal{T}_{i, 2}$. Let $R$ be of type $\{r, s\}$, let $R'$ be of type $\{r, t\}$ and let $T' := R_{\{r, s\}}(w_{R'})$. Then $T' \in \mathcal{T}_{i-1, 1}$, the canonical homomorphism $C_{(R', R)} \to U_{T', s}$ is injective and we have $C_{(R', R)} \cap E_{T', s} = C$ in $U_{T', s}$. In particular, for $T := R_{\{r, t\}}(w_R)$ we have $T \in \mathcal{T}_{i-1, 1}$, the canonical homomorphism $C_{(R, R')} \to U_{T, t}$ is injective and we have $C_{(R, R')} \cap E_{T, t} = C$ in $U_{T, t}$.
\end{lemma}
\begin{proof}
	The claim $T, T' \in \mathcal{T}_{i-1, 1}$ follows from Lemma \ref{Lemma: not both down}, as for $Z := R_{\{s, t\}}(w_R)$ we have $\ell(w_Z trs), \ell(w_Z srt) \geq \ell(w_Z) +1$. We note that $\ell(w_{T'}ts) = \ell(w_{T'}) -2$. We let $w' = w_Z$. For completeness we recall that $U_{T', s}$ is the tree product of the underlying sequence of groups with vertex groups
	\allowdisplaybreaks
	\begin{align*}
		U_{w'ts r_{\{r, t\}}}, V_{w'tstr r_{\{s, t\}}}, U_{w'tstr_{\{r, s\}}}, V_{w_{T'} strr_{\{s, t\}}}, U_{w_{T'}sr_{\{r, t\}}}, V_{w_{T'} srtr_{\{r, s\}}}, \\
		U_{w_{T'}sr r_{\{s, t\}}}, V_{w_{T'}srst r_{\{r, s\}}}, U_{w_{T'}srsr_{\{r, t\}}}, V_{w_{T'} r_{\{r, s\}}tr_{\{r, s\}}}, U_{w_{T'}rsrr_{\{s, t\}}}, \\
		V_{w_{T'}rsrtr_{\{r, s\}}}, U_{w_{T'} rs r_{\{r, t\}}}, V_{w_{T'} rst r_{\{r, s\}}}, U_{w_{T'} r r_{\{s, t\}}}.
	\end{align*}
	As the first eleven vertex groups of $U_{T', s}$ coincide with the vertex groups of $C_{(R', R)}$, Proposition \ref{treeproducts} implies that $C_{(R', R)} \to U_{T', s}$ is injective. Before we show the claim, we have to analyse the embedding $E_{T', s} \to U_{T', s}$ from Lemma \ref{ERstoURsinjective} in more detail. Using Proposition \ref{treeproducts} the group $U_{T', s}$ is isomorphic to the tree product of the following sequence of groups with vertex groups
	\allowdisplaybreaks
	\begin{align*}
		U_{w'ts r_{\{r, t\}}}, V_{w'tstr r_{\{s, t\}}}, U_{w'tstr_{\{r, s\}}}, V_{w_{T'} strr_{\{s, t\}}}, U_{w_{T'}sr_{\{r, t\}}} \hat{\star} V_{w_{T'} srtr_{\{r, s\}}}, \\
		U_{w_{T'}sr r_{\{s, t\}}} \hat{\star} V_{w_{T'}srst r_{\{r, s\}}}, U_{w_{T'}srsr_{\{r, t\}}} \hat{\star} V_{w_{T'} r_{\{r, s\}}tr_{\{r, s\}}} \hat{\star} U_{w_{T'}rsrr_{\{s, t\}}}, \\
		V_{w_{T'}rsrtr_{\{r, s\}}} \hat{\star} U_{w_{T'} rs r_{\{r, t\}}}, V_{w_{T'} rst r_{\{r, s\}}} \hat{\star} U_{w_{T'} r r_{\{s, t\}}}.
	\end{align*}
	One easily sees that each vertex group of $E_{T', s}$ is contained in the corresponding vertex group of the previous tree product. Again we deduce from Lemma \ref{Lemma: Key lemma} and Proposition \ref{treeofgroupsinjective} that $E_{T', s} \to U_{T', s}$ is injective. We have known this already before, but this time we know how the embedding looks like and we can apply Corollary \ref{intersectionwithasubtree}. We deduce from it that in $U_{T', s}$ the intersection $C_{(R', R)} \cap E_{T', s}$ is equal to the tree product of the first seven vertex groups of the underlying sequence of groups of $E_{T', s}$, which is isomorphic to $C$.
\end{proof}

\section{Natural subgroups}\label{Section: Natural subgroups}

\begin{convention}
	In this section we let $(W, S)$ be of type $(4, 4, 4)$ and $\mathcal{M} = \left( M_{\alpha, \beta}^{\Gamma} \right)_{(\Gamma, \alpha, \beta) \in \mathcal{I}}$ be a locally Weyl-invariant commutator blueprint of type $(4, 4, 4)$. Moreover, we let $S = \{r, s, t\}$.
\end{convention}

For two elements $w_1, w_2 \in W$ we define $w_1 \prec w_2$ if $\ell(w_1) + \ell(w_1^{-1} w_2) = \ell(w_2)$. For any $w\in W$ we put $C(w) := \{ w' \in W \mid w' \prec w \}$. We now define for every $i \in \NN$ a subset $C_i \subseteq W$ as follows:
\allowdisplaybreaks
\begin{align*}
	C_0 := \bigcup_{S = \{r, s, t\}} \left( C(r_{\{s, t\}}) \cup C(rr_{\{s, t\}}) \right).
\end{align*}
For each $R \in \mathcal{R}_i$ of type $J = \{s, t\}$ we define
\allowdisplaybreaks
\begin{align*}
	C(R) := C( w_R str_{\{r, s\}} ) \cup C( w_R r_J rtr ) \cup C( w_R r_J rsr ) \cup C( w_R ts r_{\{r, t\}} ).
\end{align*}
For each $\{ R, R' \} \in \mathcal{T}_{i, 2}$ we define $C( \{R, R'\} ) := C(R) \cup C(R')$. We note that this union is not disjoint. For $i\geq 1$ we define
\allowdisplaybreaks
\begin{align*}
	C_i := C_{i-1} \cup \bigcup_{R\in \mathcal{R}_{i-1}} C(R) = C_{i-1} \cup \bigcup_{R \in \mathcal{T}_{i-1, 1}} C(R) \cup \bigcup_{\{ R, R' \} \in \mathcal{T}_{i-1, 2}} C(\{ R, R' \}).
\end{align*}
Moreover, we define $D_i := \{ w_R r_{\{s, t\}} \mid R \text{ is of type } \{s, t\}, w_Rs, w_Rt \in C_i \}$.

\begin{definition}\label{Definition: Gi}
	We denote by $G_i$ the direct limit of the inductive system formed by the groups $(U_w)_{w\in C_i}$ and $(V_{w'})_{w' \in D_i}$ together with the natural inclusions $U_w \to U_{ws}$ if $\ell(ws) = \ell(w) +1$ and $U_{w_Rs} \to V_{w_Rr_{\{s, t\}}}$.
\end{definition}

\begin{remark}\label{Remark: G_i is generated by x_alpha}
	Let $i \in \NN$. We will show that $G_i = \langle x_{\alpha} \mid \alpha \in \Phi_+, C_i \not\subseteq \alpha \rangle$. Note that $G_i$ is generated by elements $x_{\alpha, w}$ and $y_{\alpha, w'}$ for $w\in C_i$, $w' \in D_i$, where $x_{\alpha, w}$ is a generator of $U_w$ and $y_{\alpha, w'}$ is a generator of $V_{w'}$. We first note that for each $w' = w_R r_{\{s, t\}} \in D_i$ and all $\alpha \in \Phi_+$ with $w_R s \notin \alpha$, we have $x_{\alpha, w_R s} = y_{\alpha, w'}$ in $G_i$. Thus $G_i = \langle x_{\alpha, w} \mid \alpha \in \Phi_+, w\in C_i,  w\notin \alpha \rangle$.
	
	Suppose $s\in S$ and $w\in W$ with $w \notin \alpha_s$. Then $\ell(sw) = \ell(w) -1$. Let $k := \ell(w)$ and let $s_2, \ldots, s_k \in S$ be such that $w = s s_2 \cdots s_k$. Then, as $U_{ss_2 \dots s_m} \to U_{ss_2 \cdots s_{m+1}}$ are the canonical inclusions for any $1 \leq m \leq k-1$, we deduce $x_{\alpha_s, s} = x_{\alpha_s, w}$ in $G_i$. Let $\alpha \in \Phi_+$ be a non-simple root and let $\proj_{P_{\alpha}} 1_W \neq d \in P_{\alpha}$ (cf.\ Lemma~\ref{Lemma: Lemma 2.22 BiConstruction}). It is a consequence of Lemma \ref{projgal} that $x_{\alpha, d} = x_{\alpha, w}$ for every $w\in W$ with $w\notin \alpha$. Thus $G_i$ is generated by $\{ x_{\alpha} \mid \alpha \in \Phi_+, C_i \not\subseteq \alpha \}$.
\end{remark}

By the definition of the direct limit we have canonical homomorphisms $G_i \to G_{i+1}$ extending the identities $U_w \to U_w$ and $V_{w'} \to V_{w'}$. Let $G$ be the direct limit of the inductive system formed by the groups $\left( G_i \right)_{i\in \NN}$ with the canonical homomorphisms $G_i \to G_{i+1}$. Then the following diagram commutes for all $i \in \NN$ by definition:
\begin{center}
	\begin{tikzcd}
		G_i \arrow[r, "U_w \to U_w"] \arrow[rd] & G_{i+1} \arrow[d] \\
		& G
	\end{tikzcd}
\end{center}
Furthermore, the universal property of direct limits yields a unique homomorphism $f_i: G_i \to U_+$ (the group $U_+$ was introduced in Definition~\ref{Definition: U_+}) extending the identities $U_w \to U_w$ and $V_{w'} \to V_{w'} \leq U_{w'}$. Thus the following diagram commutes:
\begin{center}
	\begin{tikzcd}
		G_i \arrow[rd, "f_i"] \arrow[r, "U_w \to U_w"]& G_{i+1} \arrow[d, "f_{i+1}"] \\
		& U_+
	\end{tikzcd}
\end{center}
Again, the universal property of direct limits yields a unique homomorphism $f: G \to U_+$ such that the following diagram commutes for all $i \in \NN$:
\begin{center}
	\begin{tikzcd}
		G_i \arrow[r] \arrow[rd, "f_i"] & G \arrow[d, "f"] \\
		& U_+
	\end{tikzcd}
\end{center}

\begin{remark}\label{Remark: G is generated by x_alpha}
	By Remark \ref{Remark: G_i is generated by x_alpha}, the group $G_i$ is generated by $\{ x_{\alpha} \mid \alpha \in \Phi_+, C_i \not\subseteq \alpha \}$. We let $x_{\alpha, i}$ be the elements in $G$ under the homomorphism $G_i \to G$. Then $G$ is generated by $\{ x_{\alpha, i} \mid i \in \NN, \alpha \in \Phi_+, C_i \not\subseteq \alpha \}$. By construction we have $x_{\alpha, i} = x_{\alpha, i+1}$ in $G$ for each $i \in \NN$. Thus $G$ is generated by $\{ x_{\alpha} \mid \alpha \in \Phi_+ \}$.
\end{remark}

\begin{lemma}\label{Lemma: U_+ isomorphic to G}
	The homomorphism $f: G \to U_+$ is an isomorphism.
\end{lemma}
\begin{proof}
	By Remark \ref{Remark: G is generated by x_alpha} we have $G = \langle x_{\alpha} \mid \alpha \in \Phi_+ \rangle$. We will construct a homomorphism $U_+ \to G$ which extends $U_w \to U_w$. For all $w\in W$ we have a canonical homomorphism $U_w \to G$. Suppose $w\in W$ and $s\in S$ with $\ell(ws) = \ell(w) +1$. Then the following diagram commutes:
	\begin{center}
		\begin{tikzcd}
			U_w \arrow[r] \arrow[rd] & U_{ws} \arrow[d] \\
			& G
		\end{tikzcd}
	\end{center}
	The universal property of direct limits yields a homomorphism $h: U_+ \to G$ extending the identities on $U_w \to U_w$. As both concatenations $f \circ h$ and $h\circ f$ are the identities on each generator $x_{\alpha}$, the uniqueness of such a homomorphism implies $f\circ h = \mathrm{id}_{U_+}$ and $h\circ f = \mathrm{id}_G$. In particular, $f$ is an isomorphism.
\end{proof}

\begin{lemma}\label{Lemma: HP to Gi homomorphism}
	For each $P \in \mathcal{T}_i$ we have a canonical homomorphism $H_P \to G_i$. 
\end{lemma}
\begin{proof}
	We distinguish the following cases:
	\begin{enumerate}[label=$P \in \mathcal{T}_{i, \arabic*}$:]
		\item Let $\{s, t\}$ be the type of $P$. By Remark \ref{Remark: G_i is generated by x_alpha} it suffices to show that $C_i$ contains the elements $w_P sr_{\{r, t\}}, w_P r_{\{s, t\}}, w_P tr_{\{r, s\}}$. Note that $\ell(w_P) = i$. If $i=0$, the claim follows. Thus we can assume $i>0$ and hence $\ell(w_P r) = i-1$. But then $w_P sr_{\{r, t\}} \in C(R_{\{r, s\}}(w_P)) \subseteq C_i$ and $w_P tr_{\{r, s\}} \in C(R_{\{r, t\}}(w_P)) \subseteq C_i$. If $i=1$, we have $w_P r_{\{s, t\}} \in C_0 \subseteq C_1$ and we are done. If $i>1$, we have $i-2 \in \{ \ell(w_P rs), \ell(w_P rt) \}$. Without loss of generality we assume $\ell(w_P rs) = i-2$. Then $w_P r_{\{s, t\}} \in C(R_{\{r, s\}}(w_P)) \subseteq C_i$ and the claim follows.
		
		\item Suppose $P = \{R, R'\}$, where $R$ is of type $\{r, s\}$ and $R'$ is of type $\{r, t\}$. Moreover, we define $T := R_{\{r, t\}}(w_R)$ and $T' := R_{\{r, s\}}(w_{R'})$. Again, and using symmetry, it suffices to show that $w_T rtrr_{\{s, t\}}, w_T trtr_{\{r, s\}}, w_T trr_{\{s, t\}}, w_R r_{\{r, s\}} \in C_i$. We define $Z := R_{\{s, t\}}(w_R)$. Note that $\ell(w_Z) = i-3$ and hence $w_R r_{\{s, t\}} \in C(Z) \subseteq C_{i-2} \subseteq C_i$. Moreover, we have $\ell(w_T) = i-1$ and hence $w_T rtrr_{\{s, t\}}, w_T trtr_{\{r, s\}}, w_T trr_{\{s, t\}} \in C(T) \subseteq C_i$. This finishes the claim. \qedhere
	\end{enumerate}
\end{proof}

\begin{definition}\label{Definition: natural}
	The group $G_i$ is called \textit{natural} if the following hold:
	\begin{enumerate}[label=(N$\arabic*$)]
		\item For all $w\in C_i$ and $w' \in D_i$ the homomorphisms $U_w, V_{w'} \to G_i$ are injective.
		
		\item For each $P \in \mathcal{T}_i$ the homomorphism $H_P \to G_i$ from Lemma \ref{Lemma: HP to Gi homomorphism} is injective.
	\end{enumerate}
\end{definition}

\begin{definition}\label{Definition: BP}
	Suppose $G_i$ is natural and let $P \in \mathcal{T}_i$. Then the homomorphism $H_P \to G_i$ is injective. Note that by Lemma \ref{HRtoGRinjective} and Lemma \ref{HRR'toGRR'injective} the homomorphism $H_P \to G_P$ is injective as well. Thus we can define the tree product $B_P := G_i \star_{H_P} G_P$.
\end{definition}

\section{Detecting tree products}\label{Section: Detecting tree products}

In this section we let $(W, S)$ be of type $(4, 4, 4)$ and $\mathcal{M} = \left( M_{\alpha, \beta}^{\Gamma} \right)_{(\Gamma, \alpha, \beta) \in \mathcal{I}}$ be a locally Weyl-invariant commutator blueprint of type $(4, 4, 4)$. Moreover, we let $S = \{r, s, t\}$. The goal of this section is Proposition~\ref{Giinjective} where we show that $G_{i+1}$ is a tree product provided that $G_i$ is natural for some $i \in \NN$.

\begin{notation}
	\begin{enumerate}[label=(\alph*)]
		\item For $P \in \mathcal{T}_{i, 1}$ we denote the two non-simple roots of $P$ by $\delta_P$ and $\gamma_P$.
		
		\item For $P = \{ R, R' \} \in \mathcal{T}_{i, 2}$ there exists one root which is a non-simple root of $R$ and $R'$. We denote the other non-simple root of $R$ and of $R'$ by $\delta_P$ and $\gamma_P$.
	\end{enumerate}
	Note that in both cases there exists for each $\epsilon \in \{ \delta_P, \gamma_P \}$ a unique residue $R_{\epsilon}$ of rank $2$ such that $\epsilon$ is a non-simple root of $R_{\epsilon}$. Moreover, we have $k_{\delta_P} = k_{\gamma_P} = i+2$ by Lemma~\ref{Lemma: Lemma 2.22 BiConstruction}.
\end{notation}

\begin{lemma}\label{Lemma: set of non-simple roots contains four elements}
	Let $i \in \NN$ and let $P, Q \in \mathcal{T}_i$. If $P \neq Q$, then $\lvert \{ \delta_P, \gamma_P, \delta_Q, \gamma_Q \} \rvert = 4$.
\end{lemma}
\begin{proof}
	Without loss of generality we can assume $\delta_P = \delta_Q$. Then we have $R_{\delta_P} = R_{\delta_Q}$. If $P \in \mathcal{T}_{i, 1}$, then $P = R_{\delta_P} = R_{\delta_Q}$. Moreover, $Q \in \mathcal{T}_{i, 2}$ would imply $R_{\delta_Q} \in Q$, which is a contradiction to $R_{\delta_Q} \in \mathcal{T}_{i, 1}$. Thus $Q \in \mathcal{T}_{i, 1}$ and $P = R_{\delta_Q} = Q$. But this is a contradiction to our assumption. If $P \in \mathcal{T}_{i, 2}$, then $R_{\delta_Q} = R_{\delta_P} \in P$. In particular, we have $R_{\delta_Q} \notin \mathcal{T}_{i, 1}$. As $Q \in \mathcal{T}_{i, 1}$ would imply $Q = R_{\delta_Q}$, we deduce $Q \in \mathcal{T}_{i, 2}$ and $R_{\delta_Q} \in Q$. But $R_{\delta_Q} \in P \cap Q \neq \emptyset$ implies $P=Q$, which is again a contradiction.
\end{proof}

\begin{lemma}\label{Lemma: -epsilonP contained in epsilonQ}
	Let $i \in \NN$ and $P, Q \in \mathcal{T}_i$. If $i>0$ and $P \neq Q$, then we have $(-\epsilon_P) \subseteq \epsilon_Q$ for all $\epsilon_P \in \{ \delta_P, \gamma_P \}$ and $\epsilon_Q \in \{ \delta_Q, \gamma_Q \}$.
\end{lemma}
\begin{proof}
	Let $\epsilon_P \in \{ \delta_P, \gamma_P \}$, $\epsilon_Q \in \{ \delta_Q, \gamma_Q \}$ and assume $(-\epsilon_P) \not\subseteq \epsilon_Q$. As $1_W \in \epsilon_P \cap \epsilon_Q$, we have $\epsilon_Q \not\subseteq (-\epsilon_P)$ and $\{ -\epsilon_P, \epsilon_Q \}$ is not nested. Then \cite[Lemma~$8.42(3)$]{AB08} implies that $\{ \epsilon_P, \epsilon_Q \}$ is prenilpotent. By Lemma~\ref{Lemma: set of non-simple roots contains four elements} we have $\epsilon_P \neq \epsilon_Q$. As $k_{\epsilon_P} = i+2 = k_{\epsilon_Q}$, we deduce that the order of the product $r_{\epsilon_P} r_{\epsilon_Q}$ must be finite. 
	
	\emph{Claim: $R_{\epsilon_P} \notin \partial^2 \epsilon_Q$.}
	
	We assume by contrary that $R_{\epsilon_P} \in \partial^2 \epsilon_Q$. As $k_{\epsilon_P} = k_{\epsilon_Q}$, we deduce that $\epsilon_Q$ is a non-simple root of $R_{\epsilon_P}$ and, hence, $R_{\epsilon_P} = R_{\epsilon_Q}$. If $R_{\epsilon_P} \in \mathcal{T}_{i, 1}$, then $\epsilon_Q \in \{ \delta_P, \gamma_P \}$. This is a contradiction to Lemma~\ref{Lemma: set of non-simple roots contains four elements}. If $R_{\epsilon_P} \notin \mathcal{T}_{i, 1}$, then we have $\epsilon_P = \epsilon_Q$ by definition of the roots $\delta_P, \gamma_P$. This is again a contradiction and we infer $R_{\epsilon_P} \notin \partial^2 \epsilon_Q$.
	
	Note that $\epsilon_P$ and $\epsilon_Q$ are non-simple roots. Thus we can apply Lemma \ref{prenilpotentrootsintersectinoneresidue}. Assertion $(b)$ would imply $\epsilon_Q \in \{ \delta_P, \gamma_P \}$, which is a contradiction. Assertion $(a)$ would imply $i=0$ because of $k_{\epsilon_P} = k_{\epsilon_Q}$. This is also a contradiction.
\end{proof}

\begin{lemma}\label{Lemma: -epsilonP contained in epsilonQ for i and i+1}
	Let $i \in \NN$, let $P \in \mathcal{T}_{i+1}$ and let $Q \in \mathcal{T}_i$. For all $\epsilon_P \in \{ \delta_P, \gamma_P \}$ and $\epsilon_Q \in \{ \delta_Q, \gamma_Q \}$ one of the following holds:
	\begin{enumerate}[label=(\roman*)]
		\item $(-\epsilon_Q) \subseteq \epsilon_P$.
		
		\item $R_{\epsilon_P} \cap R_{\epsilon_Q}$ is a panel containing $w_{R_{\epsilon_P}}$ and $\ell( \proj_{R_{\epsilon_Q}} 1_W ) = \ell( \proj_{R_{\epsilon_P}} 1_W ) -1$.
	\end{enumerate}
\end{lemma}
\begin{proof}
	Let $\epsilon_P \in \{ \delta_P, \gamma_P \}$ and $\epsilon_Q \in \{ \delta_Q, \gamma_Q \}$. We can assume $(-\epsilon_Q) \not\subseteq \epsilon_P$. We have to show that $R_{\epsilon_P} \cap R_{\epsilon_Q}$ is a panel containing $w_{R_{\epsilon_Q}}$ and that $\ell( \proj_{R_{\epsilon_Q}} 1_W ) = \ell( \proj_{R_{\epsilon_P}} 1_W ) -1$. Similar as in the proof of Lemma~\ref{Lemma: -epsilonP contained in epsilonQ} we deduce that $\{ \epsilon_P, \epsilon_Q \}$ is prenilpotent. As $k_{\epsilon_Q} = i+2 = k_{\epsilon_P} -1$, we deduce that the order of the product $r_{\epsilon_P} r_{\epsilon_Q}$ is finite. 
	
	Suppose $R_{\epsilon_P} \in \partial^2 \epsilon_Q$. As $k_{\epsilon_Q} = k_{\epsilon_P} -1$, it follows that $R_{\epsilon_P} \cap R_{\epsilon_Q}$ is a panel containing $w_{R_{\epsilon_P}}$ and $\ell( \proj_{R_{\epsilon_Q}} 1_W ) = \ell( \proj_{R_{\epsilon_P}} 1_W ) -1$. Suppose $R_{\epsilon_P} \notin \partial^2 \epsilon_Q$. Then we can apply Lemma \ref{prenilpotentrootsintersectinoneresidue}. As $(b)$ does not apply, we obtain again (using $k_{\epsilon_Q} = k_{\epsilon_P} -1)$ that $R_{\epsilon_P} \cap R_{\epsilon_Q}$ is a panel containing $w_{R_{\epsilon_P}}$ and $\ell( \proj_{R_{\epsilon_Q}} 1_W ) = \ell( \proj_{R_{\epsilon_P}} 1_W ) -1$.
\end{proof}

\begin{lemma}\label{Lemma: w in -deltaP union -gammaP}
	Let $i \in \NN$, $P \in \mathcal{T}_i$ and $w\in C(P) \backslash C_i$. Then $w \in (-\delta_P) \cup (-\gamma_P)$.
\end{lemma}
\begin{proof}
	We distinguish the following two cases:
	\begin{enumerate}[label=$P \in \mathcal{T}_{i, \arabic*}$:]
		\item Let $P$ be of type $\{s, t\}$. Then we have $C(P) = C(w_P str_{\{r, s\}}) \cup C(w_P r_{\{s, t\}} rtr) \cup C(w_P r_{\{s, t\}} rsr) \cup C(w_P ts r_{\{r, t\}})$. As $w \notin C_i$, we infer $C(w) \cap \{ w_P st, w_P ts \} \neq \emptyset$. But this implies $w \in (-\delta_P) \cup (-\gamma_P)$.
		
		\item Suppose $P = \{R, R'\}$, where $R$ is of type $\{r, s\}$ and $R'$ is of type $\{r, t\}$. Then we have $C(P) = C(R) \cup C(R')$. As $w \notin C_i$, we infer that $C(w) \cap \{ w_R rs, w_R srs, w_{R'} trt, w_{R'} rt \} \neq \emptyset$. But this implies $w \in (-\delta_P) \cup (-\gamma_P)$. \qedhere
	\end{enumerate}
\end{proof}

\begin{lemma}\label{Lemma: unique P}
	For all $i \in \NN$ and $w\in C_{i+1} \backslash C_i$ there exists a unique $P \in \mathcal{T}_i$ with $w\in C(P)$.
\end{lemma}
\begin{proof}
	The existence follows from definition of $C_{i+1}$. Thus we assume $P \neq Q \in \mathcal{T}_i$ with $w \in C(P) \backslash C_i$ and $w \in C(Q) \backslash C_i$. Note that we have $w\in (-\delta_P) \cup (-\gamma_P)$ as well as $w\in (-\delta_Q) \cup (-\gamma_Q)$ by Lemma~\ref{Lemma: w in -deltaP union -gammaP}. In particular, we have $w \notin \delta_P \cap \gamma_P$ and $w \notin \delta_Q \cap \gamma_Q$. Note that we have $\lvert \{ \delta_P, \gamma_P, \delta_Q, \gamma_Q \} \rvert = 4$ by Lemma~\ref{Lemma: set of non-simple roots contains four elements}.
	
	\emph{Claim: There exist $\epsilon_P \in \{ \delta_P, \gamma_P \}$, $\epsilon_Q \in \{ \delta_Q, \gamma_Q \}$ such that $\{ \epsilon_P, \epsilon_Q \}$ is prenilpotent.}
	
	Assume that non of $\{ \delta_P, \delta_Q \}$, $\{ \delta_P, \gamma_Q \}$, $\{ \gamma_P, \delta_Q \}$, $\{ \gamma_P, \gamma_Q \}$ is prenilpotent. Then \cite[Lemma $8.42(3)$]{AB08} yields that each of $\{ \delta_P, (-\delta_Q) \}$, $\{ \delta_P, (-\gamma_Q) \}$, $\{ \gamma_P, (-\delta_Q) \}$, $\{ \gamma_P, (-\gamma_Q) \}$ is nested. As $1_W \in \delta_P \cap \gamma_P \cap \delta_Q \cap \gamma_Q$, it follows that $(-\delta_Q), (-\gamma_Q) \subseteq \delta_P, \gamma_P$. But this implies $w \in (-\delta_Q) \cup (-\gamma_Q) \subseteq \delta_P \cap \gamma_P$, which is a contradiction.
	
	Suppose $i >0$. Then Lemma~\ref{Lemma: -epsilonP contained in epsilonQ} implies that $\{ (-\epsilon_P), \epsilon_Q \}$ is nested. Using \cite[Lemma~$8.42(3)$]{AB08} we infer that $\{ \epsilon_P, \epsilon_Q \}$ is not prenilpotent, which is a contradiction to the claim. Thus we have $i=0$. Let $\{s, t\}$ be the type of $P$ and let $\{r, s\}$ be the type of $Q$. Then we have $P = R_{\{s, t\}}(1_W)$ and $Q = R_{\{r, s\}}(1_W)$. Without loss of generality we let $\delta_Q = s\alpha_r, \gamma_Q = r\alpha_s$. It follows from Lemma \ref{mingallinrep} that $w\in (-\delta_P) \cup (-\gamma_P) \subseteq \alpha_r$. Note that $w \in C(P) \subseteq (-t\alpha_s) \cup \{t\} \cup C(strsr) \subseteq \delta_Q$. Lemma \ref{Lemma: intersection of roots} yields $\alpha_s \subseteq (-\alpha_r) \cup s\alpha_r$ and, as $(W, S)$ is of type $(4, 4, 4)$, we deduce $(-r\alpha_s) \subseteq (-s\alpha_r) \cup (-\alpha_r)$. This implies $\alpha_r \cap s\alpha_r \subseteq r\alpha_s$. But then $w\in \alpha_r \cap \delta_Q \subseteq \gamma_Q$, which is a contradiction to $w\notin \delta_Q \cap \gamma_Q$. This finishes the claim.
\end{proof}

\begin{notation}
	For $i \in \NN$ and $P \in \mathcal{T}_i$ we let $C'(P) \subseteq W$ be the union of all $C(w)$, where $w\in W$ and $U_w$ is a vertex group of $G_P$.
\end{notation}

\begin{lemma}\label{Lemma: C'(P) contained in C_i+1}
	For $i \in \NN$ and $P \in \mathcal{T}_i$ we have $C'(P) \subseteq C_{i+1}$.
\end{lemma}
\begin{proof}
	We distinguish the following two cases:
	\begin{enumerate}[label=$P \in \mathcal{T}_{i, \arabic*}$:]
		\item Suppose that $P$ is of type $\{s, t\}$. Note that $C'(P) = C(P) \cup C(w_P sr_{\{r, t\}}) \cup C(w_P tr_{\{r, s\}})$. By definition, we have $C(P) \subseteq C_{i+1}$ and (using symmetry) it suffices to show $C(w_P sr_{\{r, t\}}) \subseteq C_{i+1}$. For $i = 0$ we have $C(w_P sr_{\{r, t\}}) \subseteq C_0 \subseteq C_1$. For $i>0$ we have $C(w_P sr_{\{r, t\}}) \subseteq C(R_{\{r, s\}}(w_P)) \subseteq C_i \subseteq C_{i+1}$.
		
		\item Suppose $P = \{R, R'\}$, where $R$ is of type $\{r, s\}$ and $R'$ is of type $\{r, t\}$. As in the previous case it suffices to show $C(w_R rtrsts) \cup C(w_R rtrsrs) \cup C(w_R rr_{\{s, t\}}) \subseteq C_{i+1}$. As $R_{\{r, t\}}(w_R) \in \mathcal{R}_{i-1}$, we obtain $C(w_R rtrsts) \cup C(w_R rtrsrs) \cup C(w_R rr_{\{s, t\}}) \subseteq C(R_{\{r, t\}}(w_R)) \subseteq C_{i+1}$. \qedhere
	\end{enumerate}
\end{proof}

\begin{definition}
	Let $i\in \NN$ and let $R \in \mathcal{R}_i$ be a residue of type $\{s, t\}$. We let $\hat{\Phi}_R$ be the set of all non-simple roots of $R_{\{r, s\}}(w_R st)$, $R_{\{r, t\}}(w_R r_{\{s, t\}})$, $R_{\{r, s\}}(w_R r_{\{s, t\}})$ and $R_{\{r, t\}}(w_R ts)$. If $P := \{ R, R' \} \in \mathcal{T}_{i, 2}$, then we define $\hat{\Phi}_P := \hat{\Phi}_R \cup \hat{\Phi}_{R'}$.
\end{definition}

\begin{lemma}\label{Lemma: auxiliary result}
	Let $i \in \NN$, let $R \in \mathcal{R}_i$ be of type $\{s, t\}$ and let $\alpha \in \hat{\Phi}_R$. If $\ell(w_R r) = \ell(w_R) -1$ and $\ell(w_R rt) = \ell(w_R)$, then $C(R_{\{r, t\}}(w_R)) \subseteq \alpha$ and $(- w_R tr \alpha_t) \subseteq \alpha$.
\end{lemma}
\begin{proof}
	We denote the two non-simple roots of $R$ by $\alpha_R$ and $\beta_R$. Note that $\alpha_R \subseteq \alpha$ or $\beta_R \subseteq \alpha$ holds by Lemma~\ref{mingallinrep}. We abbreviate $T := R_{\{r, t\}}(w_R)$.
	
	Recall that $C(T) = C( w_T tr r_{\{s, t\}} ) \cup C( w_T r_{\{r, t\}} srs ) \cup C( w_T r_{\{r, t\}} sts ) \cup C( w_R t r_{\{r, s\}} )$. Using Lemma~\ref{mingallinrep}, we obtain $\{ w_T tr r_{\{s, t\}} , w_T r_{\{r, t\}} srs, w_T r_{\{r, t\}} sts \} \subseteq (-w_R tr \alpha_t) \subseteq \alpha_R \cap \beta_R \subseteq \alpha$. Using Lemma~\ref{mingallinrep} again, we have $w_R t r_{\{r, s\}} \in (-w_R t\alpha_r) \subseteq w_R s\alpha_t$. Note that we have $w_R s\alpha_t \subseteq \alpha$ or $\alpha \in \{ w_R tsr \alpha_t, w_R tst \alpha_r \}$. In both cases we deduce $w_R t r_{\{r, s\}} \in \alpha$. As roots are convex, we obtain $C(T) \subseteq \alpha$.
\end{proof}

\begin{lemma}\label{Lemma: C_i in alpha for i = 0, 1}
	Let $i \in \{0, 1, 2\}$, $R \in \mathcal{R}_i$ and let $\alpha \in \hat{\Phi}_R$. Then we have $C_i \subseteq \alpha$.
\end{lemma}
\begin{proof}
	Let $R$ be of type $\{s, t\}$. For $i=0$ it is not hard to see that
	\[ 	C_0 = \bigcup_{S = \{r, s, t\}} \left( C(r_{\{s, t\}}) \cup C(rr_{\{s, t\}}) \right) \subseteq \alpha .\]
	Thus we consider the case $i=1$. Then $R = R_{\{s, t\}}(r)$. Clearly, $rr_{\{s, t\}} \in \alpha$. Using Lemma \ref{mingallinrep} we see that $\alpha_r, -\alpha_s, -\alpha_t \subseteq \delta_R, \gamma_R$ and, as $\delta_R \subseteq \alpha$ or $\gamma_R \subseteq \alpha$ (cf.\ Lemma~\ref{mingallinrep}), we deduce $\alpha_r, -\alpha_s, -\alpha_t \subseteq \alpha$. Now $C_0 \subseteq \alpha$ follows from the fact that roots are convex. For $T := R_{\{s, t\}}(1_W)$ it follows from Lemma~\ref{Lemma: w in -deltaP union -gammaP} and Lemma~\ref{mingallinrep} that $C(T) \subseteq C_0 \cup (-\delta_T) \cup (-\gamma_T) \subseteq C_0 \cup \alpha_r \subseteq \alpha$. Using symmetry it suffices to show that $C(R_{\{r, t\}}(1_W)) \subseteq \alpha$. But this follows from Lemma~\ref{Lemma: auxiliary result}.
\end{proof}

\begin{lemma}\label{Lemma: head residues non-simple roots}
	Let $i\in \NN$, $P \in \mathcal{T}_i$ and let $\alpha \in \hat{\Phi}_P$ be a root. Then we have $C_i \subseteq \alpha$.
\end{lemma}
\begin{proof}
	We prove the hypothesis by induction on $i$. The cases $i \in \{0, 1\}$ are proven in Lemma~\ref{Lemma: C_i in alpha for i = 0, 1}. Thus we can assume $i\geq 2$. For $j \in \NN$ and a residue $T \in \mathcal{R}_j$ we denote by $P_T \in \mathcal{T}_j$ the unique element with $P_T = T$ or $T \in P_T$.
	
	\emph{Claim $A$: If $P \in \mathcal{T}_{i, 1}$, then $C_i \subseteq  \alpha$.}
	
	Suppose $P \in \mathcal{T}_{i, 1}$ is of type $\{s, t\}$. As $i\geq 2$, we have $\ell(w_P) -2 \in \{ \ell(w_P rs), \ell(w_P rt) \}$. Without loss of generality we can assume $\ell(w_P rs) = \ell(w_P) -2$. Note that $\delta_P \subseteq \alpha$ or $\gamma_P \subseteq \alpha$ holds (cf.\ Lemma~\ref{mingallinrep}). We define $T := R_{\{r, t\}}(w_P)$ and $T' := R_{\{r, s\}}(w_P)$. Note that $T \in \mathcal{R}_{i-1} \cup \mathcal{R}_{i-2}$ by Lemma~\ref{Lemma: not both down}.
	
	\emph{Claim $A1$: We have $C_i \subseteq C_{i-1} \cup \alpha$.}
	
	Recall that $C_i = C_{i-1} \cup \bigcup_{P \in \mathcal{T}_{i-1}} C(P)$. Let $Q \in \mathcal{T}_{i-1} \backslash \{ P_T \}$. By Lemma~\ref{Lemma: w in -deltaP union -gammaP} we obtain $C(Q) \subseteq C_{i-1} \cup (-\delta_Q) \cup (-\gamma_Q)$. Using Lemma~\ref{Lemma: -epsilonP contained in epsilonQ for i and i+1}, the fact $Q \neq P_T$ implies $(-\delta_Q), (-\gamma_Q) \subseteq \delta_P, \gamma_P$ and hence $C(Q) \subseteq C_{i-1} \cup (-\delta_Q) \cup (-\gamma_Q) \subseteq C_{i-1} \cup (\delta_P \cap \gamma_P) \subseteq C_{i-1} \cup \alpha$. If $P_T \notin \mathcal{T}_{i-1}$, then we are done. Thus we suppose $P_T \in \mathcal{T}_{i-1}$. In particular, $\ell(w_P rt) = \ell(w_P)$. We deduce from Lemma~\ref{Lemma: auxiliary result} that $C(T) \subseteq \alpha$. If $P_T \in \mathcal{T}_{i-1, 1}$, we are done. Thus we can assume $P_T  \in \mathcal{T}_{i-1, 2}$, i.e.\ $P_T = \{ T, R_{\{r, s\}}(w_P rt) \}$. Note that $C( R_{\{r, s\}}(w_P rt) ) = C( w_T tsrs r_{\{r, t\}} ) \cup C( w_T trsrtst ) \cup C( w_T trs r_{\{r, t\}} ) \cup C( w_T tr r_{\{s, t\}} )$. Using Lemma~\ref{Lemma: auxiliary result} we obtain that $\{ w_T tsrs r_{\{r, t\}}, w_T trsrtst, w_T trs r_{\{r, t\}}, w_T tr r_{\{s, t\}} \} \subseteq (-w_T \alpha_t ) \subseteq \alpha$. As roots are convex, we infer $C(P_T) = C(T) \cup C( R_{\{r, s\}}(w_P rt) ) \subseteq \alpha$.
	
	In the rest of the proof of Claim~$A$ we will show $C_{i-1} \subseteq \alpha$. Together with Claim~$A1$ this finishes the proof of Claim~$A$. Recall that $C_{i-1} = C_{i-2} \cup \bigcup_{Q \in \mathcal{T}_{i-2}} C(Q)$.
	
	\emph{Claim $A2$: If $\ell(w_P rt) = \ell(w_P) -2$, then $C_{i-1} \subseteq \alpha$.}
	
	As $\ell(w_P rt) = \ell(w_P) -2$, we have $Q := \{T, T'\} \in \mathcal{T}_{i-2, 2}$. In particular, $i-2 >0$. Then $\delta_P, \gamma_P  \in \hat{\Phi}_Q$ and the induction hypothesis implies $C_{i-2} \subseteq \delta_P \cap \gamma_P \subseteq \alpha$. Let $Z \in \mathcal{T}_{i-2} \backslash \{ Q \}$. Note that by Lemma~\ref{Lemma: Subset contained in certain root} and Lemma~\ref{mingallinrep} we have $\delta_Q \cap \gamma_Q \subseteq w_P r\alpha_r \cup \{ w_P \} \subseteq \delta_P \cap \gamma_P \subseteq \alpha$. Using Lemma~\ref{Lemma: w in -deltaP union -gammaP} and Lemma \ref{Lemma: -epsilonP contained in epsilonQ} we deduce $C(Z) \subseteq C_{i-2} \cup (-\delta_Z) \cup (-\gamma_Z) \subseteq C_{i-2} \cup (\delta_Q \cap \gamma_Q) \subseteq \alpha$. Now we consider $Z=Q$. Note that $C(Q) = C(T) \cup C(T')$ and, using symmetry, it suffices to show $C(T') \subseteq \alpha$. Recall that $C(T') = C( w_{T'} rs r_{\{r, t\}} ) \cup C( w_{T'} r_{\{r, s\}} tst ) \cup C( w_{T'} r_{\{r, s\}} trt ) \cup C( w_{T'} sr r_{\{s, t\}} )$. Using Lemma~\ref{mingallinrep}, we deduce $w_{T'} rs r_{\{r, t\}} \in w_{T'} \alpha_s \subseteq \delta_P \cap \gamma_P \subseteq \alpha$, $w_{T'} r_{\{r, s\}} tst \in (-w_P srt\alpha_s) \subseteq \delta_P \cap \gamma_P \subseteq \alpha$ and $w_{T'} r_{\{r, s\}} trt \in (-w_Pstrt\alpha_r) \subseteq \alpha$. Moreover, $w_P r_{\{s, t\}} \in \alpha$. As roots are convex, we deduce $C(T') \subseteq \alpha$.
	
	\emph{Claim $A3$: If $\ell(w_P rt) = \ell(w_P)$, then $C_{i-1} \subseteq \alpha$.}
	
	As $P \in \mathcal{T}_{i, 1}$ we have $\ell(w_P rsr) = \ell(w_P) -1$. As $\delta_P, \gamma_P \in \hat{\Phi}_{T'}$, we deduce $C_{i-2} \subseteq \delta_P \cap \gamma_P \subseteq \alpha$ by induction. As in Claim~$A2$ we deduce $C(T') \subseteq \alpha$. Suppose first $i-2 = 0$. Note that $\mathcal{T}_0 = \{ R_{\{s, t\}}(1_W), R_{\{r, s\}}(1_W), R_{\{r, t\}}(1_W) \}$. For $Q \in \{ R_{\{r, s\}}(1_W), R_{\{r, t\}}(1_W) \}$ it follows from Lemma~\ref{Lemma: w in -deltaP union -gammaP}, Lemma~\ref{mingallinrep} and induction that $C(Q) \subseteq C_0 \cup (-\delta_Q) \cup (-\gamma_Q) \subseteq C_0 \cup (\delta_P \cap \gamma_P) \subseteq \alpha$. As $R_{\{r, s\}}(1_W) = T'$, we conclude $C(Q) \subseteq \alpha$ for all $Q \in \mathcal{T}_{i-2}$. Thus we assume $i-2 >0$. Let $Q \in \mathcal{T}_{i-2} \backslash \{ P_{T'} \}$. Note that $w_P r\alpha_r \in \{ \delta_{T'}, \gamma_{T'} \}$. Then Lemma~\ref{Lemma: w in -deltaP union -gammaP}, Lemma~\ref{Lemma: -epsilonP contained in epsilonQ} and Lemma~\ref{mingallinrep} imply $C(Q) \subseteq C_{i-2} \cup (-\delta_Q) \cup (-\gamma_Q) \subseteq C_{i-2} \cup w_P r\alpha_r \subseteq C_{i-2} \cup (\delta_P \cap \gamma_P) \subseteq \alpha$. As in Claim~$A2$ we deduce $C(T') \subseteq \alpha$. If $T' \in \mathcal{T}_{i-2, 1}$ we are done. Otherwise, we have $P_{T'} = \{ T', R_{\{s, t\}}(w_{T'} r) \}$. Note that $C( R_{\{s, t\}}(w_{T'} r) ) \subseteq w_{T'} \alpha_s \subseteq \delta_P \cap \gamma_P \subseteq \alpha$ holds by Lemma~\ref{mingallinrep} and the fact that roots are convex. We deduce $C(P_{T'}) = C(T') \cup C( R_{\{s, t\}}(w_{T'} r) ) \subseteq \alpha$.
	
	\emph{Claim $B$: If $P \in \mathcal{T}_{i, 2}$, then $C_i \subseteq  \alpha$.}
	
	Suppose $P = \{ R, R' \}$, where $R$ is of type $\{r, s\}$ and $R'$ is of type $\{r, t\}$. Let $\epsilon_P := w_R s\alpha_r$. Note that there exists $\beta \in \{ \delta_P, \epsilon_P, \gamma_P \}$ with $\beta \subseteq \alpha$. Suppose $\delta_P \not\subseteq \alpha$ and $\gamma_P \not\subseteq \alpha$. Then $\epsilon_P \subseteq \alpha$. By Lemma~\ref{Lemma: Subset contained in certain root} we have $\delta_P \cap \gamma_P \subseteq \epsilon_P \cup \{ w_R sr \} \subseteq \alpha$. This implies $\delta_P \cap \gamma_P \subseteq \alpha$ in all cases. Define $M := R_{\{s, t\}}(w_R)$.
	
	\emph{Claim $B1$: We have $C_i \subseteq C_{i-1} \cup \alpha$.}
	
	Recall that $C_i = C_{i-1} \cup \bigcup_{P \in \mathcal{T}_{i-1}} C(P)$. Define $T := R_{\{r, t\}}(w_R)$ and $T' := R_{\{r, s\}}(w_{R'})$. Then $T, T' \in \mathcal{T}_{i-1, 1}$ (cf.\ Lemma~\ref{CleftrightcapHTHT'}). Let $Q \in \mathcal{T}_{i-1} \backslash \{ T, T' \}$. By Lemma~\ref{Lemma: w in -deltaP union -gammaP} we obtain $C(Q) \subseteq C_{i-1} \cup (-\delta_Q) \cup (-\gamma_Q)$. Using Lemma~\ref{Lemma: -epsilonP contained in epsilonQ for i and i+1}, the fact that $Q \notin \{T, T'\}$ implies $(-\delta_Q), (-\gamma_Q) \subseteq \delta_P, \gamma_P$ and hence $C(Q) \subseteq C_{i-1} \cup (-\delta_Q) \cup (-\gamma_Q) \subseteq C_{i-1} \cup (\delta_P \cap \gamma_P) \subseteq C_{i-1} \cup \alpha$. It is left to show $C(T) \cup C(T') \subseteq \alpha$. Using symmetry, it suffices to consider $T$. If $\alpha \in \hat{\Phi}_R$, then we deduce $C(T) \subseteq \alpha$ from Lemma~\ref{Lemma: auxiliary result}. Thus we suppose $\alpha \notin \hat{\Phi}_R$. Then Lemma~\ref{mingallinrep} implies $w_{R'} r\alpha_t \subseteq \alpha$. Using Lemma~\ref{Lemma: w in -deltaP union -gammaP} and Lemma~\ref{mingallinrep} we conclude $C(T) \subseteq C_{i-1} \cup (-\delta_T) \cup (-\gamma_T) \subseteq C_{i-1} \cup w_{R'} r\alpha_t \subseteq C_{i-1} \cup \alpha$.
	
	In the rest of the proof of Claim~$B$ we will show $C_{i-1} \subseteq \alpha$. Together with Claim~$B1$ this finishes the proof of Claim~$B$. Recall that $C_{i-1} = C_{i-2} \cup \bigcup_{Q \in \mathcal{T}_{i-2}} C(Q)$ and $C_{i-2} = C_{i-3} \cup \bigcup_{Q \in \mathcal{T}_{i-3}} C(Q)$.
	
	\emph{Claim $B2$: We have $C_{i-2} \subseteq \alpha$.}
	
	As $P_M \in \mathcal{T}_{i-3}$ and $\delta_P, \gamma_P \in \hat{\Phi}_{P_M}$, the induction hypothesis implies $C_{i-3} \subseteq \delta_P \cap \gamma_P \subseteq \alpha$. We first show $C(M) \subseteq \alpha$. Note that $C(M) = C( w_M ts r_{\{r, t\}} ) \cup C( w_M r_{\{s, t\}} rsr ) \cup C( w_M r_{\{s, t\}} rtr ) \cup C( w_M st r_{\{r, s\}} )$. Note that $w_M r_{\{s, t\}} rsr, w_M r_{\{s, t\}} rtr \in \alpha$. Using Lemma~\ref{mingallinrep} we deduce $w_M ts r_{\{r, t\}} \in (-w_M ts \alpha_r) \subseteq \delta_P \cap \gamma_P \subseteq \alpha$ and $w_M st r_{\{r, s\}} \in (-w_M st\alpha_r) \subseteq \delta_P \cap \gamma_P \subseteq \alpha$. As roots are convex, we infer $C(M) \subseteq \alpha$. Note that $\{ w_M s\alpha_t, w_M t\alpha_s \} \cap \{ \delta_{P_M}, \gamma_{P_M} \} \neq \emptyset$ and by Lemma~\ref{mingallinrep} we have $w_M s\alpha_t, w_M t\alpha_s \subseteq \delta_P, \gamma_P$. We have to show $C(Q) \subseteq \alpha$ for all $Q \in \mathcal{T}_{i-3}$. Suppose $i-3 = 0$. Then $\mathcal{T}_0 = \{ R_{\{s, t\}}(1_W), R_{\{r, s\}}(1_W), R_{\{r, t\}}(1_W) \}$. Note that $R_{\{s, t\}}(1_W) = M$ and we have already shown $C(M) \subseteq \alpha$. Using symmerty, it suffices to show $C(R_{\{r, s\}}(1_W)) \subseteq \alpha$. It follows from Lemma~\ref{Lemma: w in -deltaP union -gammaP}, Lemma~\ref{mingallinrep} and the fact that roots are convex that $C(R_{\{r, s\}}(1_W)) \subseteq C_0 \cup st\alpha_s \subseteq C_0 \cup (\delta_P \cap \gamma_P) \subseteq \alpha$. Thus we can suppose $i-3 >0$. Let $Q \in \mathcal{T}_{i-3} \backslash \{ P_M \}$. Then it follows from Lemma~\ref{Lemma: w in -deltaP union -gammaP}, Lemma~\ref{Lemma: -epsilonP contained in epsilonQ} and Lemma~\ref{mingallinrep} that $C(Q) \subseteq C_{i-3} \cup (-\delta_Q) \cup (-\gamma_Q) \subseteq C_{i-3} \cup (\delta_P \cap \gamma_P) \subseteq \alpha$. Now we consider $P_M$. We have already shown $C(M) \subseteq \alpha$. If $P_M = M$, then we are done. Thus we can assume $P_M \neq M$. Without loss of generality we can assume $P_M = \{ M, M' \}$, where $M'$ is of type $\{r, t\}$. Note that $C(M') = C( w_{M'} rt r_{\{r, s\}} ) \cup C( w_{M'} r_{\{r, t\}} sts ) \cup C( w_{M'} r_{\{r, t\}} srs ) \cup C( w_{M'} tr r_{\{s, t\}} )$. Moreover, we have $C( w_{M'} rt r_{\{r, s\}} ) \subseteq C(M) \subseteq \alpha$. By Lemma~\ref{mingallinrep} we have $\{ w_{M'} r_{\{r, t\}} sts, w_{M'} r_{\{r, t\}} srs, w_{M'} tr r_{\{s, t\}} \} \subseteq (-w_{M'} \alpha_t) \subseteq w_M t\alpha_s \subseteq \delta_P \cap \gamma_P \subseteq \alpha$. As roots are convex, we obtain $C(M') \subseteq \alpha$ and, hence, $C(P_M) = C(M) \cup C(M') \subseteq \alpha$.
	
	\emph{Claim $B3$: We have $C_{i-1} \subseteq \alpha$.}
	
	By Claim~$B2$ it suffices to show $C(Q) \subseteq \alpha$ for all $Q \in \mathcal{T}_{i-2}$. We distinguish the following cases:
	\begin{enumerate}[label=(\alph*)]
		\item Suppose $M \in \mathcal{T}_{i-3, 1}$: Define $X := R_{\{r, s\}}(w_M t)$, $Y := R_{\{r, t\}}(w_M s)$ and note that $X, Y \in \mathcal{R}_{i-2}$. Let $Q \in \mathcal{T}_{i-2} \backslash \{ P_X, P_Y \}$. Then it follows from Lemma~\ref{Lemma: w in -deltaP union -gammaP}, Lemma~\ref{Lemma: -epsilonP contained in epsilonQ for i and i+1}, Lemma~\ref{mingallinrep} and Claim~$B2$ that $C(Q) \subseteq C_{i-2} \cup (-\delta_Q) \cup (-\gamma_Q) \subseteq C_{i-2} \cup (\delta_M \cap \gamma_M) \subseteq C_{i-2} \cup (\delta_P \cap \gamma_P) \subseteq \alpha$. It is left to show $C(P_X) \cup C(P_Y) \subseteq \alpha$. Using symmetry it suffices to show $C(P_X) \subseteq \alpha$. Using Lemma~\ref{Lemma: w in -deltaP union -gammaP} we have $C(P_X) \subseteq C_{i-2} \cup (-\delta_{P_X}) \cup (-\gamma_{P_X})$. If $P_X = X$, then $\{ \delta_{P_X}, \gamma_{P_X} \} = \{ w_M tr\alpha_s, w_M ts\alpha_r \}$. Using Lemma~\ref{mingallinrep} and Claim~$B2$ we infer $C(P_X) \subseteq C_{i-2} \cup (-w_M tr\alpha_s) \cup (-w_M ts\alpha_r) \subseteq C_{i-2} \cup (\delta_P \cap \gamma_P) \subseteq \alpha$. If $P_X \neq X$, then $\{ \delta_{P_X}, \gamma_{P_X} \} = \{ w_M ts\alpha_r, w_M trs\alpha_t \}$. Lemma~\ref{mingallinrep} and Claim~$B2$ yield $C(P_X) \subseteq C_{i-2} \cup (-w_M ts\alpha_r) \cup (-w_M trs\alpha_t) \subseteq \alpha \cup w_M t\alpha_s \subseteq \alpha \cup (\delta_P \cap \gamma_P) \subseteq \alpha$.
		
		\item Suppose $M \notin \mathcal{T}_{i-3, 1}$: Without loss of generality we can assume $P_M = \{ M, M' \}$, where $M'$ is of type $\{ r, t \}$. Define $X := R_{\{r, s\}}(w_M t)$, $Y := R_{\{ r, s \}}(w_{M'} t)$ and note that $X, Y \in \mathcal{T}_{i-2, 1}$ as a consequence of Lemma~\ref{Lemma: not both down}. Let $Q \in \mathcal{T}_{i-2} \backslash \{ X, Y \}$. Then it follows from Lemma~\ref{Lemma: w in -deltaP union -gammaP}, Lemma~\ref{Lemma: -epsilonP contained in epsilonQ for i and i+1}, Lemma~\ref{mingallinrep} and Claim~$B2$ that $C(Q) \subseteq C_{i-2} \cup (-\delta_Q) \cup (-\gamma_Q) \subseteq C_{i-2} \cup (\delta_{P_M} \cap \gamma_{P_M}) \subseteq C_{i-2} \cup w_M t\alpha_s \subseteq C_{i-2} \cup (\delta_P \cap \gamma_P) \subseteq \alpha$. It is left to show $C(X) \cup C(Y) \subseteq \alpha$. As in the previous case we deduce $C(X) \subseteq \alpha$. Using Lemma~\ref{Lemma: w in -deltaP union -gammaP}, Lemma~\ref{mingallinrep} and Claim~$B2$ we deduce $C(Y) \subseteq C_{i-2} \cup (-\delta_Y) \cup (-\gamma_Y) \subseteq C_{i-2} \cup (-w_{M'} \alpha_t) \subseteq C_{i-2} \cup w_M t\alpha_s \subseteq C_{i-2} \cup (\delta_P \cap \gamma_P) \subseteq \alpha$. \qedhere
	\end{enumerate}
\end{proof}

\begin{lemma}\label{Lemma: Uw to GP}
	Let $i\in \NN$, $P \in \mathcal{T}_i$ and $w\in C(P)$. Then there is a canonical homomorphism $U_w \to G_P$. In particular, this homomorphism is injective.
\end{lemma}
\begin{proof}
	We distinguish the following two cases:
	\begin{enumerate}[label=$P \in \mathcal{T}_{i, \arabic*}$:]
		\item Suppose that $P$ is of type $\{s, t\}$. Then we have $C(P) = C(w_P str_{\{r, s\}}) \cup C(w_P r_{\{s, t\}}rtr) \cup C(w_P r_{\{s, t\}}rsr) \cup C(w_P tsr_{\{r, t\}})$. As $U_v \to U_{vs}$ is injective, we can assume $w \in \{ w_P str_{\{r, s\}}, w_P r_{\{s, t\}}rtr, w_P r_{\{s, t\}}rsr, w_P tsr_{\{r, t\}} \}$. By definition of $G_P$ and Proposition \ref{treeproducts} we see that $U_w \to G_P$ is injective.
		
		\item Suppose $P = \{ R, R' \}$, where $R$ is of type $\{r, s\}$ and $R'$ is of type $\{r, t\}$. As in the case $P \in \mathcal{T}_{i, 1}$ we can assume that $w\in \{ w_R rsr_{\{r, t\}}, w_R r_{\{r, s\}}tst, w_R r_{\{r, s\}}trt, w_R srr_{\{s,t\}} \} \cup \{ w_{R'}tr r_{\{s, t\}}, w_{R'} r_{\{r, t\}}srs, w_{R'} r_{\{r, t\}}sts, w_{R'} rtr_{\{r, s\}} \}$. Again, the claim follows from the definition of $G_P$ together with Proposition \ref{treeproducts}. \qedhere
	\end{enumerate}
\end{proof}

\begin{lemma}\label{Lemma: injective homomorphism V to GP}
	Let $i \in \NN$ and $w' = w_T r_{\{u, v\}} \in D_{i+1} \backslash D_i$. Then there exists a unique $P \in \mathcal{T}_i$ with $w_T u, w_T v \in C'(P)$ and the canonical homomorphism $V_{w'} \to G_P$ is injective.
\end{lemma}
\begin{proof}
	As $w' \in D_{i+1} \backslash D_i$, we have $\{w_T u, w_T v\} \subseteq C_{i+1}$ and $\{w_T u, w_T v\} \not\subseteq C_i$. Without loss of generality we assume $w_T u \notin C_i$. Using Lemma \ref{Lemma: unique P}, we obtain a unique $P \in \mathcal{T}_i$ with $w_T u \in C(P) \backslash C_i$. Let $\beta \in \Phi_+$ be the root with $\{ w_T, w_T v \} \in \partial \beta$. Assume that there exists $i<j \in \NN$ and $Z \in \mathcal{T}_j$ with $\beta \in \hat{\Phi}_Z$. Then Lemma~\ref{Lemma: head residues non-simple roots} implies $C_{i+1} \subseteq C_j \subseteq \beta$. As $w_T v \in C_{i+1}$ and $w_T v \notin \beta$, this yields a contradiction. Thus we have $\beta \notin \hat{\Phi}_Z$ for any $Z \in \mathcal{T}_j$ with $i < j \in \NN$. We distinguish the following cases:
	\begin{enumerate}[label=$P \in \mathcal{T}_{i, \arabic*}$:]
		\item Suppose that $P$ is of type $\{s, t\}$. We first consider the following cases:
		\[ w_T u \in \{ w_P strs, w_P stsrs, w_P stsr, w_P stsrt, w_P r_{\{s, t\}} rt \}.  \]
		Note that we have $\{ w_T u, w_T v \} \subseteq C'(P)$ in all cases. Then it follows from Proposition \ref{treeproducts} that either $V_{w'}$ is a vertex group of $G_P$ or else $U_{w'}$ is a vertex group of $G_P$ which contains $V_{w'}$ as a subgroup. Now we consider the following remaining cases:
		\[ w_T u \in \{ w_P st rsr, w_P st r_{\{r, s\}}, w_P stsrtr, w_P r_{\{s, t\}} rtr, w_P r_{\{s, t\}}r \}. \]
		The symmetric case (interchanging $s$ and $t$) follows similarly. If $w_T u = w_P str_{\{r, s\}}$, then $\beta \in \hat{\Phi}_Z$ for $Z = R_{\{r, s\}}(w_P st)$. If $w_T u = w_P r_{\{s, t\}} rtr$, then $\beta \in \hat{\Phi}_Z$ for $Z = R_{\{r, t\}}(w_P sts)$. If $w_T u = w_P stsrtr$, then $\beta \in \hat{\Phi}_Z$ for $Z = R_{\{r, s\}}(w_P st)$. If $w_T u = w_P strsr$, then $\beta \in \hat{\Phi}_Z$, where $Z = R_{\{r, t\}}(w_P s)$. Note that $w_T \neq w_P r_{\{s, t\}}$.
		
		\item Suppose $P = \{R, R'\}$, where $R$ is of type $\{r, s\}$ and $R'$ is of type $\{r, t\}$. Using exactly the same arguments, the claim follows as in the case $P \in \mathcal{T}_{i, 1}$. \qedhere
	\end{enumerate}
\end{proof}

\begin{proposition}\label{Giinjective}
	Assume that $G_i$ is natural for some $i \in \NN$. Then $G_{i+1} \cong \star_{G_i} B_P$, where $P$ runs over $\mathcal{T}_i$. In particular, the mappings $G_i \to G_{i+1}$ and $B_P \to G_{i+1}$ are injective for each $P \in \mathcal{T}_i$.
\end{proposition}
\begin{proof}
	Recall from Definition \ref{Definition: BP} that $B_P = G_i \star_{H_P} G_P$ for each $P \in \mathcal{T}_i$ and note that $G_i, G_P$ are subgroups of $B_P$ by Proposition \ref{treeproducts}. The second part follows from Proposition \ref{treeproducts} and the first part. We let $x_{\alpha}$ be the generators of $G_i$, where $C_i \not\subseteq \alpha \in \Phi_+$, and we let $x_{\alpha, P}$ be the generators of $G_P$, where $C''(P) := \{ w\in W \mid U_w \text{ vertex group of } G_P \} \not\subseteq \alpha \in \Phi_+$. We define $H_i := \star_{G_i} B_P$, where $P$ runs over $\mathcal{T}_i$. Since we have canonical homomorphisms $G_i, G_P \to G_{i+1}$ extending $x_{\alpha} \mapsto x_{\alpha}$ and $x_{\alpha, P} \to x_{\alpha}$ (cf.\ Lemma \ref{Lemma: C'(P) contained in C_i+1}) which agree on $H_P$ (cf.\ Remark \ref{Remark: G_i is generated by x_alpha}), we obtain a unique homomorphism $B_P \to G_{i+1}$. Moreover, we obtain a (surjective) homomorphism $H_i \to G_{i+1}$. Now we will construct a homomorphism $G_{i+1} \to H_i$. Before we do that, we consider the generators of $H_i$.
	
	Let $\alpha\in \Phi_+$ and suppose $P \in \mathcal{T}_i$ with $C''(P) \not\subseteq \alpha$ and $C_i \not\subseteq \alpha$. Then $x_{\alpha}$ is a generator of $G_i$ and $x_{\alpha, P}$ is a generator of $G_P$. Lemma \ref{Lemma: head residues non-simple roots} implies that $\alpha \notin \hat{\Phi}_P$ and by definition of $H_P$ we have $x_{\alpha} = x_{\alpha, P}$ in $B_P$. Thus $H_i$ is generated by the set $\{ x_{\alpha}, x_{\beta, P} \mid C_i \not\subseteq \alpha \in \Phi_+, P \in \mathcal{T}_i, \beta \in \hat{\Phi}_P \}$.
	
	\emph{Claim: If $P, Q \in \mathcal{T}_i$ and $\beta \in \hat{\Phi}_P \cap \hat{\Phi}_Q$, then $P = Q$.}
	
	Suppose $P \neq Q$ and $\alpha \in \hat{\Phi}_P$ and $\beta \in \hat{\Phi}_Q$. For $i = 0$ one can show $(-\beta) \subsetneq \alpha$. If $i>0$, then $(-\beta) \subseteq \alpha$ follows essentially from Lemma~\ref{Lemma: -epsilonP contained in epsilonQ}.
	
	We need to construct for each $w\in W$ a homomorphism $U_w \to H_i$. We start by defining a mapping from the generators $x_{\alpha, w}$ of $U_w$ to $H_i$. Let $\alpha \in \Phi_+$ be a root and let $w\in C_{i+1}$ with $w\notin \alpha$. If $C_i \not\subseteq \alpha$, we define $x_{\alpha, w} \mapsto x_{\alpha}$. If $C_i \subseteq \alpha$, then $w\notin C_i$ and there exists a unique $P \in \mathcal{T}_i$ with $w\in C(P)$ by Lemma \ref{Lemma: unique P}. We define $x_{\alpha, w} \mapsto x_{\alpha, P}$.
	
	If $w\in C_i$, then we have a canonical homomorphism $U_w \to G_i \to H_i$. Thus we assume $w\notin C_i$. As before, there exists a unique $P \in \mathcal{T}_i$ with $w\in C(P)$. We have already shown that for each $\alpha \in \Phi_+$ with $w\notin \alpha$ and $C_i \not\subseteq \alpha$, we have $x_{\alpha} = x_{\alpha, P}$ in $B_P$. Thus these mappings extend to homomorphisms $U_w \to G_P \to H_i$. Now suppose $w' = w_R r_{\{s, t\}} \in D_{i+1}$ for some $R$ of type $\{s, t\}$. We have to show that the homomorphisms $U_{w_Rs}, U_{w_Rt} \to H_i$ extend to a homomorphism $V_{w'} \to H_i$. If $w' \in D_i$, this holds by definition of $G_i$. If $w' \notin D_i$, then Lemma \ref{Lemma: injective homomorphism V to GP} implies that there exists a unique $P \in \mathcal{T}_i$ with $\{ w_R s, w_T t \} \subseteq C'(P)$ and $V_{w'} \to G_P$ is injective. In particular, $V_{w'} \to H_i$ is an injective homomorphism. Moreover, the following diagrams commute, where $R$ is a residue of type $\{s, t\}$:
	\begin{center}
		\begin{tikzcd}
			U_w \arrow[r] \arrow[rd] & U_{ws} \arrow[d] && U_{w_R s} \arrow[r] \arrow[rd] & V_{w_R r_{\{s, t\}}} \arrow[d]\\
			& H_i &&& H_i
		\end{tikzcd}
	\end{center}
	The universal property of direct limits yields a homomorphism $G_{i+1} \to H_i$. It is clear that the concatenations of the two homomorphisms $G_{i+1} \to H_i$ and $H_i \to G_{i+1}$ map each $x_{\alpha}$ to itself. Thus both concatenations are equal to the identities and both homomorphisms are isomorphisms.
\end{proof}

\section{Main result}\label{Section: Main result}

In this section we let $(W, S)$ be of type $(4, 4, 4)$ and $\mathcal{M} = \left( M_{\alpha, \beta}^{\Gamma} \right)_{(\Gamma, \alpha, \beta) \in \mathcal{I}}$ be a locally Weyl-invariant commutator blueprint of type $(4, 4, 4)$. Moreover, we let $S = \{r, s, t\}$. 

\begin{lemma}\label{Lemma: G0 natural}
	The group $G_0$ is natural.
\end{lemma}
\begin{proof}
	The group $G_0$ satisfies (N$1$) by \cite[Lemma~$4.21$]{BiRGDandTreeproducts}. Note that $\mathcal{T}_0 = \mathcal{T}_{0, 1}$. Thus $G_0$ satisfies (N$2$) by \cite[Theorem~$4.27$]{BiRGDandTreeproducts}. In particular, $G_0$ is natural.
\end{proof}

\begin{lemma}\label{Lemma: G_i star V_T O_T to G_i+1 injective}
	Suppose $i \in \NN$ is such that $G_i$ is natural. Let $R \in \mathcal{T}_{i+1, 1}$ be of type $\{s, t\}$. Let $T = R_{\{r, t\}}(w_R)$ and suppose that $\ell(w_R rt) = \ell(w_R)$. Let $Z = R_{\{r, s\}}(w_R t)$. Then $Z \in \mathcal{T}_{i+2, 1}$ and the canonical homomorphism $G_i \star_{V_Z} O_Z \to G_{i+1}$ is injective.
\end{lemma}
\begin{proof}
	Note that $Z \in \mathcal{T}_{i+2, 1}$ and, as $i \in \NN$, we have $\ell(w_R r) = \ell(w_R) -1$. We distinguish the following two cases:
	\begin{enumerate}[label=(\roman*)]
		\item $T \in \mathcal{T}_{i, 1}$: As $G_i$ is natural, we deduce from Proposition \ref{Giinjective} that $B_T \to G_{i+1}$ is injective. Using Proposition \ref{treeproducts}, Remark \ref{Remark: isomorphism preserves amalgamated product}, Lemma \ref{folding}, Lemma \ref{HRtoGRinjective} and Lemma \ref{Lemma: J_R,t cong H_R star V_T O_T} we infer
		\allowdisplaybreaks
		\begin{align*}
			B_T &= G_i \star_{H_T} G_T \\
			&\cong G_i \star_{H_T} J_{T, r} \star_{J_{T, r}} G_T \\
			&\cong (G_i \star_{H_T} J_{T, r}) \star_{J_{T, r}} G_T \\
			&\cong \left( G_i \star_{H_T} H_T \star_{V_Z} O_Z \right) \star_{J_{T, r}} G_T \\
			&\cong (G_i \star_{V_Z} O_Z) \star_{J_{T, r}} G_T.
		\end{align*}
		In particular, each of the mappings $G_i \star_{V_Z} O_Z \to B_T \to G_{i+1}$ is injective.
		
		\item $T \notin \mathcal{T}_{i, 1}$: Then there exists a unique $P_{T} \in \mathcal{T}_{i, 2}$ with $T \in P_{T}$. Suppose $P_{T} = \{ T, T'' \}$. As $G_i$ is natural, we deduce from Proposition \ref{Giinjective} that $B_{P_{T}} \to G_{i+1}$ is injective. Using Proposition \ref{treeproducts}, Remark \ref{Remark: isomorphism preserves amalgamated product}, Lemma \ref{folding}, Lemma \ref{HRR'toGRR'injective} and Lemma \ref{Lemma: J_R,R' cong H_R,R' star V_Z O_Z} we infer
		\allowdisplaybreaks
		\begin{align*}
			B_{P_T} &= G_i \star_{H_{\{T, T''\}}} G_{\{T, T''\}} \\
			&\cong G_i \star_{H_{\{T, T''\}}} J_{(T, T'')} \star_{J_{(T, T'')}} G_{\{T, T''\}} \\
			&\cong \left( G_i \star_{H_{\{T, T''\}}} J_{(T, T'')} \right) \star_{J_{(T, T'')}} G_{\{T, T''\}} \\
			&\cong \left( G_i \star_{H_{\{T, T''\}}} H_{\{T, T''\}} \star_{V_Z} O_Z \right) \star_{J_{(T, T'')}} G_{\{T, T''\}} \\
			&\cong \left( G_i \star_{V_Z} O_Z \right) \star_{J_{(T, T'')}} G_{\{T, T''\}}.
		\end{align*}
		In particular, each of the mappings $G_i \star_{V_Z} O_Z \to B_{P_{T}} \to G_{i+1}$ is injective. \qedhere
	\end{enumerate}
\end{proof}

\begin{lemma}[{\cite[Remark~$4.28$ and Corollary~$4.29$]{BiRGDandTreeproducts}}]\label{Lemma: H_R to G_0 star OT injective}
	Define $R = R_{\{s, t\}}(r)$, $Z = R_{\{r, t\}}(rs)$ and $Z' = R_{\{r, s\}}(rt)$. Then $V_Z, V_{Z'} \to G_0$ are injective. Moreover, the canonical homomorphism $H_R \to \left( G_0 \star_{V_Z} O_Z \right) \star_{G_0} \left( G_0 \star_{V_{Z'}} O_{Z'} \right)$ is injective.
\end{lemma}

\begin{theorem}\label{Theorem: G1 satisfies N2}
	The group $G_1$ satisfies (N$2$).
\end{theorem}
\begin{proof}
	Note that $\mathcal{T}_{1, 2} = \emptyset$ and hence $\mathcal{T}_1 = \mathcal{T}_{1, 1}$. Thus we have to show that $H_R \to G_1$ is injective for each $R \in \mathcal{T}_{1, 1}$. Let $R \in \mathcal{T}_{1, 1}$ be of type $\{s, t\}$, i.e.\ $R = R_{\{s, t\}}(r)$. We abbreviate $Z = R_{\{r, t\}}(rs)$ and $T = R_{\{r, s\}}(1_W)$. Since $G_0$ is natural by Lemma~\ref{Lemma: G0 natural}, it follows from the proof of Lemma~\ref{Lemma: G_i star V_T O_T to G_i+1 injective} that the canonical homomorphism $G_0 \star_{V_Z} O_Z \to B_T$ is injective. Let $Z' = R_{\{r, s\}}(rt)$ and $T' = R_{\{r, t\}}(1_W)$. Again, Lemma~\ref{Lemma: G_i star V_T O_T to G_i+1 injective} implies that the homomorphism $G_0 \star_{V_{Z'}} O_{Z'} \to B_{T'}$ is injective. Now Proposition \ref{treeofgroupsinjective} together with Lemma \ref{Lemma: H_R to G_0 star OT injective} yields that 
	\[ H_R \to \left( G_0 \star_{V_Z} O_Z \right) \star_{G_0} \left( G_0 \star_{V_{Z'}} O_{Z'} \right) \to B_T \star_{G_0} B_{T'} \]
	is injective. As $G_0$ is natural by Lemma~\ref{Lemma: G0 natural}, it follows from Proposition \ref{treeproducts} and Proposition \ref{Giinjective} that $B_T \star_{G_0} B_{T'} \to G_1$ is injective. This finishes the proof.
\end{proof}

\begin{lemma}\label{Lemma: E_Rs to G_i injective}
	Suppose $2 \leq i \in \NN$ such that $G_{i-2}$ and $G_{i-1}$ are natural. Then for each $R\in \mathcal{T}_{i, 1}$ of type $\{s, t\}$ with $\ell(w_R rs) = \ell(w_R) -2$ the canonical homomorphism $E_{R, s} \to G_i$ is injective.
\end{lemma}
\begin{proof}
	Let $R \in \mathcal{T}_{i, 1}$ be of type $\{s, t\}$ with $\ell(w_R rs) = \ell(w_R) -2$, let $T = R_{\{r, t\}}(w_R)$ and $T' = R_{\{r, s\}}(w_R)$. Suppose $\ell(w_R rt) = \ell(w_R) -2$. Using Lemma \ref{Lemma: ERs to GTT' injective}, we have $\{T, T'\} \in \mathcal{T}_{i-2, 2}$ and $E_{R, s} \to G_{\{T, T'\}}$ is injective. As $G_{i-2}$ is natural by assumption, the homomorphism $G_{\{T, T'\}} \to G_{i-2} \star_{H_{\{T, T'\}}} G_{\{T, T'\}} = B_{\{T, T'\}}$ is injective by Proposition \ref{treeproducts}. Moreover, as $G_{i-2}$ and $G_{i-1}$ are natural, the homomorphisms $B_{\{T, T'\}} \to G_{i-1}$ and $G_{i-1} \to G_i$ are injective by Proposition \ref{Giinjective}. This finishes the claim. 
	
	Thus we can assume $\ell(w_R rt) = \ell(w_R)$. We abbreviate $Z := R_{\{r, s\}}(w_R t)$. By Lemma~\ref{Lemma: G_i star V_T O_T to G_i+1 injective} the canonical mapping $G_{i-1} \star_{V_Z} O_Z \to G_i$ is injective. We will show now that $X_R \to G_{i-1}$ is injective. We distinguish the following two cases:
	\begin{enumerate}[label=(\roman*)]
		\item $T' \in \mathcal{T}_{i-2, 1}$: As $G_{i-2}$ is natural by assumption, the mapping $G_{T'} \to B_{T'} \to G_{i-1}$ is injective by Proposition \ref{Giinjective}. Now Lemma \ref{Lemma: X_R 1} implies that the homomorphism $X_R \to G_{T'}$ is injective.
		
		\item $T' \notin \mathcal{T}_{i-2, 1}$: Then there exists a unique $P_{T'} \in \mathcal{T}_{i-2, 2}$ with $T' \in P_{T'}$. As $G_{i-2}$ is natural by assumption, the mapping $G_{P_{T'}} \to B_{P_{T'}} \to G_{i-1}$ is injective by Proposition \ref{Giinjective}. Now Lemma \ref{Lemma: X_R 2} implies that the homomorphism $X_R \to G_{P_{T'}}$ is injective.
	\end{enumerate}
	We conclude that $X_R \to G_{i-1}$ is injective. Moreover, $V_Z \to X_R$ is injective by Lemma \ref{Lemma: X_R a} and hence $X_R \star_{V_Z} O_Z \to G_{i-1} \star_{V_Z} O_Z \to G_i$ is injective by Proposition \ref{treeofgroupsinjective}. Using Lemma \ref{Lemma: X_R a} again, we infer that $E_{R, s} \to X_R \star_{V_Z} O_Z$ and, in particular, $E_{R, s} \to G_i$ is injective.
\end{proof}

\begin{theorem}\label{Ginatural}
	For each $i \geq 0$ the group $G_i$ is natural.
\end{theorem}
\begin{proof}
	We show the claim by induction on $i\geq 0$. If $i = 0$, claim follows from Lemma~\ref{Lemma: G0 natural}. Thus we can assume that $i\geq 1$ and that $G_k$ is natural for all $0 \leq k <i$. We have to show that $G_i$ satisfies (N$1$) and (N$2$).
	\begin{enumerate}[label=(N$\arabic*$)]
		\item Let $w\in C_i$. If $w\in C_{i-1}$, then each of the homomorphisms $U_w \to G_{i-1} \to G_i$ is injective by induction and Proposition \ref{Giinjective}. If $w\notin C_{i-1}$, then there exists $P \in \mathcal{T}_{i-1}$ with $w\in C(P)$ by definition of $C_i$. Using Lemma \ref{Lemma: Uw to GP} and Proposition \ref{Giinjective}, each of the homomorphisms $U_w \to G_P \to G_i$ is injective. Now we consider $w' \in D_i$. If $w' \in D_{i-1}$, induction and Proposition \ref{Giinjective} imply that each of the homomorphisms $V_{w'} \to G_{i-1} \to G_i$ is injective. Thus we can assume that $w' \notin D_{i-1}$. Let $w' = w_R r_{\{s, t\}}$ for some residue $R$ of type $\{s, t\}$ with $w_R s, w_R t \in C_i$. By Lemma~\ref{Lemma: injective homomorphism V to GP} there exists a unique $P \in \mathcal{T}_{i-1}$ such that $w_R s, w_R t \in C'(P)$ and each of the homomorphisms $V_{w'} \to G_P \to G_i$ is injective by induction. Thus (N$1$) is satisfied. In particular, $G_1$ is natural by Theorem~\ref{Theorem: G1 satisfies N2} and we can assume $i\geq 2$.
		
		\item We have to show that $H_P \to G_i$ is injective for each $P \in \mathcal{T}_i$. Suppose $P \in \mathcal{T}_{i, 1}$ is of type $\{s, t\}$. As $i \geq 2$, we can assume that $\ell(w_P rs) = \ell(w_P) -2$. Since $H_P \to E_{P, s}$ is injective by Lemma \ref{ERstoURsinjective} and $E_{P, s} \to G_i$ is injective by Lemma \ref{Lemma: E_Rs to G_i injective}, the claim follows. Now suppose that $P \in \mathcal{T}_{i, 2}$. Let $P = \{R, R'\}$, where $R$ is of type $\{r, s\}$ and $R'$ is of type $\{r, t\}$. Let $T = R_{\{r, t\}}(w_R)$ and let $T' = R_{\{r, s\}}(w_{R'})$. Note that in this case we have $i\geq 3$. By Lemma \ref{CleftrightcapHTHT'} we have $T, T' \in \mathcal{T}_{i-1, 1}$. As $G_{i-1}$ is natural, Proposition \ref{Giinjective} and Proposition \ref{treeproducts} imply that the mapping $B_T \star_{G_{i-1}} B_{T'} \to G_i$ is injective. By Lemma \ref{RR'CCleftCright} we have $H_{\{R, R'\}} \cong C_{(R, R')} \star_C C_{(R', R)}$. Thus it suffices to show that $C_{(R, R')} \star_C C_{(R', R)} \to B_T \star_{G_{i-1}} B_{T'}$ is injective and we will prove it by using Proposition \ref{treeofgroupsinjective}.
		
		Using Lemma \ref{Lemma: E_Rs to G_i injective}, the mappings $E_{T, t}, E_{T', s} \to G_{i-1}$ are injective. Then Lemma \ref{ERstoURsinjective}, Proposition \ref{treeproducts}, Remark \ref{Remark: isomorphism preserves amalgamated product} and Lemma \ref{folding} yield
		\allowdisplaybreaks
		\begin{align*}
			&B_T = G_{i-1} \star_{H_T} G_T \cong G_{i-1} \star_{E_{T, t}} E_{T, t} \star_{H_T} G_T \cong G_{i-1} \star_{E_{T, t}} U_{T, t}, \\
			&B_{T'} = G_{i-1} \star_{H_{T'}} G_{T'} \cong G_{i-1} \star_{E_{T', s}} E_{T', s} \star_{H_{T'}} G_{T'} \cong G_{i-1} \star_{E_{T', s}} U_{T', s}.
		\end{align*}
		Lemma \ref{CleftrightcapHTHT'} shows that $C_{(R, R')} \to U_{T, t}$, $C_{(R', R)} \to U_{T', s}$ are injective and, in particular, $C_{(R, R')} \to B_T, C_{(R', R)} \to B_{T'}$ are injective. Moreover, Lemma \ref{CleftrightcapHTHT'} implies that $C_{(R, R')} \cap E_{T, t} = C$ holds in $U_{T, t}$ and $C_{(R', R)} \cap E_{T', s} = C$ holds in $U_{T', s}$. Corollary \ref{AcapBisC} now yields:
		\allowdisplaybreaks
		\begin{align*}
			C_{(R, R')} \cap G_{i-1} &= C_{(R, R')} \cap G_{i-1} \cap E_{T, t} = C_{(R, R')} \cap E_{T, t} = C &&\text{in } B_T, \\
			C_{(R', R)} \cap G_{i-1} &= C_{(R', R)} \cap G_{i-1} \cap E_{T', s} = C_{(R', R)} \cap E_{T', s} = C &&\text{in } B_{T'}.
		\end{align*}
		Proposition \ref{treeofgroupsinjective} implies that the canonical homomorphism $C_{(R, R')} \star_C C_{(R', R)} \to B_T \star_{G_{i-1}} B_{T'}$ is injective. This finishes the proof. \qedhere
	\end{enumerate}
\end{proof}

In some sense, the next corollary is the main result of this paper. Hence we will recall the hypothesis in its statement.

\begin{corollary}\label{Corollary: non-trivial}
	Let $\mathcal{M}$ be a locally Weyl-invariant commutator blueprint of type $(4, 4, 4)$. Then $\mathcal{M}$ is faithful.
\end{corollary}
\begin{proof}
	By Lemma \ref{Lemma: U_+ isomorphic to G} we have $G \cong U_+$. We have to show that for each $w\in W$ the canonical homomorphism $U_w \to G \cong U_+$ is injective. Note that the following diagram commutes for each $i\in \NN$ with $w\in C_i$ (cf.\ Remark \ref{Remark: G_i is generated by x_alpha} and Remark \ref{Remark: G is generated by x_alpha}):
	\begin{center}
		\begin{tikzcd}
			U_w \arrow[r] \arrow[rd] & G_i \arrow[d] \\
			&G
		\end{tikzcd}
	\end{center}
	By Theorem \ref{Ginatural} the group $G_i$ is natural for each $i \geq 0$. Proposition \ref{Giinjective} implies that the canonical homomorphisms $G_i \to G_{i+1}$ are injective for all $i \in \NN$. It follows from \cite[$1.4.9$(iii)]{Rob96} that the canonical homomorphisms $G_i \to G$ are injective. Note that for each $w\in W$ there exists $i\in \NN$ with $w\in C_i$. As $G_i$ is natural, we have $U_w \to G_i$ is injective and, hence, $U_w \to G_i \to G$ is injective as well.
\end{proof}

\section{Consequences of Theorem~\ref{Main result: Theorem integrable Weyl-invariant}}\label{Section: Applications}

\subsection*{Examples of RGD-systems}

In this subsection we use the notation from \cite{BiConstruction}. Let $K \subseteq \NN_{\geq 3}$ be non-empty, let $\mathcal{J} = \left(J_k \right)_{k\in K}$ be a family of non-empty subsets $J_k \subseteq S$ and let $\mathcal{L} = \left(L_k^j\right)_{k \in K, j \in J_k}$ be a family of subsets $L_k^j \subseteq \{ 2, \ldots, k-1 \}$. By \cite[Lemma~4.24]{BiConstruction} the commutator blueprint $\mathcal{M}(K, \mathcal{J}, \mathcal{L})$ of type $(4, 4, 4)$ is Weyl-invariant. For the precise definition see \cite[Definition~4.16 and 4.19]{BiConstruction}.

\begin{theorem}\label{Maintheorem}
	The Weyl-invariant commutator blueprint $\mathcal{M}(K, \mathcal{J}, \mathcal{L})$ is integrable.
\end{theorem}
\begin{proof}
	This is a consequence of Theorem~\ref{Main result: Theorem integrable Weyl-invariant}.
\end{proof}

\begin{corollary}\label{Corollary: D_n}
	For each $n\in \NN$ there exists an RGD-system $\mathcal{D}_n = \left(G_n, \left(U_{\alpha}^{(n)}\right)_{\alpha \in \Phi}\right)$ of type $(4, 4, 4)$ over $\FF_2$ with the following properties:
	\begin{enumerate}[label=(\roman*)]
		\item Let $w\in W$ with $\ell(w) \leq n$ and let $\alpha, \beta \in \Phi_+$ with $w\in (-\alpha) \cap (-\beta)$. If $\alpha \subseteq \beta$, then $\left[ U_{\alpha}^{(n)}, U_{\beta}^{(n)} \right] = 1$.
		
		\item There exist $\alpha, \beta \in \Phi_+$ with $\alpha \subsetneq \beta$ and $\left[ U_{\alpha}^{(n)}, U_{\beta}^{(n)} \right] \neq 1$.
	\end{enumerate}
\end{corollary}
\begin{proof}
	Note that it suffices to show the claim for $n \in \NN_{\geq 3}$. We fix $n \in \NN_{\geq 3}$. Define $K := \{n\}$, $J_n := \{r\}$ for some $r\in S$ and assume $L_n^r \neq \emptyset$. Then $\mathcal{M}(K, \mathcal{J}, \mathcal{L})$ is an integrable commutator blueprint by Theorem \ref{Maintheorem}. Let $\mathcal{D} = \left(G, \left(U_{\alpha}\right)_{\alpha \in \Phi}\right)$ be its associated RGD-system. We claim that $\mathcal{D}$ is as required. As $L_n^r \neq \emptyset$, it suffices to show that $(i)$ holds. Let $w\in W$ and let $\alpha, \beta \in \Phi_+$ be such that $w \in (-\alpha) \cap (-\beta)$, $\alpha \subseteq \beta$ and $\left[U_{\alpha}, U_{\beta} \right] \neq 1$. We will show $\ell(w) >n$. By definition of $\mathcal{M}(K, \mathcal{J}, \mathcal{L})$ there exists a minimal gallery $\Gamma = (c_0, \ldots, c_k)$ of type $(n, r)$ between $\alpha$ and $\beta$. Using \cite[Lemma~$4.17(a)$]{BiConstruction} we can extend $(c_6, \ldots, c_k)$ to a gallery $(c_0', \ldots, c_{k'}') \in \mathrm{Min}$. In particular, as $k \geq 5n +1$ by definition, we have $k' \geq k-6 \geq 5n-5$.
	
	Let $(e_0, \ldots, e_m) \in \mathrm{Min}(w)$ be a minimal gallery. As $e_0 = 1_W \in \beta$ and $e_m = w \in (-\beta)$, there exists $0 \leq j \leq m-1$ with $\{ e_j, e_{j+1} \} \in \partial \beta$. Define $R := R_{\beta, \{ e_j, e_{j+1} \}}$. As $\alpha \subsetneq \beta$, $\beta$ is a non-simple root and Lemma \ref{projgal} yields the existence of a minimal gallery $(d_0 = e_0, \ldots, d_q = e_{j+1})$ with $d_i = \proj_R 1_W$ for some $0 \leq i \leq q-1$. As $\{ c_{k-1}, c_k \} \subseteq R$, we deduce that $\ell(w) \geq i \geq k' -3 \geq 5n-8 >n$.
\end{proof}

\begin{theorem}\label{Theorem: not FPRS}
	Suppose $K = \NN_{\geq 3}$ and $L_n^j = \{2\}$ for all $n \in K$ and $j \in J_n$. Then the RGD-system $\mathcal{D} = (G, (U_{\alpha})_{\alpha \in \Phi})$ associated with the commutator blueprint $\mathcal{M}( K, \mathcal{J}, \mathcal{L})$ (cf.\ Theorem~\ref{Maintheorem}) does not satisfy condition (FPRS).
\end{theorem}
\begin{proof}
	In this proof we use the notation from \cite[Section~$2.1$]{CR09}. Let $\Gamma_n \in \mathrm{Min}$ be a minimal gallery of type $(r, r_{\{s, t\}}, r, \ldots, r_{\{s, t\}}, r)$, where $r_{\{s, t\}}$ appears $n$ times in the type. Let $\alpha_n := \alpha_{\Gamma_n}$, i.e.\ $\alpha_n$ is the last root which is crossed by $\Gamma_n$. We note that for $n \in K = \NN_{\geq 3}$ the root $\alpha_n$ is a non-simple root of the $\{ r, s \}$-residue $R$ containing $(rr_{\{s, t\}})^nr$. Using Lemma~\ref{Lemma: Lemma 2.22 BiConstruction} we have $\ell(1_W, -\alpha_n) = 5n+1$. In particular, $\lim_{n\to \infty} \ell(1_W, -\alpha_n) = \infty$.
	
	 Assume that $\mathcal{D}$ has Property (FPRS). Then there would exist $n_0\in \NN$ such that for every $n \geq n_0$ we have $r(U_{\alpha_n}) \geq 10$. In particular, $U_{\alpha_n}$ fixes $B(c_+, 10)$ pointwise. We deduce that $u_{\alpha_0}^{-1} u_{\alpha_n} u_{\alpha_0}$ and also $[u_{\alpha_0}, u_{\alpha_n}]$ fix $B(c_+, 10)$ pointwise. But $[u_{\alpha_0}, u_{\alpha_n}] = u_{\omega_2} u_{\omega_2'}$, which does not fix $B(c_+, 10)$. Thus $\mathcal{D}$ does not have Property (FPRS).
\end{proof}

\subsection*{Extension theorem for twin buildings}

\begin{theorem}\label{Theorem: extension theorem}
	The extension theorem does not hold for thick $2$-spherical twin buildings.
\end{theorem}
\begin{proof}
	Let $\mathcal{M}, \mathcal{M}'$ be two different integrable commutator blueprints as constructed in Theorem \ref{Maintheorem} and let $\mathcal{D} = (G, (U_{\alpha})_{\alpha \in \Phi})$, $\mathcal{D}' = (G', (U_{\alpha}')_{\alpha \in \Phi})$ be their associated RGD-systems. We let $\Delta = \Delta(\mathcal{D})$ and $\Delta' = \Delta(\mathcal{D}')$ be the corresponding twin buildings and let $(c_+, c_-)$ (resp.\ $(c_+', c_-')$) be the distinguished pair of opposite chambers of $\Delta$ (resp.\ $\Delta'$). Note that every residue $R$ of $\Delta$ or of $\Delta'$ of rank $2$ is isomorphic to the generalized quadrangle of order $(2, 2)$, i.e.\ to the building which is associated with the group $C_2(2)$. For each $s\in S$ we fix an order on $R_{\{s\}}(c_+) = \{ c_0 := c_+, c_1, c_2 \}$ and on $R_{\{s\}}(c_+') = \{ c_0' := c_+', c_1', c_2' \}$. Note that the mapping $\phi_s: R_{\{s\}}(c_+) \to R_{\{s\}}(c_+'), c_i \mapsto c_i'$ is a bijection and hence an isometry.
	
	\emph{Claim: Let $s\neq t \in S$ and $J := \{s, t\}$. There exists an isometry $\phi_J: R_J(c_+) \to R_J(c_+')$ with $\phi_J \vert_{R_{\{s\}}(c_+)} = \phi_s$ and $\phi_J \vert_{R_{\{t\}}(c_+)} = \phi_t$.}
	
	Using the fact that the automorphism group of the generalized quadrangle of order $(2, 2)$ acts transitive on the chambers, we obtain an isometry $R_J(c_+) \to R_J(c_+')$ mapping $c_+$ onto $c_+'$. Using the \emph{root automorphisms} (if necessary), we obtain an isometry $\phi_J: R_J(c_+) \to R_J(c_+')$ with $\phi_J \vert_{R_{\{s\}}(c_+)} = \phi_s$ and $\phi_J \vert_{R_{\{s\}}(c_+)} = \phi_s$. Note that the root automorphisms which act non-trivially on $R_{\{s\}}(c_+')$ fix $R_{\{t\}}(c_+')$ pointwise.
	
	Denote by $E_2(c)$ the union of all rank $2$ residues containing $c$. Using the claim we obtain a bijection $\phi_2: E_2(c_+) \to E_2(c_+')$ such that for all $s\neq t \in S$ we have $\phi_2 \vert_{R_{\{s, t\}}(c_+)} = \phi_{\{s, t\}}$. It is easy to see that $\phi_2$ is an isometry (e.g.\ \cite[Proposition~$4.2.4$]{WeDiss21}). It is well-known that one can fine $d \in c_+^{\mathrm{op}}, d' \in (c_+')^{\mathrm{op}}$ such that $\phi_2$ extends to an isometry $\phi: E_2(c_+) \cup \{d\} \to E_2(c_+') \cup \{d'\}$ (for a proof see \cite[Proposition~$7.1.6$]{WeDiss21}).
	
	Assume that the extension theorem holds for thick $2$-spherical twin buildings. Then we can extend $\phi$ to an isometry $\Psi: \Delta \to \Delta'$. Let $\Sigma = \Sigma(c_+, c_-)$ (resp.\ $\Sigma' = \Sigma(c_+', c_-')$) be the twin apartment in $\Delta$ (resp.\ $\Delta'$). Let $g\in G$ be such that $g(\Sigma) = \Sigma(c_+, d)$ and $g(c_+) = c_+$ and let $g' \in G'$ be such that $g'(\Sigma') = \Sigma(c_+', d')$ and $g(c_+') = c_+'$. Then $\Psi' := (g')^{-1} \circ \Phi \circ g$ is an isometry from $\Delta$ to $\Delta'$ as well. Note that $\Psi'(\Sigma) = \Sigma'$ and $\Psi'(c_+) = c_+'$. Moreover, $\Psi_0: \Aut(\Delta) \to \Aut(\Delta'), f \mapsto \Psi' \circ f \circ (\Psi')^{-1}$ is an isomorphism which maps $U_{\alpha}$ onto $U_{\alpha}'$ for every $\alpha \in \Phi$. Let $(H, \alpha, \beta) \in \mathcal{I}$ (cf.\ Section~\ref{Section: commutator blueprint of type ((W, S), D)}) with $M(\mathcal{D})_{\alpha, \beta}^H \neq M(\mathcal{D}')_{\alpha, \beta}^H$. Then we have the following:
	\allowdisplaybreaks
	\begin{align*}
		\prod_{\gamma \in M(\mathcal{D})_{\alpha, \beta}^H} u_{\gamma}' = \Psi_0\left( \prod_{\gamma \in M(\mathcal{D})_{\alpha, \beta}^H} u_{\gamma} \right) = \Psi_0([u_{\alpha}, u_{\beta}]) = [u_{\alpha}', u_{\beta}'] = \prod_{\gamma \in M(\mathcal{D}')_{\alpha, \beta}^H} u_{\gamma}'.
	\end{align*}
	But this is a contradiction to \cite[Corollary $8.34(1)$]{AB08}. Thus  the extension theorem does not hold for these two twin buildings.
\end{proof}

\subsection*{Finiteness properties}

Let $\mathcal{D} = (G, (U_{\alpha})_{\alpha \in \Phi})$ be an RGD-system of irreducible $2$-spherical type $(W, S)$ and of rank at least $2$. The \textit{Steinberg group} associated with $\mathcal{D}$ is the group $\widehat{G}$ which is the direct limit of the inductive system formed by the groups $U_{\alpha}$ and $U_{[\alpha, \beta]} := \langle U_{\gamma} \mid \gamma \in [\alpha, \beta] \rangle$ for all prenilpotent pairs $\{ \alpha, \beta \} \subseteq \Phi$. For each $\alpha \in \Phi$ we denote the canonical image of $U_{\alpha}$ in $\widehat{G}$ by $\widehat{U}_{\alpha}$. It follows from \cite[Theorem $3.10$]{Ca06} that $\widehat{D} = (\widehat{G}, (\widehat{U}_{\alpha})_{\alpha \in \Phi})$ is an RGD-system of type $(W, S)$ and the kernel of the canonical homomorphism $\widehat{G} \to G$ is contained in the center of $\widehat{G}$.

Suppose that $\mathcal{D}$ is over $\FF_2$ and that $G$ is generated by its root groups. Then $\widehat{D}$ is over $\FF_2$ as well and $\widehat{G}$ is generated by its root groups. Now it follows from \cite[Corollary $8.79$ and remark thereafter]{AB08} that $\bigcap_{\alpha \in \Phi} N_{\widehat{G}}(\widehat{U}_{\alpha}) = \langle m(u)^{-1} m(v) \mid u, v \in \widehat{U}_{\alpha_s} \backslash \{1\}, s\in S \rangle$. As $\widehat{\mathcal{D}}$ is over $\FF_2$, we have $\widehat{U}_{\alpha_s} \backslash \{1\} = \{ u_s \}$. This implies $\bigcap_{\alpha \in \Phi} N_{\widehat{G}}(\widehat{U}_{\alpha}) = 1$. As $Z(\widehat{G}) \leq \bigcap_{\alpha \in \Phi} N_{\widehat{G}}(\widehat{U}_{\alpha})$, the kernel of $\widehat{G} \to G$ is trivial and, hence, $\widehat{G} \to G$ is an isomorphism. In particular, we obtain a presentation of $G$.

\begin{lemma}\label{Lemma: finite presentation with finite generating set}
	Let $G = \langle X \mid R \rangle$ be a finitely presented group with $\lvert X \rvert <\infty$. Then there exists a finite subset $F \subseteq R$ with $G = \langle X \mid F \rangle$.
\end{lemma}
\begin{proof}
	By \cite[Corollary $12$]{Ne37} there exists a finite set $E$ of relations with $G = \langle X \mid E \rangle$. Now for each $e\in E$ there exists a finite subset $F_e \subseteq R$ with $e \in \langle\langle F_e \rangle\rangle$. For $F := \bigcup_{e\in E} F_e \subseteq R$ we have $E \subseteq \langle\langle F \rangle\rangle$. We obtain the following epimorphisms:
	\[ \langle X \mid R \rangle \overset{\cong}{\to} \langle X \mid E \rangle \twoheadrightarrow \langle X \mid F \rangle \twoheadrightarrow \langle X \mid R \rangle. \]
	Since the concatenation maps each $x\in X$ to itself, all epimorphisms must be isomorphisms and the claim follows.
\end{proof}

\begin{theorem}\label{Theorem: Kac--Moody group of type (4,4,4) over F_2 not finitely presented}
	Kac--Moody groups of type $(4, 4, 4)$ over $\FF_2$ are not finitely presented.
\end{theorem}
\begin{proof}
	Let $\mathcal{D} = (\mathcal{G}, (U_{\alpha})_{\alpha \in \Phi})$ be the RGD-system associated with a split Kac--Moody group of type $(4, 4, 4)$ over $\FF_2$. By \cite[Example~2.8]{BiRGDandTreeproducts} we have $[U_{\alpha}, U_{\beta}] = 1$ for all $\alpha, \beta \in \Phi$ with $\alpha \subseteq \beta$. As the Steinberg group associated with $\mathcal{D}$ yields a presentation of $\mathcal{G}$, we deduce $\mathcal{G} = \langle X \mid R \rangle$, where $X = \{ u_{\alpha} \mid \alpha \in \Phi \}$ and $R = \{ u_{\alpha}^2 \mid \alpha \in \Phi \} \cup \{ [u_{\alpha}, u_{\beta}] v^{-1} \mid \{ \alpha, \beta \} \text{ prenilpotent pair}, v\in U_{(\alpha, \beta)} \}$. We apply Tietze-transformations to modify this presentation. We add $\tau_s$ to the set of generators and $\tau_s = u_{-\alpha_s} u_{\alpha_s} u_{-\alpha_s}$ to the set of relations. Since $\tau_s^2 = 1$ in $\mathcal{G}$, we add this relation to the set of relations. For $\alpha \in \Phi$ there exist $w\in W$ and $s\in S$ with $\alpha = w\alpha_s$. For $w\in W$ there exist $s_1, \ldots, s_k \in S$ with $w = s_1 \cdots s_k$. Note that $u_{\alpha} = u_{\alpha_s}^{\tau_k \cdots \tau_1}$ is a relation in $\mathcal{G}$, where $\tau_i = \tau_{s_i}$. Thus we can add these relations to the set of relations. Now the relations $u_{\alpha}^2$ are consequences of $u_{\alpha_s}^2$ for $\alpha \in \Phi \backslash \{ \alpha_s \mid s\in S \}$. Thus we can delete all relations $u_{\alpha}^2$ for $\alpha \in \Phi \backslash \{ \alpha_s \mid s \in S \}$. Moreover, we delete all commutation relations $[u_{\alpha}, u_{\beta}] = v$ with $\{ \alpha, \beta \} \not\subseteq \Phi_+$. This is possible, as the commutation relations are Weyl-invariant and for each prenilpotent pair $\{ \alpha, \beta \}$ there exists $w\in W$ with $\{ w\alpha, w\beta \} \subseteq \Phi_+$. As $u_{\alpha} = u_{\alpha_s}^{\tau_k \cdots \tau_1}$ is a relation, we replace in each relation every $u_{\alpha}$ by the corresponding element $u_{\alpha_s}^{\tau_k \cdots \tau_1}$. Now we delete all generators $u_{\alpha}$ with $\alpha \in \Phi \backslash \{ \alpha_s \mid s\in S \}$ and the corresponding relations $u_{\alpha} = u_{\alpha_s}^{\tau_k \cdots \tau_1}$. Note that the relation $\tau_s = u_{\alpha_s}^{\tau_s} u_{\alpha_s} u_{\alpha_s}^{\tau_s}$ is equivalent to $(u_{\alpha_s} \tau_s)^3 = 1$. Thus we have the following presentation, where $u_{\alpha}$ has to be understood as $u_{\alpha_s}^{\tau_k \cdots \tau_1}$:
	\[ \mathcal{G} = \left\langle \{ u_{\alpha_s}, \tau_s \mid s\in S \} \;\middle|\; \begin{cases*}
		\forall s\in S: u_{\alpha_s}^2 = \tau_s^2 = (u_{\alpha_s} \tau_s)^3 = 1 \\
		\forall \{ \alpha, \beta \} \subseteq \Phi_+ \text{ prenilpotent:} \\
		\qquad [u_{\alpha}, u_{\beta}] = v \text{ for suitable } v \in U_{(\alpha, \beta)}
	\end{cases*} \right\rangle.
	\]	
	Now we assume that $\mathcal{G}$ is finitely presented. By Lemma~\ref{Lemma: finite presentation with finite generating set} there exists a finite set $F$ of the set of relations such that $\mathcal{G} = \langle \{ u_{\alpha_s}, \tau_s \mid s\in S \} \mid F \rangle$. Let $k := \max \{ k_{\alpha} \mid u_{\alpha} \text{ appears in some } f \in F \}$. We consider the RGD-systems $\mathcal{D}_k = (G, (V_{\alpha})_{\alpha \in \Phi})$ obtained from Corollary \ref{Corollary: D_n}. Then $[V_{\alpha}, V_{\beta}] = 1$ for $\alpha, \beta \in \Phi_+$ with $\alpha \subseteq \beta$, if there exists $w\in W$ with $\ell(w) \leq k$ and $w \in (-\alpha) \cap (-\beta)$. Moreover, $[V_{\delta}, V_{\gamma}] \neq 1$ for some $\delta \subsetneq \gamma \in \Phi_+$. It is not hard to see that we obtain a homomorphism $\phi: \mathcal{G} \to \mathcal{D}_k$ from the finite presentation to $\mathcal{D}_k$ such that $u_{\alpha_s} \mapsto u_{\alpha_s}, \tau_s \mapsto \tau_s$ (recall that for $\alpha \subsetneq \beta$ we have $[U_{\alpha}, U_{\beta}] = 1$ in $\mathcal{G}$). The commutation relations of $\mathcal{G}$ and $\mathcal{D}_k$ imply $1 = \phi(1) = \phi([U_{\delta}, U_{\gamma}]) = [\phi(U_{\delta}), \phi(U_{\gamma})] = [V_{\delta}, V_{\gamma}] \neq 1$. This is a contradiction and the Kac--Moody group is not finitely presented.
\end{proof}

\begin{theorem}\label{Uplusnotfinitelygenerated}
	Let $\mathcal{D} = (G^{\mathcal{D}}, (U_{\alpha})_{\alpha \in \Phi})$ be an RGD-system of type $(4, 4, 4)$ over $\FF_2$. Then the group $U_+$ is not finitely generated.
\end{theorem}
\begin{proof}
	The group $U_+$ is isomorphic to the direct limit of its subgroups $(U_w)_{w\in W}$ by \cite[Theorem $8.85$]{AB08}. We have shown in Lemma \ref{Lemma: U_+ isomorphic to G} that $U_+$ is isomorphic to the direct limit $G$ of the inductive system formed by the groups $G_i$. Moreover, the homomorphisms $G_i \to G_{i+1}$ are injective by Theorem \ref{Ginatural} and Proposition \ref{Giinjective}. Thus the homomorphisms $G_i \to G$ are injective by \cite[$1.4.9(iii)$]{Rob96}. By construction, the canonical homomorphism $G_i \to G_{i+1}$ is not surjective and, hence, $G_i \to G$ are not surjective as well. Assume that $U_+$ is finitely generated, i.e.\ $U_+ = \langle g_1, \ldots, g_n \rangle$. Since $U_+ = \langle u_{\alpha} \mid \alpha \in \Phi_+ \rangle$, there exists $i\in \NN$ with $U_+ = \langle U_w \mid w\in C_i \rangle$.	This implies $G = \langle U_w \mid w \in C_i \rangle = G_i$, i.e.\ the canonical homomorphism $G_i \to G$ is surjective. This is a contradiction and hence $U_+$ is not finitely generated.
\end{proof}

\appendix
\section{Figures}

We illustrate here all groups defined in Section~\ref{Section: Locally Weyl-invariant Commutator blueprints}.

\begin{figure}[h]
	\begin{minipage}{0.4\linewidth}

		\caption{Illustration of the group $V_R$}
	\end{minipage}
	\begin{minipage}{0.4\linewidth}
		%
		\caption{Illustration of the group $O_R$}
	\end{minipage}
\end{figure}

\begin{figure}[h]
	\centering
	%
	\caption{Illustration of the group $H_R$}
\end{figure}

\begin{figure}[h]
	\centering
	%
	\caption{Illustration of the group $J_{R, t}$}
\end{figure}

\begin{figure}[h]
	\centering
	%
	\caption{Illustration of the group $G_R$}
\end{figure}

\begin{figure}[h]
	\centering
	%
	\caption{Illustration of the group $E_{R, s}$}
\end{figure}

\begin{figure}[h]
	\centering
	%
	\caption{Illustration of the group $U_{R, s}$}
\end{figure}

\begin{figure}[h]
	\centering
	%
	\caption{Illustration of the group $X_R$}
\end{figure}

\begin{figure}[h]
	\centering
	%
	
	\caption{Illustration of the group $H_{\{R, R'\}}$}
\end{figure}

\begin{figure}[h]
	\centering
	%
	\caption{Illustration of the group $J_{(R, R')}$}
\end{figure}

\begin{figure}[h]
	\centering
	%
	\caption{Illustration of the group $G_{\{R, R'\}}$}
\end{figure}

\begin{figure}[h]
	\centering
	%
	\caption{Illustration of the group $C$}
\end{figure}

\begin{figure}[h]
	\centering
	%
	\caption{Illustration of the group $C_{(R', R)}$}
\end{figure}

\clearpage

\bibliographystyle{amsalpha}
%\bibliography{references}
\bibliography{references}

@incollection {Ab04,
	AUTHOR = {Abramenko, P.},
	TITLE = {Finiteness properties of groups acting on twin buildings},
	BOOKTITLE = {Groups: topological, combinatorial and arithmetic aspects},
	SERIES = {London Math. Soc. Lecture Note Ser.},
	VOLUME = {311},
	PAGES = {21--26},
	PUBLISHER = {Cambridge Univ. Press, Cambridge},
	YEAR = {2004},
	MRCLASS = {20E42},
	MRNUMBER = {2073344},
	DOI = {10.1017/CBO9780511550706.003},
	URL = {https://doi.org/10.1017/CBO9780511550706.003},
}

@book {AB08,
	AUTHOR = {Abramenko, Peter and Brown, Kenneth S.},
	TITLE = {Buildings},
	SERIES = {Graduate Texts in Mathematics},
	VOLUME = {248},
	NOTE = {Theory and applications},
	PUBLISHER = {Springer, New York},
	YEAR = {2008},
	PAGES = {xxii+747},
	ISBN = {978-0-387-78834-0},
	MRCLASS = {20E42 (20F55 20J06 51E24 51F15)},
	MRNUMBER = {2439729},
	MRREVIEWER = {Ralf Koehl},
	DOI = {10.1007/978-0-387-78835-7},
	URL = {https://doi.org/10.1007/978-0-387-78835-7},
}

@misc{Abramenko_Gates:Finite_presentability_of_Kac-Moody_groups_over_finite_fields,
	title={Finite presentability of {K}ac--{M}oody groups over finite fields}, 
	author={Peter Abramenko and Zachary Gates},
	year={2019},
	eprint={1912.05611},
	archivePrefix={arXiv},
	primaryClass={math.GR},
	howpublished={https://arxiv.org/abs/1912.05611}, 
}

@article {AM97,
	AUTHOR = {Abramenko, Peter and M\"{u}hlherr, Bernhard},
	TITLE = {Pr\'{e}sentations de certaines {$BN$}-paires jumel\'{e}es comme sommes
	amalgam\'{e}es},
	JOURNAL = {C. R. Acad. Sci. Paris S\'{e}r. I Math.},
	FJOURNAL = {Comptes Rendus de l'Acad\'{e}mie des Sciences. S\'{e}rie I.
	Math\'{e}matique},
	VOLUME = {325},
	YEAR = {1997},
	NUMBER = {7},
	PAGES = {701--706},
	ISSN = {0764-4442},
	MRCLASS = {20E42},
	MRNUMBER = {1483702},
	DOI = {10.1016/S0764-4442(97)80044-4},
	URL = {https://doi.org/10.1016/S0764-4442(97)80044-4},
}

@article {As23,
	AUTHOR = {Ashcroft, Calum J.},
	TITLE = {Link conditions for cubulation},
	JOURNAL = {Int. Math. Res. Not. IMRN},
	FJOURNAL = {International Mathematics Research Notices. IMRN},
	YEAR = {2023},
	NUMBER = {12},
	PAGES = {9950--10012},
	ISSN = {1073-7928,1687-0247},
	MRCLASS = {20F65 (57K10)},
	MRNUMBER = {4601616},
	MRREVIEWER = {Harry\ Petyt},
	DOI = {10.1093/imrn/rnac111},
	URL = {https://doi.org/10.1093/imrn/rnac111},
}

@article {Bischof_On_Growth_Functions_of_Coxeter_Groups,
	AUTHOR = {Bischof, Sebastian},
	TITLE = {On {G}rowth {F}unctions of {C}oxeter {G}roups},
	JOURNAL = {Proc. Edinb. Math. Soc. (2)},
	FJOURNAL = {Proceedings of the Edinburgh Mathematical Society. Series II},
	VOLUME = {68},
	YEAR = {2025},
	NUMBER = {3},
	PAGES = {979--993},
	ISSN = {0013-0915,1464-3839},
	MRCLASS = {20F55 (20F69 51F15)},
	MRNUMBER = {4943977},
	DOI = {10.1017/S0013091525000094},
	URL = {https://doi.org/10.1017/S0013091525000094},
}

@article {BiConstruction,
	AUTHOR = {Bischof, Sebastian},
	TITLE = {Construction of commutator blueprints},
	JOURNAL = {J. Aust. Math. Soc.},
	FJOURNAL = {Journal of the Australian Mathematical Society},
	VOLUME = {120},
	YEAR = {2026},
	NUMBER = {1},
	PAGES = {1--33},
	ISSN = {1446-7887,1446-8107},
	MRCLASS = {20D15 (20E42 20F12)},
	MRNUMBER = {5026450},
	DOI = {10.1017/s1446788725101183},
	URL = {https://doi.org/10.1017/s1446788725101183},
}

@article {BiIsomorphism,
	AUTHOR = {Bischof, Sebastian},
	TITLE = {Isomorphisms of groups of {K}ac--{M}oody type over {$\Bbb
	F_{2}$}},
	JOURNAL = {J. Group Theory},
	FJOURNAL = {Journal of Group Theory},
	VOLUME = {29},
	YEAR = {2026},
	NUMBER = {3},
	PAGES = {601--631},
	ISSN = {1433-5883,1435-4446},
	MRCLASS = {20G44 (20E36 20E42 22E65 51E24)},
	MRNUMBER = {5066260},
	DOI = {10.1515/jgth-2025-0061},
	URL = {https://doi.org/10.1515/jgth-2025-0061},
}

@misc {BiRGD,
	title={{RGD}-systems over $\mathbb{F}_2$}, 
	author={Sebastian Bischof},
	howpublished={https://doi.org/10.4171/JCA/114},
	year={2025},
	note={to appear in {J}ournal of {C}ombinatorial {A}lgebra}
}

@misc {BiRGDandTreeproducts,
	title={{RGD}-systems of type $(4, 4, 4)$ over $\mathbb{F}_2$ and Tree Products},
	author={Sebastian Bischof},
	howpublished={https://arxiv.org/abs/2504.03577},
	year={2025},
}

@book {Bo02,
	AUTHOR = {Bourbaki, Nicolas},
	TITLE = {Lie groups and {L}ie algebras. {C}hapters 4--6},
	SERIES = {Elements of Mathematics (Berlin)},
	NOTE = {Translated from the 1968 French original by Andrew Pressley},
	PUBLISHER = {Springer-Verlag, Berlin},
	YEAR = {2002},
	PAGES = {xii+300},
	ISBN = {3-540-42650-7},
	MRCLASS = {17-01 (00A05 20E42 20F55 22-01)},
	MRNUMBER = {1890629},
	DOI = {10.1007/978-3-540-89394-3},
	URL = {https://doi.org/10.1007/978-3-540-89394-3},
}

@article {BCM21,
	AUTHOR = {Bischof, Sebastian and Chosson, Anton and M\"{u}hlherr, Bernhard},
	TITLE = {On isometries of twin buildings},
	JOURNAL = {Beitr. Algebra Geom.},
	FJOURNAL = {Beitr\"{a}ge zur Algebra und Geometrie. Contributions to Algebra
	and Geometry},
	VOLUME = {62},
	YEAR = {2021},
	NUMBER = {2},
	PAGES = {441--456},
	ISSN = {0138-4821},
	MRCLASS = {20E42 (51E24)},
	MRNUMBER = {4254628},
	DOI = {10.1007/s13366-020-00524-5},
	URL = {https://doi.org/10.1007/s13366-020-00524-5},
}

@article{BM23,
	url = {https://doi.org/10.1515/advgeom-2023-0013},
	title = {Isometries of wall-connected twin buildings},
	author = {Sebastian Bischof and Bernhard Mühlherr},
	pages = {371--388},
	volume = {23},
	number = {3},
	journal = {Advances in Geometry},
	doi = {doi:10.1515/advgeom-2023-0013},
	year = {2023},
	lastchecked = {2023-08-19}
}

@article {Ca06,
	AUTHOR = {Caprace, Pierre-Emmanuel},
	TITLE = {On 2-spherical {K}ac--{M}oody groups and their central
	extensions},
	JOURNAL = {Forum Math.},
	FJOURNAL = {Forum Mathematicum},
	VOLUME = {19},
	YEAR = {2007},
	NUMBER = {5},
	PAGES = {763--781},
	ISSN = {0933-7741,1435-5337},
	MRCLASS = {20E42 (20F05 20F55 20G15 22E65)},
	MRNUMBER = {2350773},
	MRREVIEWER = {Guy\ Rousseau},
	DOI = {10.1515/FORUM.2007.031},
	URL = {https://doi.org/10.1515/FORUM.2007.031},
}

@article {CM06,
	AUTHOR = {Caprace, Pierre-Emmanuel and M\"{u}hlherr, Bernhard},
	TITLE = {Isomorphisms of {K}ac--{M}oody groups which preserve bounded
	subgroups},
	JOURNAL = {Adv. Math.},
	FJOURNAL = {Advances in Mathematics},
	VOLUME = {206},
	YEAR = {2006},
	NUMBER = {1},
	PAGES = {250--278},
	ISSN = {0001-8708},
	MRCLASS = {20E42 (17B67 20E36 20G15 22E65 51E24)},
	MRNUMBER = {2261755},
	MRREVIEWER = {Guy Rousseau},
	DOI = {10.1016/j.aim.2005.08.008},
	URL = {https://doi.org/10.1016/j.aim.2005.08.008},
}

@article {CM11,
	AUTHOR = {Caprace, Pierre-Emmanuel and Monod, Nicolas},
	TITLE = {Decomposing locally compact groups into simple pieces},
	JOURNAL = {Math. Proc. Cambridge Philos. Soc.},
	FJOURNAL = {Mathematical Proceedings of the Cambridge Philosophical
	Society},
	VOLUME = {150},
	YEAR = {2011},
	NUMBER = {1},
	PAGES = {97--128},
	ISSN = {0305-0041,1469-8064},
	MRCLASS = {22D05 (20E32)},
	MRNUMBER = {2739075},
	MRREVIEWER = {George\ A.\ Willis},
	DOI = {10.1017/S0305004110000368},
	URL = {https://doi.org/10.1017/S0305004110000368},
}

@article {CR09,
	AUTHOR = {Caprace, Pierre-Emmanuel and R\'{e}my, Bertrand},
	TITLE = {Simplicity and superrigidity of twin building lattices},
	JOURNAL = {Invent. Math.},
	FJOURNAL = {Inventiones Mathematicae},
	VOLUME = {176},
	YEAR = {2009},
	NUMBER = {1},
	PAGES = {169--221},
	ISSN = {0020-9910},
	MRCLASS = {20G44 (20E42 20G15 51E24)},
	MRNUMBER = {2485882},
	MRREVIEWER = {Guy Rousseau},
	DOI = {10.1007/s00222-008-0162-6},
	URL = {https://doi.org/10.1007/s00222-008-0162-6},
}

@article {CR09b,
	AUTHOR = {Caprace, Pierre-Emmanuel and R\'{e}my, Bertrand},
	TITLE = {Groups with a root group datum},
	JOURNAL = {Innov. Incidence Geom.},
	FJOURNAL = {Innovations in Incidence Geometry. Algebraic, Topological and
	Combinatorial},
	VOLUME = {9},
	YEAR = {2009},
	PAGES = {5--77},
	ISSN = {2640-7337},
	MRCLASS = {20E42 (17B67 20F55 20G44 51E24)},
	MRNUMBER = {2658894},
	MRREVIEWER = {Herbert Abels},
	DOI = {10.2140/iig.2009.9.5},
	URL = {https://doi.org/10.2140/iig.2009.9.5},
}

@article {CRW17b,
	AUTHOR = {Caprace, Pierre-Emmanuel and Reid, Colin D. and Willis, George
	A.},
	TITLE = {Locally normal subgroups of totally disconnected groups.
	{P}art {II}: {C}ompactly generated simple groups},
	JOURNAL = {Forum Math. Sigma},
	FJOURNAL = {Forum of Mathematics. Sigma},
	VOLUME = {5},
	YEAR = {2017},
	PAGES = {Paper No. e12, 89},
	MRCLASS = {22D05 (20E15 20E32)},
	MRNUMBER = {3659769},
	MRREVIEWER = {R\"{o}gnvaldur G. M\"{o}ller},
	DOI = {10.1017/fms.2017.8},
	URL = {https://doi.org/10.1017/fms.2017.8},
}

@misc{GaDiss,
	title={Finite presentability of groups acting on locally finite twin buildings}, 
	author={Zachary Gates},
	year={2018},
	note={PhD thesis, University of Virginia},
	howpublished={https://doi.org/10.18130/V3M03XX09}, 
}

@article {KWM05,
	AUTHOR = {Kapovich, Ilya and Weidmann, Richard and Miasnikov, Alexei},
	TITLE = {Foldings, graphs of groups and the membership problem},
	JOURNAL = {Internat. J. Algebra Comput.},
	FJOURNAL = {International Journal of Algebra and Computation},
	VOLUME = {15},
	YEAR = {2005},
	NUMBER = {1},
	PAGES = {95--128},
	ISSN = {0218-1967},
	MRCLASS = {20F10 (20F65)},
	MRNUMBER = {2130178},
	MRREVIEWER = {Vitaly A. Roman\cprime kov},
	DOI = {10.1142/S021819670500213X},
	URL = {https://doi.org/10.1142/S021819670500213X},
}

@article {KS70,
	AUTHOR = {Karrass, A. and Solitar, D.},
	TITLE = {The subgroups of a free product of two groups with an
	amalgamated subgroup},
	JOURNAL = {Trans. Amer. Math. Soc.},
	FJOURNAL = {Transactions of the American Mathematical Society},
	VOLUME = {150},
	YEAR = {1970},
	PAGES = {227--255},
	ISSN = {0002-9947},
	MRCLASS = {20.52},
	MRNUMBER = {260879},
	MRREVIEWER = {Hanna Neumann},
	DOI = {10.2307/1995492},
	URL = {https://doi.org/10.2307/1995492},
}

@article {MR95,
	AUTHOR = {M\"{u}hlherr, Bernhard and Ronan, Mark},
	TITLE = {Local to global structure in twin buildings},
	JOURNAL = {Invent. Math.},
	FJOURNAL = {Inventiones Mathematicae},
	VOLUME = {122},
	YEAR = {1995},
	NUMBER = {1},
	PAGES = {71--81},
	ISSN = {0020-9910},
	MRCLASS = {20E42 (17B67 51E24)},
	MRNUMBER = {1354954},
	MRREVIEWER = {Herbert Abels},
	DOI = {10.1007/BF01231439},
	URL = {https://doi.org/10.1007/BF01231439},
}

@article{Ne37,
	author = {Neumann, B. H.},
	title = {Some Remarks on Infinite Groups},
	journal = {Journal of the London Mathematical Society},
	volume = {s1-12},
	number = {2},
	pages = {120-127},
	doi = {https://doi.org/10.1112/jlms/s1-12.46.120},
	url = {https://londmathsoc.onlinelibrary.wiley.com/doi/abs/10.1112/jlms/s1-12.46.120},
	eprint = {https://londmathsoc.onlinelibrary.wiley.com/doi/pdf/10.1112/jlms/s1-12.46.120},
	year = {1937}
}

@article {Re02,
	AUTHOR = {R{\'e}my, Bertrand},
	TITLE = {Immeubles de {K}ac--{M}oody hyperboliques, groupes non
	isomorphes de m\^eme immeuble},
	JOURNAL = {Geom. Dedicata},
	FJOURNAL = {Geometriae Dedicata},
	VOLUME = {90},
	YEAR = {2002},
	PAGES = {29--44},
	ISSN = {0046-5755,1572-9168},
	MRCLASS = {20E42 (20E08 22E65 51E24)},
	MRNUMBER = {1898149},
	MRREVIEWER = {Guy\ Rousseau},
	DOI = {10.1023/A:1014969911151},
	URL = {https://doi.org/10.1023/A:1014969911151},
}

@book {Rob96,
	AUTHOR = {Robinson, Derek J. S.},
	TITLE = {A course in the theory of groups},
	SERIES = {Graduate Texts in Mathematics},
	VOLUME = {80},
	EDITION = {Second},
	PUBLISHER = {Springer-Verlag, New York},
	YEAR = {1996},
	PAGES = {xviii+499},
	ISBN = {0-387-94461-3},
	MRCLASS = {20-01},
	MRNUMBER = {1357169},
	DOI = {10.1007/978-1-4419-8594-1},
	URL = {https://doi.org/10.1007/978-1-4419-8594-1},
}

@article {RR06,
	AUTHOR = {R\'{e}my, Bertrand and Ronan, Mark},
	TITLE = {Topological groups of {K}ac--{M}oody type, right-angled
	twinnings and their lattices},
	JOURNAL = {Comment. Math. Helv.},
	FJOURNAL = {Commentarii Mathematici Helvetici. A Journal of the Swiss
	Mathematical Society},
	VOLUME = {81},
	YEAR = {2006},
	NUMBER = {1},
	PAGES = {191--219},
	ISSN = {0010-2571,1420-8946},
	MRCLASS = {20E42 (17B67 20G25 22E40 51E24)},
	MRNUMBER = {2208804},
	MRREVIEWER = {Guy\ Rousseau},
	DOI = {10.4171/CMH/49},
	URL = {https://doi.org/10.4171/CMH/49},
}

@book {Se79,
	AUTHOR = {Serre, Jean-Pierre},
	TITLE = {Trees},
	SERIES = {Springer Monographs in Mathematics},
	NOTE = {Translated from the French original by John Stillwell,
	Corrected 2nd printing of the 1980 English translation},
	PUBLISHER = {Springer-Verlag, Berlin},
	YEAR = {2003},
	PAGES = {x+142},
	ISBN = {3-540-44237-5},
	MRCLASS = {20E08 (05C05 20E06 20G25)},
	MRNUMBER = {1954121},
}

@book {Ti74,
	AUTHOR = {Tits, Jacques},
	TITLE = {Buildings of spherical type and finite {BN}-pairs},
	SERIES = {Lecture Notes in Mathematics, Vol. 386},
	PUBLISHER = {Springer-Verlag, Berlin-New York},
	YEAR = {1974},
	PAGES = {x+299},
	MRCLASS = {20G15 (50A20)},
	MRNUMBER = {0470099},
	MRREVIEWER = {Bruce Cooperstein},
}

@article {Ti86,
	AUTHOR = {Tits, Jacques},
	TITLE = {Ensembles ordonn\'{e}s, immeubles et sommes amalgam\'{e}es},
	JOURNAL = {Bull. Soc. Math. Belg. S\'{e}r. A},
	FJOURNAL = {Bulletin de la Soci\'{e}t\'{e} Math\'{e}matique de Belgique. S\'{e}rie A},
	VOLUME = {38},
	YEAR = {1986},
	PAGES = {367--387 (1987)},
	ISSN = {0037-9476},
	MRCLASS = {20G15 (06A10 17B67 22E20 51B05)},
	MRNUMBER = {885545},
	MRREVIEWER = {Guy Rousseau},
}

@incollection {Ti92,
	AUTHOR = {Tits, Jacques},
	TITLE = {Twin buildings and groups of {K}ac--{M}oody type},
	BOOKTITLE = {Groups, combinatorics \& geometry ({D}urham, 1990)},
	SERIES = {edited by M. Liebeck and J. Saxl, London Math. Soc. Lecture Note Ser.},
	VOLUME = {165},
	PAGES = {249--286},
	PUBLISHER = {Cambridge Univ. Press, Cambridge},
	YEAR = {1992},
	MRCLASS = {20E42 (17B67 51E24)},
	MRNUMBER = {1200265},
	MRREVIEWER = {Guy Rousseau},
	DOI = {10.1017/CBO9780511629259.023},
	URL = {https://doi.org/10.1017/CBO9780511629259.023},
}

@misc {WeDiss21,
	AUTHOR = {Wendlandt, Katharina},
	TITLE = {Exceptional twin buildings of type $\tilde{C}_2$},
	howpublished={http://dx.doi.org/10.22029/jlupub-9836},
	YEAR = {2021},
	note={PhD thesis, Justus-Liebig-Universität Gie\ss en},
}

\end{document}